\documentclass[10pt]{article}%
\usepackage{amssymb,amsmath,amsfonts,amsthm,array,bm,bbm,color}%
\usepackage{graphicx}
\usepackage{hyperref}
\usepackage{mathrsfs}
\usepackage{mathabx}
\usepackage{dsfont}
\usepackage{accents}
\hypersetup{colorlinks=true, linkcolor=red, anchorcolor=blue, citecolor=blue}
\usepackage{tikz}

\providecommand{\U}[1]{\protect\rule{.1in}{.1in}}

\allowdisplaybreaks

\numberwithin{equation}{section}

\newtheorem{theorem}{Theorem}[section]
\newtheorem{lemma}[theorem]{Lemma}
\newtheorem{corollary}[theorem]{Corollary}
\newtheorem{proposition}[theorem]{Proposition}
\newtheorem{remark}[theorem]{Remark}

\newtheorem{definition}[theorem]{Definition}

\def\<{\langle}
\def\>{\rangle}
\def\d{\,{\rm d}}
\def\&{\,&}

\def\N{\mathbb{N}}

\def\M{\mathfrak{M}}
\def\ls{\lesssim}
\def\R{\mathbb{R}}
\def\SS{\mathbb{S}}

\def\Z{\mathbb{Z}}

\def\p{\partial}
\def\t{\tilde}

\def\eps{\varepsilon}
\def\vphi{\varphi}
\def\mP{\mathbf{P}}
\def\a{\mathfrak{a}}

\def\mF{\mathbf{F}}

\def\l{\mathfrak{l}}

\def\cB{{\mathcal B}}

\def\cD{{\mathcal D}}

\def\cF{{\mathfrak F}}

\def\cH{{\mathcal H}}

\def\cP{{\mathcal P}}
\def\cQ{{\mathcal Q}}

\def\cS{{\mathcal S}}

\def\cU{{\mathcal U}}

\def\gs{\gtrsim}
\def\ls{\lesssim}

\def\tP{{\tilde{\mathcal{P}} }}
\def\tF{{\tilde{\mathfrak F} }}
\def\F{{\mathfrak F }}
\def\U{{\mathfrak U }}
\def\fD{{\mathfrak D }}

\let\f=\frac
\def\mU{{\mathbf U }}
\def\mS{{\mathbf S }}
\def\mB{{\mathbf B }}
\def\mC{{\mathbf C }}
\def\mY{{\mathbf Y }}
\def\mX{{\mathbf X }}

\usepackage{relsize} 

\newcommand{\ccap}{\mathbin{\text{\smaller$\bigcap$}}}

\definecolor{darkgreen}{rgb}{0.0, 0.42, 0.0}
\newcommand{\red}{\textcolor{red}}

\newcommand{\ben}{\begin{eqnarray}}
\newcommand{\een}{\end{eqnarray}}
\newcommand{\beno}{\begin{eqnarray*}}
\newcommand{\eeno}{\end{eqnarray*}}

\begin{document}

\title{Uniqueness for the spatially homogeneous Boltzmann equation in critical Sobolev spaces}

\author{  Tianhao Dong\footnote{Email: dth23@mails.tsinghua.edu.cn. Department of Mathematical Sciences, Tsinghua University, Beijing 100084, China} \, Shuchen Guo\footnote{Email: guo@maths.ox.ac.uk. Mathematical Institute, University of Oxford, Oxford, OX2 6GG, UK.}  \,   Jie Ji\footnote{Email: jij\_24@nuaa.edu.cn. School of Mathematics, Nanjing University of Aeronautics and Astronautics, Nanjing 211106, China}}

\date{}
\maketitle
\setcounter{tocdepth}{1}


\begin{abstract}
We study the spatially homogeneous Boltzmann equation without angular cutoff for very soft potentials satisfying the inverse power law relation $\gamma+4s=1$. Our main results establish the existence, uniqueness, stability and regularization estimates of solutions in the critical Sobolev space \( H^{-(\gamma + 2s + \frac{3}{2})} \) with a logarithmic correction. Combined with the recently established monotonicity of the Fisher information, the solutions extend globally in time. Our primary tools are energy estimates based on a simultaneous dyadic localization in the phase and frequency variables, together with sharp commutator estimates between the collision operator and the localization operators.
\end{abstract}



\tableofcontents


\section{Introduction}

We consider the spatially homogeneous Boltzmann equation with very soft potentials in the non-cutoff regime. The equation reads
\begin{equation}\label{Boltzmann}
    \p_t f=Q(f,f),
\end{equation}
where $f(t,v)\ge 0$ denotes the velocity distribution at time $t>0$. The collision operator $Q$ acts only on the velocity variable and is defined as
\begin{equation*}
    Q(f,f)=\int_{\R^3}\int_{\SS^2}B(v-v_*,\sigma)\,\big(f' f'_* - f f_*\big)\d \sigma\d v_*,
\end{equation*}
where we use the standard shorthand $f=f(v)$, $f_*=f(v_*)$, $f'=f(v')$, and $f'_*=f(v'_*)$ and the post-collisional velocities satisfy
\begin{equation*}
    v'=\frac{v+v_*}{2}+\frac{|v-v_*|}{2}\sigma,\qquad 
    v'_*=\frac{v+v_*}{2}-\frac{|v-v_*|}{2}\sigma,\qquad \sigma\in \SS^2,
\end{equation*}
with the elastic collision identities
\begin{equation*}
    v+v_*=v'+v'_*,\qquad |v|^2+|v_*|^2=|v'|^2+|v'_*|^2.
\end{equation*}
We assume the collision kernel $B$ has the standard factorised form
\[
B(v-v_*,\sigma)=|v-v_*|^\gamma\,b(\cos\theta),\qquad 
\cos\theta:=\frac{v-v_*}{|v-v_*|}\cdot\sigma,
\]
it depends only on the relative speed $|v-v_*|$ and the deviation angle $\theta$. We work in the non-cutoff setting, where grazing collisions produce an angular singularity: there exist $K_b>0$ and $s\in(0,1)$ such that, for $0\le \theta\le \pi/2$,
\begin{equation*}
K_b\,\theta^{-1-2s}\le\beta(\theta):= \sin\theta\, b(\cos\theta)\le K_b^{-1}\,\theta^{-1-2s},
\end{equation*}
so that $\beta$ is not locally integrable near $\theta=0$. Without loss of generality, we may assume $b$ is supported on $\{0\le \theta\le \pi/2\}$ (equivalently $\cos\theta\ge 0$), since otherwise one can replace $b$ by its standard symmetrisation.

Throughout the paper we consider very soft potentials $-1<\gamma+2s<0$ under the inverse power law relation $\gamma+4s=1$. In particular, the parameters $\gamma$ and $s$ satisfy
\begin{equation}\label{assumption}
\gamma\in(-3,-1],\qquad s\in(1/2,1),\qquad \gamma+4s=1\quad \mbox{and}\quad-1<\gamma+2s<0. \tag{A}
\end{equation}
Under these assumptions, the collision operator exhibits a strong angular singularity with a very soft potential. The aim of this work is to establish the existence and uniqueness within a suitably weighted Sobolev space.

\subsection{Preliminaries and main results} 
We first introduce the essential notation used throughout this work.
\begin{itemize}
\item [$(i)$] We use the notations $a\ls b$ $(a\gs b)$ and $a\ls_C b$ $(a\gs_C b)$ to indicate that there is a constant $C$ which is uniform or depends on parameter $c$ and may be different on different lines, such that $a\leq Cb$ $(a\geq Cb)$. We use the notation $a\sim b$ ($a\sim_c b$) whenever $a\ls b$ and also $b\ls a$ ($a\ls_C b$ and also $a\gs_Cb$).

\item [$(ii)$] We denote $C_{a_1,a_2,\cdots,a_n}$ by a constant depending on some parameters $a_1,a_2,\cdots,a_n$. Moreover, the parameter $\eps$ is used to represent different positive numbers much smaller than $1$ and is determined in different cases.


\item [$(iii)$] $\mathbf{1}_\Omega$ is the characteristic function of the set $\Omega$. We use $(f,g)_{v}$ to denote the inner product of $f, g$ in space $L^2(\R^3_v)$.

\item [$(iv)$] We denote the Fourier transform of a function $g$ with respect to the variable $v$ by $\widehat g$, and its inverse Fourier transform by $\widecheck g$, namely,
$$
\widehat g(\xi)=\int_{\R^3} g(v)e^{-2\pi i v\cdot \xi}\d v,
\qquad
\widecheck g(\xi)=\int_{\R^3} g(v)e^{2\pi i v\cdot \xi}\d v.
$$

\item [$(v)$] Suppose $A$ and $B$ are two operators, then the commutator $[A,B]$ between $A$ and $B$ is defined by $[A,B]=AB-BA$. 
\end{itemize}

\smallskip

\noindent We  give several definitions of spaces involving different variables.

\begin{itemize}
\item [$(i)$]\textit{Function spaces in $v$ variable.} Let $f=f(v)$ and $\<v\>:=(1+|v|^2)^{1/2}$. For $m,l\in\R $, we define the weighted space $L^p_l$, $H_l^m$ and $H_l^{m,\a}$ as follows:
\begin{equation}\label{hms}
\begin{aligned}
&\vspace{0.5cm} L^p_l:=\Big\{f(v)\,\Big|\,\|f\|_{ L^p_l}=\|\<\cdot\>^lf\|_{L^p}<+\infty\Big\},\quad\\
   & H^m_l:=\Big\{f(v)\,\Big|\,\|f\|_{ H^m_l}=\|\<D\>^m\<\cdot\>^lf\|_{L^2}<+\infty\Big\}\quad\mbox{and}\\
&H^{m,\a}_l:=\Big\{f(v)\,\Big|\,\|f\|_{ H^{m,\a}_l}=\|\<D\>^m\log^\a(2+\<D\>)\<\cdot\>^lf\|_{L^2}<+\infty\Big\},
\end{aligned}
\end{equation}
where $a(D)$ is a pseudo-differential operator with the symbol $a(\xi)$, i.e., 
\begin{equation}\label{def pseudo}
    (a(D)f)(v):= \int_{\R^3}\int_{\R^3}e^{2\pi i(v-u)\xi}a(\xi)f(u)\d u\d\xi.
\end{equation}
For a non-negative function $f$, we say that $f$ has finite entropy, 
denoted by $f \in L\log L$, if 
\begin{equation*}
\cH(f) := \int_{\R^3} f \log f \, \d v < +\infty.
\end{equation*}


\item [$(ii)$]\textit{Function spaces in $t, v$ variables.} Let $f=f(t, v)$ and $X$ be a function space in $v$ variables. Then $L^p([0,T],X)$ and $ L^\infty([0,T],X)$ are defined by
 \[ L^p([0,T],X):=\Big\{f(t,x,v)\bigg|\, \|f\|^p_{ L^p([0,T],X)}=\int_0^T\|f(t)\|^p_{X}\d t <+\infty\Big\},\quad 1\leq p<\infty,\]
  \[L^\infty([0,T],X):=\Big\{f(t,x,v)\,\bigg|\,\|f\|_{ L^\infty([0,T],X)}=\mathrm{esssup}_{t\in [0,T]}\|f(t)\|_{X}<+\infty\Big\}. \]
  We also define space $L^p_{loc}([0,\infty),X):=\bigcap _{T<\infty}L^p([0,T],X)$.
\end{itemize}

Let us recall some standard facts for the spatially homogeneous Boltzmann equation. Formally, if $f$ solves \eqref{Boltzmann} with some initial datum $f_0$, then mass, momentum, and energy are conserved:
\begin{equation*}
    \frac{\d}{\d t}\int_{\R^3} f(t,v)\varphi(v)\d v=0,
\qquad \varphi(v)=1,\ v,\ |v|^2 .
\end{equation*}
For convenience, we normalize the initial data to satisfy
\begin{equation}\label{f0}
\int_{\R^3} f_0(v)\d v=1,\qquad 
\int_{\R^3} v\,f_0(v)\d v=0,\qquad
\int_{\R^3} |v|^2 f_0(v)\d v=3, \quad \text{for } f_0\geq 0.
\end{equation}
The entropy $\cH(f)$ defined above is decreasing along smooth solutions:
\begin{equation*}
    \frac{\d }{\d t}\cH(f(t))=-\cD(f(t))\le 0,
\end{equation*}
where the entropy dissipation is given by
\begin{equation*}
    \begin{aligned}
\cD(f):=\frac14\int_{\R^3}\int_{\R^3}\int_{\SS^2}
B(v-v_*,\sigma)\,(f_*'f'-f_*f)\,
\log\frac{f_*'f'}{f_*f}\d\sigma\d v_*\d v \geq 0 .
\end{aligned}
\end{equation*}

Throughout this work, we consider weak solutions in the following sense.

\begin{definition}\label{def:weak-solution}
Let $f_0$ satisfies \eqref{f0}. A nonnegative measurable function
\begin{equation*}
    f(t,v)\in L^{\infty}_{loc}\left([0,\infty), L^1_2\ccap L\log L(\R^3)\right)
\end{equation*}
is a weak solution of \eqref{Boltzmann} if the following hold:

\smallskip
\noindent{\rm (i)} $f$ conserves mass, momentum and energy, and satisfies the entropy inequality
\begin{equation*}
    \cH(f(t))+\int_{0}^{t}\cD(f(s))\d s\le \cH(f_0),\qquad t\ge 0.
\end{equation*}

\smallskip
\noindent{\rm (ii)} For any $\psi\in C_c^2(\R^3)$, the map
\begin{equation*}
    t\mapsto \int_{\R^3}\psi(v)\, f(t,v)\d v
\end{equation*}
is absolutely continuous on $[0,\infty)$ and, for a.e. $T\in[0,\infty)$, the following identity holds:
\begin{equation}\label{weakform}
\begin{aligned}
&\int_{\R^3} f(T,v)\psi(v)\d v-\int_{\R^3} f_0(v)\psi(v)\d v
\\
=&\lim_{\varepsilon\to 0}\frac{1}{2}\int_0^T\int_{\R^6}f(v)f(v_*)
\int_{\SS^2}B^{\varepsilon}(v-v_*,\sigma)\,
\big[\psi(v')+\psi(v'_*)-\psi(v)-\psi(v_*)\big]\,
\d\sigma\d v\d v_*\d t,
\end{aligned}
\end{equation}
where
\begin{equation*}
    B^{\varepsilon}(v-v_*,\sigma)=B(v-v_*,\sigma)\mathbf{1}_{\frac{v-v_*}{|v-v_*|}\cdot\sigma\le \cos\varepsilon}.
\end{equation*}
\end{definition}

\begin{remark}
 The collision operator $Q$ on the right hand side of \eqref{weakform} is defined  in the sense of principal value through the deviation angle $\theta$. This is crucial in particular for the high singularity case $s\ge 1/2$ (see \cite{HJZ} for more details).
\end{remark}
\begin{remark} The existence of weak solutions was proved in \cite{villani1998new}. We also remark that \eqref{weakform} can be replaced by the following formula, that is,  for $\psi(t,v)\in C^1([0,\infty);C_c^2(\R^3))$ and $T\ge0$,
\beno 
&& \int_{\R^3} f(T,v)\psi(T,v)\d v -\int_{\R^3} f_0\psi(0,v)\d v -\int_0^T\int_{\R^3} f(t,v)\partial_t\psi(t,v)\d v \d t \nonumber\\
=&&\lim_{\varepsilon\to 0}\frac{1}{2}\int_0^T\int_{\R^6}f(v)f(v_*)
\int_{\SS^2}B^{\varepsilon}(v-v_*,\sigma)\,
\big[\psi(t,v')+\psi(t,v'_*)-\psi(t,v)-\psi(t,v_*)\big]\,
\d\sigma\d v\d v_*\d t.\eeno
\end{remark}

Now we state our main results.
Recall the assumption \eqref{assumption} on the parameters, where $s=\frac{1-\gamma}{4}$, and we separate the well-posedness result into cases:
\begin{flalign}\label{VS}
 &\quad\text{{\bf $\bullet$ very soft potential} where} -1<\gamma+2s<-\frac12, \text{ namely } -3<\gamma<-2;& \tag{VS}
\end{flalign}
\vspace{-0.6cm}
\begin{flalign}\label{MS}
 &\quad\text{{\bf$\bullet$ moderately soft potential} where} -\frac12\leq\gamma+2s<0, \text{ namely } -2\leq\gamma<-1.& \tag{MS}
\end{flalign}
 For the very soft potential case, we obtain the following results.
\begin{theorem}[Existence, regularization, and uniqueness]\label{maintheorem1}
Assume that $\gamma$ and $s$ satisfy \eqref{VS} in the very soft potential
regime, and let the initial datum $f_0$ satisfy \eqref{f0}. Suppose, in
addition, that
\begin{equation}\label{L1}
    f_0\in
    L^1_{l+\gamma+s}(\R^3)
    \ccap
    H^{-\zeta,\a}_{l}(\R^3),
    \qquad
    \a>0,
    \qquad
    l\geq-\gamma,
\end{equation}
where
\[
    \zeta
    :=
    \gamma+2s+\frac32
    \;\big(=2+\tfrac{\gamma}{2} \text{ under \eqref{assumption}}\big).
\]
Then the following statements hold.

\begin{enumerate}

\item[(i)] 
There exists $T>0$ and a local-in-time weak solution $f$ satisfying
\[
    f\in
    L^\infty\big([0,T],H^{-\zeta,\a}_{l}(\R^3)\big)
    \ccap
    L^2\big([0,T],H^{s-\zeta,\a}_{l+\frac{\gamma}{2}}(\R^3)\big).
\]
If, in addition,
\[
    \a>\frac12
    \qquad\text{and}\qquad
    l\geq-\gamma+2s,
\]
then this local solution is unique.

\item[(ii)]
Let $n\geq0$. If, in addition,
\[
    l
    \geq
    -\gamma-\frac{\gamma(n+\zeta)}{2s}
    =
    -\frac{(5+2n)\gamma}{1-\gamma},
\]
then, for every $t\in(0,T]$,
\begin{equation}\label{regularisation-estimate}
\begin{aligned}
    \|f(t)\|_
    {H^{n,\a}_{\,l+\frac{\gamma(n+\zeta)}{2s}}}
    \ls
    t^{-\frac{n+\zeta}{2s}}
    =
    t^{-\frac{2n+\gamma+4}{1-\gamma}}.
\end{aligned}
\end{equation}
In particular, the solution becomes instantaneously Sobolev regular for
every positive time, provided that the corresponding weight condition
above is satisfied.

\item[(iii)]
Suppose furthermore that $\a>\frac12$ and that the initial datum has the
higher moment
\[
    f_0\in
    L^1_{\frac{4l}{5+\gamma}+1}(\R^3)
    \ccap
    H^{-\zeta,\a}_{l}(\R^3),\quad \text{with}\quad l\geq-\frac{9\gamma}{1-\gamma}.
\]
Then the local solution obtained in {\rm (i)} extends to a global-in-time
solution. Moreover, this global solution is unique.

\end{enumerate}
\end{theorem}

We also have the following stability estimate.
\begin{corollary}
Suppose \( \gamma + 2s \in (-1, -\frac{1}{2}) \),  \(f\) and \(g\) are two solutions constructed above with initial data \(f_0\) and \(g_0\), respectively, then, for $\a>\frac{1}{2}$ and $\zeta$ defined above, it holds
\begin{equation}\label{estimateofunique}
    \|f(t)-g(t)\|_{H^{-\zeta,\a}_{-\gamma}}
    \ls C_{\a,f_0,g_0}e^{\int_0^tC(\omega)d\omega}
    \|f_0-g_0\|_{H^{-\zeta,\a}_{-\gamma}},
    \qquad \forall\, t\in [0,T],
\end{equation}
where
\[
    C(\omega)
    =
    1+\|f(\omega)\|^2_{H^{s-\zeta,\a}_{-\frac{\gamma}{2}}}
    +\|g(\omega)\|^2_{H^{s-\zeta,\a}_{-\frac{\gamma}{2}+2s}}
    +\omega^{-\frac{2s-\tau}{2s}}
    \in L^1[0,T]
\]
with \(\tau=\min\{1,2s-\zeta\}>0\). For the case \( \gamma + 2s = -\frac{1}{2} \), the results are similar to \eqref{estimateofunique}, except that \( \a \) is replaced with \( \a + 1 \).
\end{corollary}


For the moderately soft potential case, we obtain the following results.

\begin{theorem}[Moderately soft potentials]\label{maintheorem2}
Assume that $\gamma$ and $s$ satisfy \eqref{MS} in the moderately soft
potential regime, and let the initial datum $f_0$ satisfy \eqref{f0}.
Suppose, in addition, that
\begin{equation}\label{L1moderate}
f_0\in
\begin{cases}
L^1_{l+\gamma+s}(\R^3)
\ccap
H^{-1,\a}_l(\R^3),
&
-\dfrac12<\gamma+2s<0,
\\[2mm]
L^1_{l+\gamma+s}(\R^3)
\ccap
H^{-1,\a+1}_l(\R^3),
&
\gamma+2s=-\dfrac12,
\end{cases}
\end{equation}
where
\[
\a>0,
\qquad
l\geq-\gamma.
\]
Then the following statements hold.

\begin{enumerate}

\item[(i)] 
There exists $T>0$ and a local-in-time weak solution $f$ satisfying
\[
f\in
\begin{cases}
L^\infty\big([0,T],H^{-1,\a}_l\big)
\ccap
L^2\big([0,T],H^{s-1,\a}_{l+\frac{\gamma}{2}}\big),
&
-\dfrac12<\gamma+2s<0,
\\[2mm]
L^\infty\big([0,T],H^{-1,\a+1}_l\big)
\ccap
L^2\big([0,T],H^{s-1,\a+1}_{l+\frac{\gamma}{2}}\big),
&
\gamma+2s=-\dfrac12.
\end{cases}
\]
If, in addition,
\[
\a>\frac12
\qquad\text{and}\qquad
l\geq-2\gamma,
\]
then the local solution is unique.

\item[(ii)] 
Let $n\geq0$. If
\begin{equation}\label{moderate-weight-n}
l
\geq
-\gamma-\frac{\gamma(n+1)}{2s},
\end{equation}
then, for every $t\in(0,T]$,
\begin{equation}\label{moderate-regularisation}
\begin{cases}
\|f(t)\|_
{H^{n,\a}_{\,l+\frac{\gamma(n+1)}{2s}}}
\ls
t^{-\frac{n+1}{2s}},& -\f12<\gamma+2s<0,\\
\|f(t)\|_
{H^{n,{\a+1}}_{\,l+\frac{\gamma(n+1)}{2s}}}
\ls
t^{-\frac{n+1}{2s}},& \gamma+2s=-\f12.
\end{cases}
\end{equation}

\item[(iii)] 
Assume further that $\a>\frac12$ and that
\begin{equation}\label{moderate-global-moment}
f_0\in
\begin{cases}
L^1_{\frac{4l}{3}+1}(\R^3)
\ccap
H^{-1,\a}_l(\R^3),
&
-\dfrac12<\gamma+2s<0,
\\[2mm]
L^1_{\frac{4l}{3}+1}(\R^3)
\ccap
H^{-1,\a+1}_l(\R^3),
&
\gamma+2s=-\dfrac12.
\end{cases}
\end{equation}
Suppose moreover that $
l\geq
2-\f{6\gamma}{1-\gamma}$,
then the local solution obtained in {\rm (i)} extends globally in time
and is unique.
\end{enumerate}
\end{theorem}

 \begin{remark}
The non-cutoff Boltzmann equation is comparable to the semi-linear heat-type equation
\begin{equation}\label{toymodel}
    \partial_t f = \langle v \rangle^\gamma (-\Delta_v)^s f + (|\cdot|^\gamma * f)f,
\end{equation}
since the collision operator behaves like a fractional Laplace operator. Specifically, we have 
$-Q(f, \cdot) \sim C_f \langle v \rangle^\gamma (-\Delta_v)^s + \text{L.O.T.}$ It is straightforward to see that the critical space for the solution is \( L^p \) with \( p = \frac{3}{3 + \gamma + 2s}=\frac{6}{7+\gamma} \), which corresponds to the Sobolev space \( H^{-(\gamma + 2s + \frac{3}{2})} \) (see Subsection \ref{critical}).
 \end{remark}

\begin{remark}
 For local existence, we can choose \(l=-\gamma\), so it suffices that \(f_0\in L^1_2\) in \eqref{L1} since $l+\gamma+s=s<1$. For local uniqueness, we treat the cases \(\gamma+2s\in(-1,-1/2]\) and \(\gamma+2s\in(-1/2,0)\) by two distinct methods. The former is proved directly by the energy method. Choosing \(l=-\gamma+2s\) gives \(l+\gamma+s=3s<3\), so it suffices that \(f_0\in L^1_3\). The latter relies on the conditional uniqueness result (see Proposition \ref{conditionaluniqueness}) together with regularity estimates. The extension to global well-posedness relies on the Fisher information estimate obtained in  \cite{imbert2026monotonicity}.
\end{remark}

\begin{remark}
When \(\gamma = -3\) and \(s = 1\), this reduces to the Landau–Coulomb equation, for which the parallel well-posedness theory has been established in \cite{HeJiLuo2024}. However, the non-locality of the Boltzmann collision operator is stronger than that of the Landau collision operator. See Subsection \ref{strategy} for a simple explanation or refer to Subsection \ref{proveunique} for details.
\end{remark}

\subsection{Previous related work}
Since the Boltzmann equation without angular cutoff has been widely investigated, we will only mention those works that are closely related to ours.

Many results concerning existence are available; see \cite{DesvillettesMouhot,DW,villani1998new} and the references therein. Prior to \cite{imbert2026monotonicity}, global-in-time existence had only been proved when \(\gamma \ge -2\). The work of \cite{imbert2026monotonicity} established the monotonicity of the Fisher information for the Boltzmann equation, which leads to the global existence of solutions under very soft potentials.

Regarding regularity, the non-cutoff Boltzmann equation, including the inhomogeneous case, has been extensively studied. Notable results include those in Sobolev spaces \cite{AMU1, CH1, CH2, DW}, as well as \( C^\infty \) estimates established using De Giorgi iterations and the Schauder argument \cite{IS1, IS2, IS4}. We note that the latter methods can also be utilized to prove well-posedness (see \cite{henderson2025classical, HST2, HW}).

Uniqueness is more delicate and is often obtained in a conditional form, requiring additional integrability or regularity.  For instance, Desvillettes--Mouhot \cite{desvillettes2009stability} proved stability estimates for the homogeneous equation including hard and moderately soft potentials, which can be combined with propagation of moments or regularity to conclude uniqueness. A complementary probabilistic approach uses Wasserstein distances to control the stability of solutions. Fournier \cite{fournier2008uniqueness} proved inequalities for Wasserstein distances leading to uniqueness and well-posedness results for classes of non-cutoff kernels, including soft potentials. More precisely, the authors obtained
\begin{theorem}[Conditional uniqueness, \cite{fournier2008uniqueness}]
\label{conditionaluniqueness}
Under our assumptions \eqref{assumption}, the weak solution \(f(t,v)\) is unique if it satisfies \(f\in L^{\infty}([0,T],L^1_2(\mathbb{R}^3))\ccap  L^1_{\rm{loc}}([0,T], L^p(\mathbb{R}^3))\) for some \(p\in\left(\frac{3}{3+\gamma},\infty\right)\).
\end{theorem}
More recent works continue to refine both the existence theory and the conditional uniqueness criteria in soft and very soft regimes, often by combining entropy dissipation bounds with weighted $L^p$ or Sobolev estimates; see, for example, the recent entropy dissipation lower bounds and their applications in \cite{jamil2023entropy}.  

We also mention some corresponding conclusions related to the Landau equation, which is closely related to the non-cutoff Boltzmann equation through grazing-collision limits, but with an anisotropic diffusion structure in $v$. In the Coulomb case, the conditional uniqueness is due to Fournier in \cite{FournierLandau}. He proved the uniqueness of solutions in the class $L^\infty_{loc}([0,\infty); L^1_2(\R^3)) \ccap  L^1_{loc}([0,\infty);L^\infty(\R^3))$. This result implies local well-posedness, assuming further that the initial data lie in $L^p (\R^3)$ with $p>3/2$. For more details, we refer readers to \cite{golding2024local}. It is worth mentioning that the authors of \cite{GSUN} recently applied the method of the \(\mathcal{M}\)-operator to obtain uniqueness of solutions in the critical space \(L^{3/2}\), where the existence in this space was obtained in \cite{golding2025global}. 
In fact, solutions exist globally because of the breakthrough by Guillen and Silvestre in \cite{guillen2025landau}, where they proved that the Fisher information is monotonically decreasing in time, ensuring that classical solutions to the equation will never blow up. Note that the space \(H^{-1/2}\) is precisely the \(H^s\)-space into which \(L^{3/2}\) embeds. Correspondingly, a well-posedness result is established for initial data in a logarithmically modified Sobolev space \(H^{-1/2,\a}\) with \(\a>\frac{1}{2}\) in \cite{HeJiLuo2024}.

In this paper, we aim to obtain well-posedness for the equation \eqref{Boltzmann} in the critical negative Sobolev space (hereafter,``negative" refers to Sobolev spaces with negative regularity indices). 
 As can be seen from \eqref{fracLap}, Theorem \ref{conditionaluniqueness} does not apply here, similar to the Landau-Coulomb case. Therefore, we need to combine moment propagation with a stability mechanism. We will utilize the following result on the propagation of \( L^1 \) moments:
\begin{proposition}[\cite{Lu2009}]\label{propagationofL1moment}
Under our assumptions \eqref{assumption}, the weak solutions satisfy that if \( f_0 \in L^1_l \) with \( l > 2 \), then 
\begin{equation}
    \|f(t)\|_{L^1_l} \leq C_l(1+t),\quad \forall t \in \mathbb{R}^+,
\end{equation}
where \( C_l \) depends only on \( \gamma, l, K_b\) and \(\|f_0\|_{L^1_2} \).
\end{proposition}

\subsection{Explanation of the critical space}\label{critical}

Before presenting our proof, we briefly explain why the negative Sobolev space with logarithmic modification
\[
H^{-(\gamma+2s+\frac32),\a}
\]
is the natural functional setting. The discussion below is heuristic and is based on the toy model \eqref{toymodel}.

From the viewpoint of local existence, the standard $L^p$ estimate takes the form
\begin{equation}\label{Lp-estimate}
\frac{\d}{\d t}\|f\|_{L^p_l}^p
\lesssim
-\|f^{p/2}\|_{H^s_{\gamma/p+l}}^2
+
\int_{\R^3}(|\cdot|^\gamma*f)\,f^p\<v\>^{pl}\d v.
\end{equation}
The dissipation controls the higher integrability
\[
\|f^{p/2}\|_{H^s}
\gtrsim
\|f\|_{L^{\lambda p}},
\qquad
\lambda=\frac3{3-2s},
\]
while the nonlinear term is estimated by the Hardy--Littlewood--Sobolev inequality. Balancing these two terms leads to the condition
\[
\frac3{\gamma+2s+3}
<
p
<
\frac3{3+\gamma}.
\]
Therefore,
\[
p=\frac3{\gamma+2s+3}
\]
is the critical integrability exponent for the local existence theory.

The corresponding Sobolev space is obtained from the embedding
\[
L^{p}(\R^3)
\hookrightarrow
H^{-(\gamma+2s+\frac32)}(\R^3),
\]
which explains the appearance of the negative-order Sobolev exponent in our main theorem.

The same space also naturally arises from the uniqueness problem. Indeed, for the linear fractional heat equation,
\begin{equation}\label{fracLap}
 \partial_t f+(-\Delta)^sf=0,   
\end{equation}
the semigroup satisfies
\[
\|e^{-t(-\Delta)^s}f_0\|_{L^{\frac3{3+\gamma}}}
\lesssim
t^{-1}|\log t|^{-\a}
\|f_0\|_{H^{-(\gamma+2s+\frac32),\a}}.
\]
The exponent $L^{\frac3{3+\gamma}}$ coincides with the critical space appearing in Fournier's conditional uniqueness result \cite{fournier2008uniqueness}. However, the estimate above is only borderline integrable in time, so the conditional uniqueness argument cannot be applied directly. This motivates the logarithmic modification and the negative Sobolev framework adopted in the present paper.

\color{black}

\subsection{Strategy of the Proof}\label{strategy}
The main idea is to combine the energy method in negative Sobolev spaces
with a simultaneous localization in the phase and frequency variables, in order to obtain sharper estimates. Specifically, we localize the equation \eqref{Boltzmann} to obtain
\begin{equation*}
    \partial_t \cP_k\F_jf=\cP_k\F_jQ(f,f),
\end{equation*}
where we refer to Subsection \ref{subsec:decomposition} for the definition of the dyadic decomposition, 
and we further get the energy estimate 
\begin{equation*}
\begin{aligned}
    \f12\f{\d}{\d t }\|\cP_k\F_jf\|^2_{L^2}=\&(\cP_k\F_jQ(f,f),\cP_k\F_j f)_{v}\\
    = \&(Q(f,\cF_j\cP_kf),\cF_j\cP_kf)_v+(\cF_jQ(f,\cP_kf)-Q(f,\cF_j\cP_kf),\cF_j\cP_kf)_v\\\&+(\cP_kQ(f,f)-Q(f,\cP_k f),\cF_j^2\cP_kf)_v.
\end{aligned}
\end{equation*}
The first term is the principal coercive contribution.  The second term
measures the failure of the Fourier localization $\F_j$ to commute with
the collision operator, while the third term measures the corresponding
failure of the physical localization $\cP_k$ to commute with $Q$.
Accordingly, the proof is reduced to a coercivity estimate together with
sharp bounds for these two types of commutators.  These estimates are developed in Sections~\ref{coercivityandupperbound}--\ref{commutatorestimatesQPk}.
At the heuristic level, the coercivity estimate yields
\begin{equation}\label{eq:schematic-local-dissipation}
    -\big(Q(f,\F_j\cP_kf),\F_j\cP_kf\big)_v
    \gtrsim
    2^{\gamma k}2^{2sj}
    \|\cP_k\F_jf\|_{L^2}^2
    +
    \text{remaining terms}.
\end{equation}

At this point, we can utilize the characterization of Sobolev spaces through dyadic localization (see Lemma \ref{le1.4}). By multiplying by appropriate coefficients and summing, we obtain energy estimates in any weighted Sobolev space. For example, by multiplying \(2^{2lk}2^{2nj}j^{2\a}\), with \(l, n \in \mathbb{R},\a>0\), and summing over \(j, k \geq -1\), we have 
\begin{equation*}
    \frac{1}{2} \frac{\d}{\d t } \sum_{j,k=-1}^\infty 2^{2lk}2^{2nj}j^{2\a} \|\cP_k \F_j f\|^2_{L^2} \leq -c 2^{(\gamma+2l) k}2^{2(s+n)j}j^{2\a}\|\cP_k \F_j f\|^2_{L^2} + \text{remaining terms},
\end{equation*}
where the remaining terms here are not the same as the previous ones. More precisely, it holds
\begin{equation*}
    \frac{1}{2} \frac{\d}{\d t }\|f\|^2_{H^{n,\a}_l}+\kappa\|f\|^2_{H^{n+s,\a}_{l+\gamma/2}}\leq\text{remaining terms},
\end{equation*}
where the existence and regularity estimates of solutions can be obtained by choosing suitable values of \(n,l\) and $\a$.

For uniqueness, we directly use the method of differences. Consider the case \(\gamma + 2s < -\frac{1}{2}\). Let \(f\) and \(g\) be two solutions with the same initial value \(f_0\); then \(h = f - g\) satisfies the equation, at least formally,
\[\partial_t h = Q(f, h) + Q(h, g),\quad h(0)=0.
\]
Similarly, we have
\begin{equation}\label{strah}
  \begin{aligned}
&\frac{1}{2}\frac{\d}{\d t}\sum_{j,k=-1}^\infty2^{-2\gamma k}2^{-2\zeta j}j^{2\a}\|\cF_j\cP_k h\|_{L^2}^2=2^{-2\gamma k}2^{-2\zeta j}j^{2\a}\Big((Q(f,h),\cP_k\cF_j^2\cP_kh)_v+(Q(h,g),\cP_k\cF_j^2\cP_kh)_v\Big).
\end{aligned}  
\end{equation}
We aim to obtain
$$
\frac{\d}{\d t } \|h(t)\|^2_{H^{-\zeta,\a}_{-\gamma}}+c\|h(t)\|^2_{H^{s-\zeta,\a}_{-\gamma/2}} \le C(t) (1+\|h(t)\|^3_{H^{-\zeta,\a}_{-\gamma}}),
$$
with $\zeta$ defined in Theorems \ref{maintheorem1} and \ref{maintheorem2} while \(C(t)\) is an integrable function of time \(t\). Thus, Grönwall's inequality yields uniqueness. 

To this end, in addition to handling the first term on the right-hand side of equation \eqref{strah} using the same method as before, we also need to address the second term. We point out that, when considering the Landau-Coulomb collision operator
\begin{equation*}
Q_L(g,h)=\nabla\cdot\left([a*g]\nabla h-[a*\nabla g]h\right),\quad \mbox{with}\quad a(z)=|z|^{-1}\Big(\mathrm{Id}-\f{z\otimes z}{|z|^2}\Big),
\end{equation*}
it always holds that 
\[(Q_L(h,g), \mathcal{P}_k \F_j^2 \mathcal{P}_k h)_v \sim  (Q_L(h, \mathcal{P}_k g), \mathcal{P}_k \F_j^2 \mathcal{P}_k h)_v,\] 
because the variables of the last two functions are consistent in the phase space. However, the Boltzmann operator is different, which is one of the reasons we refer to the Boltzmann collision operator as more non-local. 

Combining \eqref{ubdecom} and \eqref{eq m1234}, we can decompose it into several terms, among which the most challenging term to handle is given by
\[
\sum_{a>j}\sum_{m>k} 2^{-2\gamma k} 2^{-2\zeta j} j^{2\a}\left(Q(\F_a\cP_m h,\F_j\cU_{m}g),\cP_k\cF_jh\right)_v.
\]
We can obtain the following bound for this term:
\[
\sum_{a>j}\sum_{m>k} 2^{-2\gamma k} 2^{-2\zeta j} j^{2\a} 2^{(\gamma+\frac{3}{2})m} 2^{\theta(m-k)} 2^{(2s-\theta)(j-a)} \|\F_a\cP_m h\|_{L^2} \|\F_j\cU_{m}g\|_{L^2} \|\cP_k\cF_jh\|_{L^2}, \quad \forall \theta \in [0, 1].
\]
Since \(\gamma < -2s-\f12<-\frac{3}{2}\) and \(\zeta \leq1 < 2s\), we select \(\theta \in \left(0, \frac{1}{2}\min\{2s - \zeta, -\gamma\}\right)\) and use Lemma \ref{le1.4} and Cauchy-Schwarz inequality to get the bound:
\[
\begin{aligned}
    &\sum_{a>j}\sum_{m>k} 2^{(\gamma + \theta)(m-k)} 2^{(\gamma + \frac{3}{2})m} 2^{(\zeta - 2s + \theta)a}a^{-\a} (2^{-\zeta a} a^\a 2^{-\gamma m} \|\F_a\cP_m h\|_{L^2})(2^{(2s - \theta - \zeta)j} j^\a \|\F_j \cU_m g\|_{L^2}) \\
    &\times \left(2^{-\zeta j} j^\a 2^{-\gamma k} \|\cP_k \cF_j h\|_{L^2}\right) \ls \|g\|_{H^{2s-\theta-\zeta,\a}} \|h\|^2_{H^{-\zeta,\a}_{-\gamma}} \ls t^{-\frac{2s-\theta}{2s}} \|h\|^2_{H^{-\zeta,\a}_{-\gamma}},
\end{aligned}
\]
where the coefficient is integrable with respect to \(t \in [0, T]\), as derived from Theorem \ref{maintheorem1} \eqref{regularisation-estimate}. Thus, the uniqueness of the solution holds.

\subsection{Organization of the paper}
The rest of this paper is organized as follows.  In Section \ref{Auxiliarytools}, we present the basic framework of the dyadic decomposition and its applications to weighted Sobolev spaces and the collision operator. In Section \ref{coercivityandupperbound}, we establish the coercivity estimate and upper bounds for the collision operator. In Section \ref{commutatorestimatesQFj} and Section \ref{commutatorestimatesQPk}, we derive commutator estimates for the collision operator with the localization operators \(\mathcal{P}_k\) and \(\F_j\). With these preparations in place, we prove the main theorem in Section \ref{proofofmaintheorem} and Section \ref{sec:global}. Technical tools and lemmas are collected in the appendix.

\bigskip

\section{Auxiliary tools}\label{Auxiliarytools}
In this section, we list some basic facts on dyadic decomposition and then apply it to the localization in Sobolev spaces and to the collision operator.

\subsection{Dyadic decomposition}\label{subsec:decomposition}
Let $B_{\frac{4}{3}}:=\{\xi\in\R^3\mid|\xi|\leq\frac{4}{3}\}$ and $C:=\{\xi\in\R^3\mid \frac{3}{4}\leq|\xi|\leq\frac{8}{3}\}$. We introduce two radial non-negative functions $\psi\in C_0^\infty(B_{\frac{4}{3}})$ and $\vphi\in C_0^\infty(C)$ such that 
$$
\psi(\xi)+\sum_{j\geq0}\vphi(2^{-j}\xi)=1,~\xi\in\R^3,
$$
moreover, their supports satisfy the following separation properties:
\begin{equation*}
    \begin{aligned}
 &     \operatorname{supp}\vphi(2^{-j}\cdot)\ccap 
\operatorname{supp}\vphi(2^{-k}\cdot)=\emptyset
\quad \text{when}\quad |j-k|\geq 2,\\
&\mbox{and}\quad\operatorname{supp}\psi(\cdot)\ccap 
\operatorname{supp}\vphi(2^{-j}\cdot)=\emptyset
\quad \text{when}\quad j\geq 1.  
    \end{aligned}
\end{equation*}
We denote the inverse Fourier transform of $\vphi$ and $\psi$ by $\widecheck\vphi$ and $\widecheck\psi$, and we also let $\widecheck\vphi_j(\cdot)=2^{3j}\widecheck\vphi(2^j\cdot)$. 
 
\medskip

\noindent{\bf Dyadic decomposition in  phase space}: The dyadic operator $\cP_l$ in  phase space can be defined as follows:
\begin{equation*}
    \cP_{-1}f(x):=\psi(x)f(x),\qquad\cP_lf(x):=\vphi(2^{-l}x)f(x),\,\,l\geq0. 
\end{equation*}
By definition, there exists a integer $N_0\geq2$ such that  $\cP_l\cP_m=0$ if $|l-m|\geq N_0$. For any smooth function $f$, we can verify that  \[f=\cP_{-1}f+\sum\limits_{l\geq0}\cP_lf.\]
For later use, we also define an enlarged phase cutoff
$\tP_lf(x):=\sum\limits_{|k-l|\leq N_0}\cP_kf(x)$ and an low-phase cutoff $\cU_lf(x):=\sum\limits_{k\leq l}\cP_kf(x)$.

\medskip

\noindent {\bf Dyadic decomposition in frequency space}:
The dyadic operator $\F_j$  in frequency space can be defined as follows:
\begin{equation*}
    \F_{-1}f(x):=\int_{\R^3}\widecheck\psi(x-y)f(y)dy,\qquad\F_jf(x):=\int_{\R^3}\widecheck\vphi_j(x-y)f(y)dy,\,\,j\geq0. 
\end{equation*}
Then for any tempered distribution $f$, it holds that
 \[f:=\F_{-1}f+\sum_{j\geq0}\F_jf.\]
For later use, we also define  an enlarged frequency cutoff $\tF_jf(x):=\sum\limits_{|k-j|\leq 3N_0}\F_kf(x)$ and an low-frequency  cutoff  $\cS_jf(x):=\sum\limits_{k\leq j}\F_kf(x)$. 

\subsection{Applications in Sobolev spaces}
\begin{definition}\label{Fj} Let  $\alpha=(\alpha_1,\alpha_2,\alpha_3)\in \N^3,|\alpha|:=\alpha_1+\alpha_2+\alpha_3$ and $$\vphi_\alpha:=(\frac{1}{i}\partial_{x_1})^{\alpha_1}(\frac{1}{i}\partial_{x_2})^{\alpha_2}(\frac{1}{i}\partial_{x_3})^{\alpha_3}\vphi.$$ To simplify the presentation of the estimates for the commutator $[\cP_k, \F_j]$, we introduce  $\cP_{j,\alpha}, \F_{j,\alpha}$ and $\hat{\F}_{j,\alpha}$ defined by
\begin{equation*}
    \begin{aligned}
&\mathcal P_{-1,\alpha} f:=\psi_{\alpha} f, 
\qquad \qquad \qquad 
\mathcal P_{l,\alpha} f:=\varphi_\alpha(2^{-l}\cdot) f,
\qquad l\geq 0,\\
&\F_{-1,\alpha} f:=\psi_{\alpha}(D) f,
\qquad \qquad \,
\F_{j,\alpha} f:=\varphi_\alpha(2^{-j}D) f,
\qquad j\geq 0,\\
&\widehat{\F}_{-1,\alpha} f:=(\psi^2)_{\alpha}(D) f,
\qquad \quad
\widehat{\F}_{j,\alpha} f:=(\varphi^2)_\alpha(2^{-j}D) f,
\qquad j\geq 0,
\end{aligned}
\end{equation*}
where $\hat{\F}$ will be used in commutator estimates.
To unify  notations $\F_j, \F_{j,\alpha}$, $\hat{\F}_{j,\alpha}$ and $\cP_l, \cP_{l,\alpha}$, we introduce  localized operators $\mF_j$   and $\mP_j$:
\smallskip

  \noindent {\rm (i)} The support of the Fourier transform of $\mF_jf$ and the support of $\mP_lf$  will be localized in the annulus $\{|\cdot|\sim 2^j\}$ and $\{|\cdot|\sim 2^l\}$, respectively.\smallskip

 \noindent{\rm (ii)} It holds that for   fixed $N\in\N$, $\|\F_j f\|_{L^2}+\sum\limits_{|\alpha|\le N} (\|\F_{j,\alpha}f\|_{L^2}+\|\hat{\cF}_{j,\alpha}f\|_{L^2})\le C_N\|\mF_jf\|_{L^2}$ and $\|\tP_l f\|_{L^2}+\sum\limits_{|\alpha|\le N} \|\tP_{l,\alpha}f\|_{L^2}\le C_N\|\mP_lf\|_{L^2}$.
 
Finally, we also introduce $\mU_l:=\sum\limits_{k\leq l}\mP_k$ and $\mS_j:=\sum\limits_{k\leq j}\mF_k$.
\end{definition}

 We now present some basic properties of the localized operators and their commutators.
\begin{lemma}\label{PFcommutator}    
(see Lemma 4.14 in \cite{CHJ}) $\bullet$ For any $N\in\N$, there exists a constant $C_N$ such that
        \begin{equation}\label{pkfj}
            \begin{aligned}
            &\|[\cP_k,\cF_j]f\|_{L^2}=\|(\cP_k\cF_j-\cF_j\cP_k)f\|_{L^2}\leq C_N(2^{-j}2^{-k}\sum_{|\alpha|=1}^{2N}\|\cP_{k,\alpha}\cF_{j,\alpha}f\|_{L^2}+2^{-jN}2^{-kN}\|f\|_{H^{-N}_{-N}}),\\
            &\|[\cU_k,\cF_j]f\|_{L^2}=\|(\cU_k\cF_j-\cF_j\cU_k)f\|_{L^2}\leq C_N(2^{-j}\sum_{|\alpha|=1}^{2N}\|\cU_{k,\alpha}\cF_{j,\alpha}f\|_{L^2}+2^{-jN}\|f\|_{H^{-N}_{-N}}).
                    \end{aligned}
        \end{equation}
        
        $\bullet$ For $|m-p|> N_0$ and $\forall N\in \N$, there exists a constant $C_N$ such that
         \begin{equation}\label{m-p>N0}
        \begin{aligned}
            \|\cF_m\cP_k\cF_pg\|_{L^2}\leq C_N 2^{-(p+m+k)N}\|\cF_pg\|_{L^2_{-N}},~~\|\cF_m\cU_k\cF_pg\|_{L^2}\leq C_N 2^{-(p+m)N}\|\cF_pg\|_{L^2_{-N}},
        \end{aligned}
        \end{equation}

    $\bullet$ For any $a,\omega,\l\in\R$ and $j\geq-1$, we have
    \begin{equation}\label{a001}
       \begin{aligned}
        \|\cU_{k}h\|_{H^a}\leq C_{a,\omega}2^{k(-\omega)^+}\|h\|_{H^a_{\omega}},~\|\cS_{j}h\|_{L^2_l}\leq C_l\|h\|_{L^2_l},~\|\cS_{j}f\|_{L^1}+\|\F_j f\|_{L^1}\leq C_l \|f\|_{L^1}.
    \end{aligned} 
    \end{equation}
\end{lemma}

We can also characterize Sobolev spaces by means of the dyadic localization.
\begin{lemma}[{\cite[Lemma 4.15]{CHJ}}]\label{le1.4}  Let $m,l\in \R$. for $f\in H_l^m$, it holds that
\begin{equation}\label{sim1}
    \begin{aligned}
&\sum_{k=-1}^\infty2^{2kl}\|\tP_kf\|^2_{H^m}\sim_{m,l}\sum_{k=-1}^\infty2^{2kl}\|\cP_kf\|^2_{H^m}\sim_{m,l}\|f\|^2_{H^m_l},\\
&\sum_{j=-1}^\infty2^{2jm}\|\tF_jf\|^2_{L^2_l}\sim_{m,l}\sum_{j=-1}^\infty2^{2jm}\|\F_jf\|^2_{L^2_l}\sim_{m,l}\|f\|^2_{H^m_l},\\
&\sum_{k=-1}^\infty2^{2kl}\|\mP_{k}f\|^2_{H^m}+
  \sum_{j=-1}^\infty2^{2jm}\|\mF_{j}f\|^2_{L^2_l}\ls_{m,l}\|f\|^2_{H^m_l}.
    \end{aligned}
\end{equation}
As a corollary, we have
\begin{equation}\label{sim2}
    \begin{aligned}
&\sum_{j,k=-1}^\infty2^{2kl}2^{2mj}\|\cP_k\F_jf\|^2_{L^2}\sim_{m,l}\sum_{j,k=-1}^\infty2^{2kl}2^{2mj}\|\F_j\cP_kf\|^2_{L^2}\sim_{m,l}\|f\|^2_{H^m_l},\\&\sum_{j,k=-1}^\infty2^{2kl}2^{2mj}\|\mP_k\mF_jf\|^2_{L^2}+\sum_{j,k=-1}^\infty2^{2kl}2^{2mj}\|\mF_j\mP_kf\|^2_{L^2}\ls_{m,l}\|f\|^2_{H^m_l}.
    \end{aligned}
\end{equation}
\end{lemma}
\begin{remark}
	We emphasize that the above results also hold if we replace norm $H^m_l$ by $H^{m,\a}_l,\a\geq0$ and the proof is analogous.
\end{remark}

\subsection{Applications to the collision operator}\label{localcollision}
We use the dyadic decompositions to decompose the collision operator.

\subsubsection{Dyadic decomposition of the operator in the  phase space} We first localize it to the annulus in phase space. We set 
 \begin{equation}\label{DefPhi} 
 \Phi_k^\gamma(v):=
\left\{\begin{aligned} & |v|^\gamma \varphi(2^{-k}|v|), \quad\mbox{if}\quad k\ge0;\\
& |v|^\gamma \psi( |v|),\quad\mbox{if}\quad k=-1.\end{aligned}\right.
 \end{equation}
where for $k\geq0$, the function $\Phi_k^\gamma$ is smooth and supported away from the singularity, and for $k=-1$, the singularity remains. Then we derive that 
 \begin{equation}
     (Q(g, h), f )_{v}=\sum_{k=-1}^\infty ( Q_k(g, h), f )_v=\sum_{k=-1}^\infty\sum_{j=-1}^\infty ( Q_k(\mathcal{P}_jg, h), f )_v, 
 \end{equation}
where
\begin{equation*}
    Q_{k}(g, h):=\iint_{v_*\in \R^3,\sigma\in \SS^2} \Phi_k^\gamma(|v-v_*|)b(\cos\theta) (g'_*h'-g_*h)d\sigma \d v _*.
\end{equation*}
We have decomposition of the collision operator:
\begin{equation*}
    (Q_k(g,h),f)_v=\sum_{p=-1}^\infty\sum_{l=-1}^\infty (Q_k(\F_pg,\F_lh),f)_v.
\end{equation*}
It is not difficult to check that there exists an integer $N_0\in \N$ such that (see also  $(2.1)$ in \cite{he2018sharp})
\begin{equation}\label{ubdecom}
\begin{aligned}
( Q(g,h), f)_v &=\sum_{k\ge N_0-1}( Q_k(\mathcal{U}_{k-N_0} g, \tilde{\mathcal{P}}_kh), \tilde{\mathcal{P}}_kf )_v +
\sum_{j\ge k+N_0}( Q_k(\mathcal{P}_{j} g, \tilde{\mathcal{P}}_jh), \tilde{\mathcal{P}}_jf )_v\\
&\quad+\sum_{|j-k|\le N_0}( Q_k( \mathcal{P}_{j} g, \mathcal{U}_{k+N_0}h), \mathcal{U}_{k+N_0}f )_v.  
\end{aligned}
\end{equation}

\subsubsection{Dyadic decomposition of the operator in the  frequency space} Next, we decompose the collision operator in frequency space. Recall the Bobylev's formula
\begin{equation}\label{bobylev}
 \begin{aligned}
& (\widehat{ Q_k}(g, h), \widehat{f} \,)_\xi\\
&=\iint_{\sigma\in \SS^2, \eta,\xi\in \R^3} b\left(\frac{\xi}{|\xi|}\cdot \sigma\right)\big[ \widehat{ \Phi_k^\gamma } (\eta-\xi^{-})-\widehat{ \Phi_k^\gamma}(\eta)\big]\widehat{g}(\eta)\widehat{h}(\xi-\eta)\overline{\widehat{f}}(\xi)\d\sigma \d\eta \d\xi, 
 \end{aligned}
 \end{equation}
where $\xi^{\pm}:=\frac{\xi\pm|\xi|\sigma}{2}$. Then we have the further decomposition of the operator (see \cite{he2018sharp}):
\begin{equation*}
\begin{aligned}
&(Q_k(g,h),f)_v=\sum_{p=-1}^\infty\sum_{l=-1}^\infty (Q_k(\F_pg,\F_lh),f)_v\\
=&\sum_{l\leq p-N_0}\M_{k,p,l}^1+\sum_{l\geq p+N_0}\M_{k,p,l}^2+\sum_{|l-p|<N_0}\Big(\sum_{|m-p|\leq 2N_0}\M_{k,p,l,m}^3+\sum_{m<p- 2N_0}\M_{k,p,l,m}^4\Big),
\end{aligned} 
\end{equation*}
where 
\begin{equation*}
    \M_{k,p,l}^1(g,h,f):=\iint_{\R^6\times\SS^2}(\t \F_p\Phi^\gamma_k)(|v-v_*|)b(\cos\theta)\left( \F_pg\right)_*\left( \F_lh\right)\big[(\t \F_pf)'-(\t \F_pf)\big]\d \sigma \d v_*\d v,
\end{equation*}
\begin{equation*}
    \M_{k,p,l}^2(g,h,f):=\iint_{\R^6\times\SS^2}\Phi^\gamma_k(|v-v_*|)b(\cos\theta)\left( \F_pg\right)_*\left( \F_lh\right)\big[(\t \F_lf)'-(\t \F_lf)\big]\d \sigma \d v_*\d v,
\end{equation*}
\begin{equation*}
    \M_{k,p,l,m}^3(g,h,f):=\iint_{\R^6\times\SS^2}\Phi^\gamma_k(|v-v_*|)b(\cos\theta)\left( \F_pg\right)_*\left( \F_lh\right)\left[\left( \F_mf\right)'-\left( \F_mf\right)\right]\d \sigma \d v_*\d v,
\end{equation*}
\begin{equation*}
    \M_{k,p,l,m}^4(g,h,f):=\iint_{\R^6\times\SS^2}(\t \F_p\Phi^\gamma_k)(|v-v_*|)b(\cos\theta)\left( \F_pg\right)_*\left( \F_lh\right)\left[\left( \F_mf\right)'-\left( \F_mf\right)\right]\d \sigma \d v_*\d v.
\end{equation*}
By simple computation, we arrive at
\begin{equation}\label{eq m1234}
\begin{aligned}
(Q_k(g,h),f)_v=\&\sum_{l\leq p-N_0}\M_{k,p,l}^1+\sum_{l\geq p+N_0}\M_{k,p,l}^2+\sum_{p\geq -1}\M_{k,p}^3+\sum_{m< p-N_0}\M_{k,p,m}^4,
\end{aligned}    
\end{equation}
where 
\begin{equation*}
    \M_{k,p}^3:=\iint_{\R^6\times\SS^2}\Phi^\gamma_k(|v-v_*|)b(\cos\theta)( \F_pg)_*( \t\F_ph)\big[( \t\F_pf)'-( \t\F_pf)\big]\d \sigma \d v_*\d v,
\end{equation*}
\begin{equation*}
    \M_{k,p,m}^4:=\iint_{\R^6\times\SS^2}(\t \F_p\Phi^\gamma_k)(|v-v_*|)b(\cos\theta)\left( \F_pg\right)_*( \t\F_ph)\left[\left( \F_mf\right)'-\left( \F_mf\right)\right]\d \sigma \d v_*\d v.
\end{equation*}
Compared to the definition of \(\M_{k,l}^2\) in \cite{he2018sharp}, we do not sum over \(p \leq l - N_0\) here.

Using the same idea, for any fixed $j\geq-1$, one may derive that
\begin{equation}\label{2.2}
    \begin{aligned}
         &(\F_jQ_k(g,h),\F_jf)_v
=\sum_{|p-j|<2N_0}\sum_{p'\leq p+3N_0}(\F_jQ_k(\F_{p'}g,\F_ph),\F_jf)_v\\
 &
 +\sum_{p>j+2N_0}\sum_{|p-p'|\leq N_0}(\F_jQ_k(\F_{p'}g,\F_ph),\F_jf)_v
+\sum_{p<j-2N_0}\sum_{|m-j|\leq 2N_0}(\F_jQ_k(\F_mg,\F_ph),\F_jf)_v.
    \end{aligned}
\end{equation}

\section{Coercivity estimates and upper bound}\label{coercivityandupperbound}
 In this section, we give the coercivity estimate and upper bounds for the collision operator, and these results are of independent interest. 

\begin{proposition}[Coercivity]\label{prop:coercivity}
Let $\gamma$ and $s$ satisfy assumption \eqref{assumption}, $\delta,\lambda>0$, and suppose $g$ is a non-negative and smooth function verifying that
\begin{equation*}
    \|g\|_{L^1}>\delta,\quad \mbox{and}\quad \|g\|_{L^1_2}+\|g\|_{L\log L}<\lambda,
\end{equation*} 
Define 
\begin{equation}
    \mathcal{E}^\gamma_g(f):=\int_{\R^6\times\SS^2}B(|v-v_*|,\sigma)g_*(f'-f)^2d\sigma \d v _*\d v , 
\end{equation}
 and let $\mathbf{A}=0,1$. Then for sufficiently small $\eta>0$, there exist constants $C_1(\delta,\lambda,\eta^{-1}), C_2(\delta,\lambda), \\C_3(\delta,\lambda,\eta^{-1})$, $C_4(\delta,\lambda)$ and $C_5(\delta,\lambda)$ such that
 \begin{equation}\label{1E}
    \begin{aligned}
       \mathcal{E}^\gamma_g(f)\gs&~\mathbf{A}\Big[
C_1(\delta,\lambda,\eta^{-1})
\Bigl(\bigl\|(-\Delta_{\SS^2})^{s/2}f\bigr\|_{L^2_{\gamma/2}}^{2}
      +\|f\|_{H^s_{\gamma/2}}^{2}\Bigr)
-\eta\,C_2(\delta,\lambda)\,\|f\|_{L^2_{\gamma/2+s}}^{2}
\\
&-C_3(\delta,\lambda,\eta^{-1})
\|f\|_{L^2_{\gamma/2}}^{2}
\Big]+C_4(\lambda,\delta)\|f\|^2_{H^s_{\gamma/2}}-C_5(\lambda,\delta)\|f\|^2_{\gamma/2}.
    \end{aligned}
\end{equation}
\end{proposition}
\begin{proof}
The desired result can be obtained by \cite[Lemma 3.2 and Lemma 3.4]{he2018sharp}.

\end{proof}

 Next, we present the upper bound estimates for the collision operator. Compared with   \cite[Theorem 1.1 and Theorem 1.4]{he2018sharp}, the \(L^2\) norm is improved to \(H^{-\zeta}\) since we work in negative Sobolev spaces for equation \eqref{Boltzmann}.
\begin{proposition}[Upper bound]\label{prop upper}
Let $\gamma$ and $s$ satisfy assumption \eqref{assumption},  and $\zeta$ defined as
\begin{equation}\label{zeta1}
\zeta=\left\{\begin{array}{ll}
         &1,\quad -\frac{1}{2}<\gamma+2s< 0,\\ 
         \\
&\gamma+2s+\frac{3}{2}\big(=2+\frac{\gamma}{2}\big),\quad -1<\gamma+2s< -\frac{1}{2}.
    \end{array}\right. 
\end{equation}
The corresponding function space $H^{-\zeta,\a}_{\omega}$ is defined as
\begin{equation*}
H_{\omega}^{-\zeta,\a}=\left\{\begin{array}{ll}
         &H_{\omega}^{-1,\a},\quad -\frac{1}{2}<\gamma+2s< 0,\\   \\&H_{\omega}^{-1,1+\a},\quad \gamma+2s=-\frac{1}{2}, \\ \\
&H_{\omega}^{-(\gamma+2s+\f32),\a},\quad -1<\gamma+2s< -\frac{1}{2}.
    \end{array}\right.
\end{equation*}

\noindent $\bullet$ Let $w_1,w_2\in \R$, $a,b\in[0,2s]$ with $w_1+w_2=\gamma+2s$ and $a+b=2s$, $\a>\f12$. For smooth functions $g,h$ and $f$, we have 
\begin{equation}\label{upperbound1}
\begin{aligned}
  |(Q(g,h),f)_{v}|&\ls \big(\|g\|_{L^1_{\t w}}+\|g\|_{ H^{-\zeta,\a}_{-(\gamma+2s)}}\big)\|h\|_{H^{a}_{w_1}}\|f\|_{H^{b}_{w_2}}+\|g\|_{H_{w_1}^{a}}\|h\|_{H^{-\zeta,\a}_{-(\gamma+2s)}}\|f\|_{H^{b}_{w_2}}\\
    +& C_N(\|g\|_{H^{-\zeta,\a}}\|h\|_{H^a_{-N}}\|f\|_{H^b_{-N}}+\|g\|_{H^a}\|h\|_{H^{-\zeta,\a}_{-N}}\|f\|_{H^b_{-N}})
\end{aligned}
\end{equation}
for any $N\in\N$ and $\t w=\max\{-(\gamma+2s),\gamma+2s+(-\omega_1)^++(-\omega_2)^+\}$.

\medskip

\noindent $\bullet$ Let $w_1,w_2\in\R$,$a_1,b_1\in [0,s]$, $a,b\in[0,2s]$ with $w_1+w_2=\gamma+s$, $a+b=2s$ and $a_1+b_1=s,\a>\f12$. For smooth functions $g,h$ and $f$, we have 
\begin{equation}\label{upperbound2}
\begin{aligned}
   &|(Q(g,h),f)_{v}|\ls \|g\|_{H_{\gamma/2}^a}\|h\|_{H^{-\zeta,\a}_{-\gamma}}\|f\|_{H^b_{\gamma/2}}+\big(\|g\|_{L^1_{-\gamma+2s}}+\|g\|_{L^1_{\gamma+s+(-\omega_1)^{+}+(-\omega_2)^{+}}}+\|g\|_{ H^{-\zeta,\a}_{-\gamma}}\big)\\
\times&\Big[\big(\|(-\Delta_{\mathbb{S}^2})^{\frac a 2}h\|_{L^2_{\gamma/2}}+\|h\|_{H^a_{\gamma/2}}\big)\big((-\Delta_{\mathbb{S}^2})^{\frac b2}f\|_{L^2_{\gamma/2}}+\|f\|_{H^b_{\gamma/2}}\big)+\|h\|_{H^{a_1}_{w_1}}\|f\|_{H^{b_1}_{w_2}}\Big]\\
&+C_N(\|g\|_{H^{-\zeta,\a}}\|h\|_{H^a_{-N}}\|f\|_{H^b_{-N}}+\|g\|_{H^a}\|h\|_{H^{-\zeta,\a}_{-N}}\|f\|_{H^b_{-N}}),\quad \forall N\in\N.
\end{aligned}
\end{equation}
\end{proposition}
\begin{remark}
For the case \(\gamma + 2s = -\frac{1}{2}\), remember that the corresponding space is \(H^{-1, \a + 1}_\omega\). For simplicity, we will continue to use the notation \(H^{-\zeta, \a}_\omega\) throughout this paper. And \eqref{upperbound1} and \eqref{upperbound2} reflect the anisotropic coercive structure of the non-cutoff Boltzmann operator: the collision operator regularizes both in the radial Fourier variable and along angular directions.
\end{remark}

We will give the proofs in several steps and begin with \eqref{upperbound1}. Recall the decomposition \eqref{eq m1234}, for \(\M_{k,p,l}^1\) and \(\M_{k,p,l}^4,k\geq-1\), we have
\begin{lemma}[Estimates of $\M_{k,p,l}^1$ and $\M_{k,p,l}^4$]\label{M14}
 For any $k\geq 0$ and $N\in \mathbb{N}$, it holds that
\begin{equation}\label{ineq kgeq0 M1}
\begin{split}
    |\M_{k,p,l}^1|&\ls C_N2^{-Nk}2^{-Np}\|\F_p g\|_{L^2}\|\F_l h\|_{L^2}\|\t\F_pf\|_{L^2},
\end{split}
\end{equation}
\begin{equation}\label{ineq kgeq0 M4}
\begin{split}
    |\M_{k,p,m}^4|&\ls C_N2^{-Nk}2^{-Np}\|\F_p g\|_{L^2}\|\t \F_p h\|_{L^2}\|\t\F_mf\|_{L^2}.
\end{split}
\end{equation}
For $k=-1$, it holds that
\begin{equation}\label{ineq k=-1 M1}
|\M_{-1,p,l}^1|\ls (2^{-\tilde{\zeta} p}2^{2sl}+{2^{\frac 32 l}2^{-(\gamma+3)p}})\|\F_p g\|_{L^2}\|\t\F_l h\|_{L^2}\|\t\F_pf\|_{L^2},
\end{equation}
\begin{equation}\label{ineq k=-1 M4}
|\M_{-1,p,m}^4|\ls 2^{-\tilde{\zeta}  p}2^{2sm}\|\F_p g\|_{L^2}\|\t \F_p h\|_{L^2}\|\t\F_mf\|_{L^2},
\end{equation}
where
\begin{equation}\label{tzeta}
    \tilde{\zeta}=\left\{\begin{array}{ll}
         &1, \quad -\frac{1}{2}<\gamma+2s<0,\\
&1-\frac{\log_2p}{p}, \quad \gamma+2s=-\frac{1}{2},\\
&\gamma+2s+\frac{3}{2}, \quad -1<\gamma+2s<-\frac{1}{2}.
    \end{array}\right.
\end{equation}
\end{lemma}
With the convention that in the borderline case $\gamma+2s=-\f12$ we set $\t\zeta=1$ for $p\le1$; the values at finitely many indices only affect the constants.
\begin{proof}
For the case $k=-1$, i.e., \eqref{ineq k=-1 M1} and \eqref{ineq k=-1 M4}, we use the estimates in \cite[Lemma 2.1, Lemma 2.2]{he2018sharp}, where the parameters \(\eta_2\) and \(\tilde\zeta\) here satisfy \(\eta_2+\f12=\tilde\zeta\).

For the case $k\geq0$, we first give the proof of \eqref{ineq kgeq0 M1}. Recall the definition that
\begin{equation*}
    \begin{aligned}
      \M_{k,p,l}^1=&\iint_{\R^6\times\SS^2}(\t \F_p\Phi^\gamma_k)(|v-v_*|)b(\cos\theta)\left( \F_pg\right)_*\left( \F_lh\right)\big[(\t \F_pf)'-(\t \F_pf)\big]\d \sigma \d v_*\d v,\quad l\leq p-N_0.
    \end{aligned}
\end{equation*}
Using Taylor expansion
\begin{equation*}
    \begin{aligned}
        (\tF_pf)'-\tF_p f=(v'-v)\cdot(\nabla\tF_p f)(v)+\f12\int_0^1(1-\kappa)(v'-v)\otimes(v'-v):(\nabla^2\tF_p f)(\kappa(v))d\kappa,
    \end{aligned}
\end{equation*}
where $\kappa(v)=v+\kappa(v'-v)$, we have 
    \begin{equation*}
        \begin{aligned}
        &\M_{k,p,l}^1=\iint_{\R^6\times\SS^2}(\t \F_p\Phi^\gamma_k)(|v-v_*|)b(\cos\theta)\left( \F_pg\right)_*\left( \F_lh\right)\nabla\tF_pf(v)\cdot(v'-v)\d \sigma \d v_*\d v\\
        +&\int_0^1\iint_{\R^6\times\SS^2}(\t \F_p\Phi^\gamma_k)(|v-v_*|)b(\cos\theta)\left( \F_pg\right)_*\left( \F_lh\right)\nabla^2\tF_pf(\kappa(v)):(v'-v)\otimes(v'-v)\d \sigma \d v_*\d v\d \kappa\\
        &~~~~~~~~=:~Z_1+Z_2.
    \end{aligned}
    \end{equation*}
   
    For $Z_1$, noticing that $\cos\theta=\f{v-v_*}{|v-v_*|}\cdot\sigma$ and 
        \begin{equation}\label{symmetric}
        \int_{\SS^2}b\Big(\f{v-v_*}{|v-v_*|}\cdot\sigma\Big)(v-v')\chi(|v-v'|)d\sigma=\int_{\SS^2}b\Big(\f{v-v_*}{|v-v_*|}\cdot\sigma\Big)\f{1-(\f{v-v_*}{|v-v_*|}\cdot\sigma)}2\chi(|v-v'|)d\sigma \cdot(v-v_*)
    \end{equation}
    for any function $\chi$, then using the fact $b(\cos\theta)\sim \theta^{-2-2s}$, we have
    \begin{equation*}
        \Big|\int_{\SS^2}b(\cos\theta)(v'-v)d\sigma\Big|\ls \int_0^{\f\pi 2}\theta^{-2-2s}\sin^2\f\theta 2\sin\theta d\theta|v-v_*|\ls |v-v_*|.
    \end{equation*}
Thus by Cauchy-Schwarz inequality, we have
    \begin{equation*}
        \begin{aligned}
        |Z_1|\ls &~\big(\int_{\R^6}|v-v_*|^2|\tF_p(\Phi_k^{\gamma})(|v-v_*|)|^2|\nabla\tF_pf(v)|^2\d v\d v_*\big)^{\frac{1}{2}}\big(\int_{\R^6}|( \F_pg)(v_*)|^2| \F_lh(v)|^2\d v\d v_*\big)^{\frac{1}{2}}
        \\
        \ls&~\||\cdot|\tF_p(\Phi_k^\gamma)\|_{L^2}\|\F_p g\|_{L^2}\| \F_l h\|_{L^2}\|\nabla\tF_pf(v)\|_{L^2}.
    \end{aligned}
    \end{equation*}
By direct calculation, we have 
\begin{equation*}
    \widehat{\Phi^\gamma_k}(\xi)=2^{(\gamma+3)k}\widehat{\Phi^\gamma_0}(2^k\xi),
\end{equation*}
which yields 
\begin{equation}\label{tFpphigammak}
    \|\tF_p\Phi^\gamma_k\|_{L^2_2}\ls \|\widehat{\F_p\Phi^\gamma_k}\|_{H^2}\ls C_N2^{-Nk-(N+1)p}\|\vphi\|_{L^\infty}\|\Phi^\gamma_0\|_{L^{2}_{2N+2}}.
\end{equation}
From this and Lemma \ref{Bern}, we obtain 
\begin{equation}\label{Z1}
    |Z_1|\ls C_N 2^{-Nk}2^{-Np}\|\F_p g\|_{L^2}\| \F_l h\|_{L^2}2^{-p}\|\nabla\tF_pf(v)\|_{L^2}\ls C_N 2^{-Nk}2^{-Np}\|\F_p g\|_{L^2}\| \F_l h\|_{L^2}\|\tF_pf\|_{L^2}.
\end{equation}

    For $Z_2$, note that $|v-v'|=|v-v_*|\sin\f\theta 2$, by Cauchy-Schwarz inequality, we have
    \begin{equation*}
        \begin{aligned}
        |Z_2|\ls &~\Big(\int_0^1\iint_{\R^6\times\SS^2}|v-v_*|^4|\tF_p(\Phi_k^{\gamma})(|v-v_*|)|^2b(\cos\theta)\sin^2\f\theta 2| \F_lh(v)|^2d\sigma\d v\d v_*d\kappa\Big)^{\frac{1}{2}}\\
        &\times~\Big(\int_0^1\iint_{\R^6\times\SS^2}b(\cos\theta)\sin^2\f\theta 2 | \F_pg(v_*)|^2|\nabla^2\tF_pf(\kappa(v))|^2\d v\d v_*\d\sigma d\kappa\Big)^{\frac{1}{2}}.
    \end{aligned}
    \end{equation*}
For the second part, we follow the change of variables $(v_*,v)\rightarrow(v_*,u_1=\kappa(v))$. Thanks to the fact
\begin{equation}\label{Jacobi}
    \Big|\f{\partial \kappa(v)}{\partial v}\Big|=(1-\f \kappa 2)^2\{(1-\f\kappa 2)+\f\kappa 2\f{v-v_*}{|v-v_*|}\cdot\sigma\}>\f18,
\end{equation}
we derive that
\begin{equation*}
\begin{aligned}
   &\Big|\iint_{\R^6\times\SS^2}b(\cos\theta)\sin^2\f\theta 2 | \F_pg(v_*)|^2|\nabla^2\tF_pf(\kappa(v))|^2\d v\d v_*\d\sigma\Big|\\
   \ls~&\Big|\iint_{\R^6\times\SS^2}b(\cos\t\theta)\sin^2\f{\t\theta} 2 | \F_pg(v_*)|^2|\nabla^2\tF_pf(u_1)|^2\d u_1\d v_*\d\sigma\Big|,
\end{aligned}
\end{equation*}
where $\t\theta$ verifies $\cos\t\theta=\f{u_1-v_*}{|u_1-v_*|}\cdot\sigma$ satisfies $\theta/2\leq \t\theta\leq \theta$. Thus we can obtain
    \begin{equation}\label{Z2}
        \begin{aligned}
        |Z_2|\ls\||\cdot|^2\tF_p(\Phi_k^\gamma)\|_{L^2}\|\F_p g\|_{L^2}\| \F_l h\|_{L^2}\|\nabla^2\tF_pf\|_{L^2}\ls C_N 2^{-Nk}2^{-Np}\|\F_p g\|_{L^2}\| \F_l h\|_{L^2}\|\tF_pf(v)\|_{L^2},
    \end{aligned}
    \end{equation}
    where we use \eqref{tFpphigammak} in the last step. Combining \eqref{Z1} and \eqref{Z2}, we obtain \eqref{ineq kgeq0 M1}.
    
\medskip

 The proof of \(\mathcal{M}_{k,p,m}^4\) is similar to that of \(\mathcal{M}_{k,p,l}^1\) and  we do not repeat the argument. We finally remark that the \(\|\F_p g\|_{L^2}\) in \eqref{ineq kgeq0 M1} and \eqref{ineq kgeq0 M4} can be replaced by \(\|\F_p g\|_{L^1}\) since $\|\F_pg\|_{L^2}\ls 2^{\f32 p}\|\F_pg\|_{L^1}$ due to Lemma \ref{Bern}.
\end{proof}

For $\M_{k,p,l}^2$ and $\M_{k,p}^3,k\geq-1$, we have the following lemma. 
\begin{lemma}[Estimates of $\M_{k,p,l}^2$ and $\M_{k,p}^3$]\label{M23}
For any $k\geq0$ and $N\in\N$, it holds that
\begin{equation}\label{ineq kgeq0 M2}
\begin{aligned}
   \big |\sum_{p\leq l-N_0}\M_{k,p,l}^2\big|\ls &\min\Big\{2^{(\gamma+2s)k}2^{2sl}\|\cS_{l-N_0} g\|_{L^1},
    2^{(\gamma+\frac{3}{2})k}\sum_{p\leq l-N_0}(2^{2s(l-p)}\|\F_pg\|_{L^2})\Big\}\|\F_l h\|_{L^2}\|\t\F_lf\|_{L^2},
\end{aligned}   
\end{equation}
\begin{equation}\label{ineq kgeq0 M3}
\begin{split}
    |\M_{k,p}^3|\ls &\min\Big\{2^{(\gamma+2s)k}2^{2sp}\|\F_p g\|_{L^1},2^{(\gamma+\frac{3}{2})k}\|\F_pg\|_{L^2}\Big\}\|\t \F_p h\|_{L^2}\|\t\F_pf\|_{L^2}.
\end{split}
\end{equation}
For $k=-1$, it holds that
\begin{equation}\label{ineq k=-1 M2}
|\M_{-1,p,l}^2|\ls 2^{-\tilde{\zeta} p}2^{2sl}\|\F_p g\|_{L^2}\|\F_l h\|_{L^2}\|\t\F_lf\|_{L^2},
\end{equation}
\begin{equation}\label{ineq k=-1 M3}
|\M_{-1,p}^3|\ls 2^{-\tilde{\zeta}  p}2^{2sp}\|\F_p g\|_{L^2}\|\t \F_p h\|_{L^2}\|\t\F_pf\|_{L^2},
\end{equation}
where $\tilde{\zeta}$ is defined in \eqref{tzeta}.
\end{lemma}
\begin{remark}
    We remark that the \(L^2\) norms of function $g$ appearing in the upper bounds \eqref{ineq kgeq0 M2} and \eqref{ineq kgeq0 M3} are used only in the proof of uniqueness, one may see Section \ref{proveunique} for the details.
\end{remark}

\begin{proof}
{\bf Estimate of $\M_{-1,p}^3$:} We first focus on the case $k=-1$ and begin with the proof of \eqref{ineq k=-1 M3}. 









Let $\cos{\theta}=\frac{\xi}{|\xi|}\cdot\sigma$, then $|\xi^-|=|\xi|\sin\frac{{\theta}}{2}$. If $|\eta_*|\sim 2^p$ and $|\xi|\sim 2^p$, there exists a constant $\delta_0$, such that  $|\eta_*-\xi^-|\sim 2^p$ when $\theta\leq \delta_0$. We cut the kernel $b$ into two parts $b=b_{ \theta\leq \delta_0}+b_{ \theta> \delta_0}:=b_1+b_2$. Then from Bobylev's formula \eqref{bobylev}, we divide $\M^3_{-1,p}$ into three parts
\begin{equation*}
  \begin{aligned}
|\M_{-1,p}^3|\leq\& \Big|\iint_{\R^6\times\SS^2}\t \F_p(\Phi^\gamma_{-1})(|v-v_*|)b_1(\cos\theta)( \F_pg)_*( \t\F_ph)
\big[( \t\F_pf)'-( \t\F_pf)\big]\d \sigma \d v_*\d v\Big|\\\&+\sum_{l\leq p+N_0}\Big|\iint_{ \R^6\times\SS^2} \F_{l}(\Phi^\gamma_{-1})(|v-v_*|)b_2(\cos \theta)\left( \F_pg\right)_*( \t\F_ph)( \t\F_pf)'\d \sigma \d v_*\d v\Big|\\\&+\sum_{l\leq p+N_0}\Big|\iint_{ \R^6\times\SS^2} \F_l(\Phi^\gamma_{-1})(|v-v_*|)b_2(\cos\theta)( \F_pg)_*( \t\F_ph)( \t\F_pf)\d \sigma \d v_*\d v\Big|\\=:\&~I_1+I_2+I_3.   
\end{aligned}
\end{equation*}     
The estimate of $I_1$ is the same with that in $\M_{-1,p,l}^1$, which is , we have
\begin{equation}\label{I_1}
  |I_1|\ls (2^{-\tilde{\zeta} p}2^{2sp}+2^{-(\gamma+\frac 32)p})\|\cF_p g\|_{L^2}\|\t\cF_p h\|_{L^2}\|\t \cF_p f\|_{L^2}
\ls  2^{-\tilde\zeta p}2^{2sp}\|\cF_p g\|_{L^2}\|\t\cF_p h\|_{L^2}\|\t \cF_p f\|_{L^2},  
\end{equation}
where we use the fact that $\tilde{\zeta}\leq \gamma+2s+\frac 32$.

We next turn to the estimate of $I_2$. Noticing that $\theta\geq \delta_0$, by Cauchy-Schwarz inequality, it holds
\begin{equation*}
    \begin{aligned}
|I_2|\ls\& \sum_{l\leq p+N_0}\Big(\iint_{\R^6\times\SS^2}|\F_l(\Phi^\gamma_{-1})(|v-v_*|)|^2|\t\F_ph(v)|^2b_2(\cos\theta)\d\sigma\d v\d v_*\Big)^\frac{1}{2}\\\&
\times\Big(\iint_{\R^6\times\SS^2}|\F_pg(v_*)|^2|\t\F_pf(v')|^2b_2(\cos\theta)\d\sigma\d v\d v_*\Big)^\frac{1}{2}.
\end{aligned}
\end{equation*}
For the first part, using $\int_{\SS^2}b_2(\cos\theta)\d\sigma\ls 1$, we have
\begin{equation*}
    \begin{aligned}
    \Big(\iint_{\R^6\times\SS^2}|\F_l(\Phi^\gamma_{-1})(|v-v_*|)|^2|\t\F_ph(v)|^2b_2(\cos\theta)\d\sigma\d v\d v_*\Big)^\frac{1}{2}\ls\|\F_l(\Phi^\gamma_{-1})\|_{L^2}\|\t\F_ph\|_{L^2}.
\end{aligned}
\end{equation*}
For the second part, using change of variables $(v,v_*)\mapsto (u=v',v_*)$, $|\frac{\p v'}{\p v}|=\f18(1+\frac{v-v_*}{|v-v_*|}\cdot \sigma)>0$ and $\cos\f\theta 2=\frac{v'-v_*}{|v'-v_*|}\cdot \sigma$, we have 
\begin{equation*}
    \begin{aligned}
\Big(\iint_{\R^6\times\SS^2}|\F_pg(v_*)|^2|\t\F_pf(v')|^2b_2(\cos\theta)\d\sigma\d v\d v_*\Big)^\frac{1}{2}\ls\|\F_pg\|_{L^2}\|\t\F_pf\|_{L^2}.
\end{aligned}
\end{equation*}
Then we get
\begin{equation}\label{I_2a}
    |I_2|\ls\sum_{l\leq p+N_0}\|\F_l(\Phi^\gamma_{-1})\|_{L^2}\|\F_pg\|_{L^2}\|\t\F_ph\|_{L^2}\|\t\F_pf\|_{L^2}.
\end{equation}
Note that
\begin{equation*}
    \sum_{l\leq p+N_0}\|\F_l(\Phi^\gamma_{-1})\|_{L^2}=\sum_{l\leq p+N_0}\|\vphi_l(\widehat{\Phi^\gamma_{-1}})\|_{L^2}\ls \left\{\begin{array}{ll}
&2^{-(\gamma+\frac{3}{2})p},\quad \gamma+\frac{3}{2}<0,\\
&p,\quad \gamma+\frac{3}{2}=0,\\ \vspace{-0.3cm}\\
&1, \quad \gamma+\frac{3}{2}>0.
\end{array}\right.
\end{equation*}
Therefore
\begin{equation}\label{I_2}
 |I_2|\ls 2^{-\Lambda p}2^{2sp}\|\F_pg\|_{L^2}\|\t\F_ph\|_{L^2}\|\t\F_pf\|_{L^2},   
\end{equation}
with
\begin{equation*}
    \Lambda=\left\{\begin{array}{ll}
  &\gamma+\frac{3}{2}+2s,\quad \gamma+\frac{3}{2}<0,
  \\ &2s-\frac{\log_2 p}{p},\quad \gamma+\frac{3}{2}=0,\\ &2s, \quad 
  \gamma+\frac{3}{2}>0.
\end{array}\right.
\end{equation*}

For $I_3$, by the Cauchy-Schwarz inequality, we have 
\begin{equation*}
    \begin{aligned}
|I_3|\ls\& \sum_{l\leq p+N_0}\Big(\iint_{\R^6\times\SS^2}|\F_l(\Phi^\gamma_{-1})(|v-v_*|)|^2|\t\F_ph(v)|^2b_2(\cos\theta)\d\sigma\d v\d v_*\Big)^\frac{1}{2}\\\&
\times\Big(\iint_{\R^6\times\SS^2}|\F_pg(v_*)|^2|\t\F_pf(v)|^2b_2(\cos\theta)\d\sigma\d v\d v_*\Big)^\frac{1}{2}.
\end{aligned}
\end{equation*}
Then we can use the same argument as $I_2$ to get that
\begin{equation}\label{I_3}
  |I_3|\ls 2^{-\Lambda p}2^{2sp}\|\F_pg\|_{L^2}\|\t\F_ph\|_{L^2}\|\t\F_pf\|_{L^2}.   
\end{equation}

Combining the estimates \eqref{I_1},\eqref{I_2} and \eqref{I_3} and one may check that $\Lambda\geq\tilde{\zeta}$ when $\gamma+2s\leq 0$ and $\gamma+4s=1$.
Indeed, when $\gamma+\frac{3}{2}<0$, it is easy to see $\Lambda=\gamma+2s+\frac{3}{2}\geq\tilde{\zeta}$. When $\gamma+\frac{3}{2}\geq 0$, from $\gamma+4s=1$ and $\gamma+2s\leq 0$, we have $\frac{1}{2}<s\leq \frac{5}{8}$ which implies $\gamma+2s=1-2s\geq-\frac{1}{4}$, so that $\Lambda\geq2s-\frac{\log_2 p}{p}>1=\tilde{\zeta}$ for any $p\in\N$.  Then we obtain \eqref{ineq k=-1 M3}.

\medskip
{\bf Estimate of $\M_{-1,p,l}^2$:}
Next we give the proof of \eqref{ineq k=-1 M2}. Recall that
\begin{equation*}
    \M_{-1,p,l}^2=\iint_{\R^6\times\SS^2}\Phi^\gamma_{-1}(|v-v_*|)b(\cos\theta)\left( \F_pg\right)_*\left( \F_lh\right)\big[(\t \F_lf)'-(\t \F_lf)\big]\d \sigma \d v_*\d v.
\end{equation*}
By Bobylev's formula, we first split it into two parts:
\begin{equation*}
  \begin{aligned}
&\M_{-1,p,l}^2=\iint_{\R^6\times\SS^2}\Phi^\gamma_{-1}(|v-v_*|)b(\cos\theta)\left( \F_pg\right)_*\left( \F_lh\right)\big[(\t \F_lf)'-(\t \F_lf)\big]\d \sigma \d v_*\d v\\=&\iint_{\R^6\times\SS^2}b(\frac{\xi}{|\xi|}\cdot\sigma)\big[\widehat{\Phi_{-1}^\gamma}(\eta_*-\xi^{-})-\widehat{\Phi_{-1}^\gamma}(\eta_*)]\widehat{\F_pg}(\eta_*)\widehat{\F_lh}(\xi-\eta_*)\overline{\widehat{\t\F_lf}}(\xi)\d \sigma\d \eta_* \d \xi
\\
=&\int_{\sigma\in\SS^2}\iint_{|\xi^-|<\frac{1}{2}|\eta_*|}b(\frac{\xi}{|\xi|}\cdot\sigma)\big[\widehat{\Phi_{-1}^\gamma}(\eta_*-\xi^{-})-\widehat{\Phi_{-1}^\gamma}(\eta_*)]\widehat{\F_pg}(\eta_*)\widehat{\F_lh}(\xi-\eta_*)\overline{\widehat{\t\F_lf}}(\xi)\d \sigma\d \eta_* \d \xi
\\
&+\int_{\sigma\in\SS^2}\iint_{|\xi^-|\geq\frac{1}{2}|\eta_*|}b(\frac{\xi}{|\xi|}\cdot\sigma)\big[\widehat{\Phi_{-1}^\gamma}(\eta_*-\xi^{-})-\widehat{\Phi_{-1}^\gamma}(\eta_*)]\widehat{\F_pg}(\eta_*)\widehat{\F_lh}(\xi-\eta_*)\overline{\widehat{\t\F_lf}}(\xi)\d \sigma\d \eta_* \d \xi
\\
=:&~D_1+D_2.
\end{aligned}  
\end{equation*}

{\bf Estimate of $D_1$:} For $D_1$, using Taylor expansion
\begin{equation*}
    \begin{aligned}
    \widehat{\Phi_{-1}^\gamma}(\eta_*-\xi^{-})-\widehat{\Phi_{-1}^\gamma}(\eta_*)=\nabla\widehat{\Phi_{-1}^\gamma}(\eta_*) \cdot (-\xi^{-})
    +\int_0^1t\nabla^2\widehat{\Phi_{-1}^\gamma}(\eta_*-t\xi^{-}):\xi^{-}\otimes \xi^{-}\d t,
\end{aligned}
\end{equation*}
we have
\begin{equation*}
    \begin{aligned}
    |D_1&|\leq \Big|\int_{\sigma\in\SS^2}\iint_{|\xi^-|<\frac{1}{2}|\eta_*|}b(\frac{\xi}{|\xi|}\cdot\sigma)\big[
    \nabla\widehat{\Phi_{-1}^\gamma}(\eta_*) \cdot \xi^{-}
    ]\widehat{\F_pg}(\eta_*)\widehat{\F_lh}(\xi-\eta_*)\overline{\widehat{\t\F_lf}}(\xi)\d \sigma\d \eta_* \d \xi\Big|
    \\
    &+\Big|\int_0^1\int_{\sigma\in\SS^2}\iint_{|\xi^-|<\frac{1}{2}|\eta_*|}b(\frac{\xi}{|\xi|}\cdot\sigma)\big[
    t\nabla^2\widehat{\Phi_{-1}^\gamma}(\eta_*-t\xi^{-}):\xi^{-}\otimes \xi^{-}
    ]\times\widehat{\F_pg}(\eta_*)\widehat{\F_lh}(\xi-\eta_*)\overline{\widehat{\t\F_lf}}(\xi)\d \sigma\d \eta_* \d \xi\Big|
    \\
    &=:~D_{1,1}+D_{1,2}.
\end{aligned}
\end{equation*}
For $D_{1,1}$, on one hand, noticing that  in this integral region, we have $|\eta_*|\sim 2^p$ and $|\xi|\sim 2^l$, then $|\xi^-|<\frac{1}{2}|\eta_*|$ implies that $ \theta \ls 2^{p-l}$ since $|\xi^-|=|\xi|\sin \frac { \theta }2$. On the other hand, by symmetry, we have 
\begin{equation}\label{bcdotxi_-}
 \int_{\mathbb{S}^2}b(\frac{\xi}{|\xi|}\cdot\sigma)\xi^-\d\sigma=\int_{\mathbb{S}^2}b(\frac{\xi}{|\xi|}\cdot\sigma)(\xi^-\cdot\frac{\xi}{|\xi|})\frac{\xi}{|\xi|}\d\sigma.   
\end{equation}
Recall that $\xi^-=\frac{\xi-|\xi|\sigma}2$ which implies that $|\xi^-\cdot\frac{\xi}{|\xi|}|=|\xi|\sin^2\frac{ \theta}2$. Therefore, using Cauchy-Schwarz inequality, the fact $b(\cos\theta)\sin\theta\sim \theta^{-1-2s}$ and $|\partial_{\alpha}\widehat{\Phi^\gamma_{-1}}(\eta_*)|\ls\<\eta_*\>^{-(\gamma+3+|\alpha|)}$(see \eqref{parPhi-1}) for $|\alpha|\leq 2$, we have
\begin{equation*}
    \begin{aligned}
    D_{1,1}\ls&\Big(\int_0^{\frac{\pi}{2}}\mathbf{1}_{\{\theta\ls2^{p-l}\}}b(\cos{\hat\theta})\sin\hat\theta \sin^2{\frac{\hat\theta}{2}}\mathbb\d\hat\theta\Big)2^{-(\gamma+4)p}2^l\Big(\iint_{|\eta_*|\sim2^p}|\widehat{\t\F_lf}(\xi)|^2\d\xi\d\eta_*\Big)^{\frac{1}{2}}\\
    &\times\Big(\iint|\widehat{\F_pg}|(\eta_*)|^2|\widehat{\F_lh}(\xi-\eta_*)|^2\d\xi\d\eta_*\Big)^{\frac{1}{2}}\\
    \ls&~2^{(2s-2)(l-p)}2^{-(\gamma+\frac{5}{2})p}2^l\|\F_pg\|_{L^2}\|\F_lh\|_{L^2}\|\t\F_lf\|_{L^2}
    =2^{(2s-1)l}2^{-(\gamma+2s+\frac{1}{2})p}\|\F_pg\|_{L^2}\|\F_lh\|_{L^2}\|\t\F_lf\|_{L^2}.
\end{aligned}
\end{equation*}
For $D_{1,2}$, noticing that $|\eta_*-t\xi^-|\gs |\eta_*|$ for any $t\in(0,1)$ and $|\xi^-\otimes\xi^-|\ls |\xi|^2\sin^2\theta\ls2^{2l}\sin^2\theta$, we can derive that
\begin{equation*}
    \begin{aligned}
    D_{1,2}\ls&\Big(\int_0^{\frac{\pi}{2}}\mathbf{1}_{\{\theta\ls2^{p-l}\}}b(\cos{\theta})\sin\theta \sin^2{\frac{\theta}{2}}\mathbb\d\theta\Big)2^{-(\gamma+5)p}2^{2l}\Big(\iint_{|\eta_*|\sim2^p}|
    \widehat{\cF}_lf(\xi)|^2\d\xi\d\eta_*\Big)^{\frac{1}{2}}\\&\times\Big(\iint|\widehat{\F}_pg|(\eta_*)|^2|\widehat{\F_lh}(\xi-\eta_*)|^2\d\xi\d\eta_*\Big)^{\frac{1}{2}}\\
    \ls&~2^{(2s-2)(l-p)}2^{-(\gamma+\frac 72)p}2^{2l}\|\F_pg\|_{L^2}\|\F_lh\|_{L^2}\|\t\F_lf\|_{L^2}
    =2^{2sl}2^{-(\gamma+2s+\frac{3}{2})p}\|\F_pg\|_{L^2}\|\F_lh\|_{L^2}\|\t\F_lf\|_{L^2}.
\end{aligned}
\end{equation*}
Combining the estimates of $D_{1,1}$ and $D_{1,2}$ and note that $p\leq l$, we obtain 
\begin{equation}\label{D_1}
     |D_1|\ls 2^{2sl}2^{-(\gamma+2s+\frac{3}{2})p}\|\F_pg\|_{L^2}\|\F_lh\|_{L^2}\|\t\F_lf\|_{L^2}.   
\end{equation}

{\bf Estimate of $D_2$:} For $D_2$, we first notice that
\begin{equation*}
    \begin{aligned}
    |\xi^-|\geq\frac{1}{2}|\eta_*|\Longrightarrow \theta\gs\frac{|\eta_*|}{|\xi|}\quad\mbox{and}\quad\theta\gs\frac{|\eta_*-\xi^-|}{|\xi|}.
\end{aligned}
\end{equation*}
Then we have
\begin{equation*}
    \begin{aligned}
    D_2=&\int_{\sigma\in\SS^2}\iint_{|\xi^-|\geq\frac{1}{2}|\eta_*|}b(\frac{\xi}{|\xi|}\cdot\sigma)\big[\widehat{\Phi^\gamma_{-1}}(\eta_*-\xi^{-})-\widehat{\Phi^\gamma_{-1}}(\eta_*)]\widehat{\F_pg}(\eta_*)\widehat{\F_lh}(\xi-\eta_*)\overline{\widehat{\t\F_lf}}(\xi)\d \sigma\d \eta_* \d \xi\\
    \leq&\Big|\int_{\sigma\in\SS^2}\iint_{\R^6}\mathbf{1}_{\{\theta\gs\frac{|\eta_*-\xi^-|}{|\xi|}\}}b(\frac{\xi}{|\xi|}\cdot\sigma)\widehat{\Phi^\gamma_{-1}}(\eta_*-\xi^{-})\widehat{\F_pg}(\eta_*)\widehat{\F_lh}(\xi-\eta_*)\overline{\widehat{\t\F_lf}}(\xi)\d \sigma\d \eta_* \d \xi\Big|\\
    &+\Big|\int_{\sigma\in\SS^2}\iint_{\R^6}\mathbf{1}_{\{\theta\gs\frac{|\eta_*|}{|\xi|}\}}b(\frac{\xi}{|\xi|}\cdot\sigma)\widehat{\Phi^\gamma_{-1}}(\eta_*)\widehat{\F_pg}(\eta_*)\widehat{\F_lh}(\xi-\eta_*)\overline{\widehat{\t\F_lf}}(\xi)\d \sigma\d \eta_* \d \xi\Big|
    =:D_{2,1}+D_{2,2}.
\end{aligned}
\end{equation*}

For $D_{2,2}$, since $\theta\gs \frac{|\eta_*|}{|\xi|}\gs 2^{p-l}$, then by the fact $b(\cos\theta)\sin\theta\sim \theta^{-1-2s}$ and $|\widehat{\Phi^\gamma_{-1}}(\eta_*)|\ls\<\eta_*\>^{-(\gamma+3)}$, it is easy to get that
\begin{equation}\label{D_22}
  \begin{aligned}
    D_{2,2}\leq&\Big(\int_0^{\frac{\pi}{2}}\mathbf{1}_{\{ \theta\gs 2^{p-l}\}}b(\cos{\theta})\sin\theta\mathbb\d\theta\Big)2^{-(\gamma+3)p}\Big(\iint_{|\eta_*|\sim2^p}|\widehat{\t \F_lf}(\xi)|^2\d\xi\d\eta_*\Big)^{\frac{1}{2}}\\
    &\times\Big(\iint|\widehat{\F}_pg(\eta_*)|^2|\widehat{\F_lh}(\xi-\eta_*)|^2\d\xi\d\eta_*\Big)^{\frac{1}{2}}\\
    \ls&~2^{2sl}2^{-(\gamma+2s+\frac{3}{2})p}\|\F_pg\|_{L^2}\|\F_lh\|_{L^2}\|\t\F_lf\|_{L^2}.
\end{aligned}  
\end{equation}

For $D_{2,1}$, we separate it again to get that
\begin{equation*}
    \begin{aligned}
    D_{2,1}
    \leq&\Big|\int_{\sigma\in\SS^2}\iint_{|\eta_*-\xi^-|\geq\frac{1}{2}|\eta_*|}\mathbf{1}_{\{\theta\gs\frac{|\eta_*-\xi^-|}{|\xi|}\}}b(\frac{\xi}{|\xi|}\cdot\sigma)\widehat{\Phi^\gamma_{-1}}(\eta_*-\xi^{-})\widehat{\F_pg}(\eta_*)\widehat{\F_lh}(\xi-\eta_*)\overline{\widehat{\t\F_lf}}(\xi)\d \sigma\d \eta_* \d \xi\Big|\\
    +&\Big|\int_{\sigma\in\SS^2}\iint_{|\eta_*-\xi^-|<\frac{1}{2}|\eta_*|}\mathbf{1}_{\{\theta\gs\frac{|\eta_*-\xi^-|}{|\xi|}\}}b(\frac{\xi}{|\xi|}\cdot\sigma)\widehat{\Phi^\gamma_{-1}}(\eta_*-\xi^{-})\widehat{\F_pg}(\eta_*)\widehat{\F_lh}(\xi-\eta_*)\overline{\widehat{\t\F_lf}}(\xi)\d \sigma\d \eta_* \d \xi\Big|\\
    =:&~D_{2,1,1}+D_{2,1,2}.
\end{aligned}
\end{equation*}
For $D_{2,1,1}$, noticing that $|\eta_*-\xi^-|\geq \frac12|\eta_*|\gs 2^p$. Then similar to the estimate of $D_{2,2}$, we can deduce that
\begin{equation}\label{D_211}
  \begin{aligned}
    D_{2,1,1}\leq2^{2sl}2^{-(\gamma+2s+\frac{3}{2})p}\|\F_pg\|_{L^2}\|\F_lh\|_{L^2}\|\t\F_lf\|_{L^2}.
\end{aligned}  
\end{equation}
For $D_{2,1,2}$, on one hand, we have $\f12|\eta_*|>|\eta_*-\xi^-|>|\eta_*|-|\xi^-|$, which implies that  $\theta\gs \f{|\xi^-|}{|\xi|}\gs \f{|\eta_*|}{|\xi|}\sim 2^{p-l}$. On the other hand, since $\f{\xi}{|\xi|}\cdot\sigma=\cos\theta$ and $\cos\f{\theta}2=\f{\xi^+}{|\xi|}\cdot\sigma$, we have $b(\f{\xi}{|\xi|}\cdot\sigma)\sim b(\f{\xi^+}{|\xi^+|}\cdot\sigma)$. Then by Cauchy-Schwarz inequality and the change of variables $\eta_*-\xi^-\rightarrow \t \eta_*$(thus $\xi-\eta_*\rightarrow\xi^+-\t\eta_*$ and $|\t \eta_*|\ls 2^p$), we have
\begin{equation*}
    \begin{aligned}
    D_{2,1,2}\ls&\Big(\int_{\sigma\in\SS^2}\iint \mathbf{1}_{\theta\gs 2^{p-l}} b^\kappa(\frac{\xi}{|\xi|}\cdot\sigma)|\widehat{\F_pg}(\eta_*)|^2|\overline{\widehat{\t\F_lf}}(\xi)|^2\d \sigma\d \eta_* \d \xi\Big)^{\frac{1}{2}}\\&\times
    \Big(\int_{\sigma\in\SS^2}\iint_{|\xi^+|\sim 2^l}\mathbf{1}_{\theta\gs \f{|\t \eta_*|}{|\xi^+|}} b^{2-\kappa}(\frac{\xi^+}{|\xi^+|}\cdot\sigma)|\widehat{\Phi^\gamma_{-1}}(\t \eta_*)|^2|\widehat{\F_lh}(\xi^+-\t \eta_*)|^2\d \sigma\d \t\eta_* \d \xi^+\Big)^{\frac{1}{2}}=:I_1\times I_2.
\end{aligned}
\end{equation*}
Here $\kappa\in(0,2)$ and we also use the fact that $|\frac{\partial(\eta_*,\xi)}{\partial(\t \eta_*,\xi^+)}|\sim |\f{\partial \xi}{\partial \xi^+}|=\frac{8}{(1+\cos\theta)}\sim 1$ and $|\xi^+|\sim |\xi|\sim 2^l$. 

If $\gamma< -\f32$, let $\delta=-(\gamma+\f32)/2>0$ and we choose $\kappa=\f{2+4s-\delta}{2+2s}$. Then for the first integral, we have
\begin{equation*}
\begin{aligned}
      (I_1)^2\ls&\int_{\theta \gs 2^{p-l}} \theta^{\delta-2-4s}\sin\theta \d\theta\iint  |\widehat{\F_pg}(\eta_*)|^2|\overline{\widehat{\t\F_lf}}(\xi)|^2\d \eta_* \d \xi\ls 2^{(4s-\delta)(l-p)}\|\F_pg\|^2_{L^2}\|\t\F_l f\|^2_{L^2}.
\end{aligned}
\end{equation*}
For the second integral, we have
\begin{equation*}
\begin{aligned}
       (I_2)^2&\ls \iint_{|\t\eta_*|\ls 2^p,|\xi^+|\sim2^l}(\int_{\theta\gs \f{|\t\eta_*|}{|\xi^+|}}\theta^{-2-\delta}\sin\theta \d\theta)  \<\t\eta_*\>^{-2(3+\gamma)}\widehat{\F_lh}(\xi^+-\t \eta_*)|^2\d \t\eta_* \d \xi^+\\
    &\ls \iint_{|\t\eta_*|\ls 2^p,|\xi^+|\sim2^l} \<\t\eta_*\>^{-2(3+\gamma)}|\t\eta_*|^{-\delta}|\xi^+|^{\delta}\widehat{\F_lh}(\xi^+-\t \eta_*)|^2\d \t\eta_* \d \xi^+\ls 2^{-(2\gamma+3+\delta) p}2^{\delta l}\|\F_lh\|^2_{L^2},
\end{aligned}
\end{equation*}
where we use the fact that $-2(3+\gamma)-\delta=-3-\f 32(\f32+\gamma)>-3$. Thus, we can get
\begin{equation*}
\begin{aligned}
      D_{2,1,2}\ls I_1\times I_2\ls 2^{2s l}2^{-(\gamma+2s+\f32)p}\|\F_p g\|_{L^2}\|\F_lh\|_{L^2}\|\t\F_l f\|_{L^2}.
\end{aligned}
\end{equation*}

If $\gamma\geq -\f32$, then we have $2s=\frac{1-\gamma}{2}>1$, we choose $\kappa=\f 4{2+2s}$. Then for the first integral, we have
\begin{equation*}
\begin{aligned}
      (I_1)^2\ls&\int_{\theta \gs 2^{p-l}} \theta^{-4}\sin\theta\d\theta \iint  |\widehat{\F_pg}(\eta_*)|^2|\overline{\widehat{\t\F_lf}}(\xi)|^2\d \eta_* \d \xi\ls 2^{2(l-p)}\|\F_pg\|^2_{L^2}\|\t\F_l f\|^2_{L^2}.
\end{aligned}
\end{equation*}
For the second integral , we have
\begin{equation*}
\begin{aligned}
       (I_2)^2&\ls (\int_{\theta\gs \f{|\t\eta_*|}{|\xi^+|}}\theta^{-4s}\sin\theta \d\theta) \iint_{|\t\eta_*|\ls 2^p,|\xi^+|\sim2^l} \<\t\eta_*\>^{-2(3+\gamma)}\widehat{\F_lh}(\xi^+-\t \eta_*)|^2\d \t\eta_* \d \xi^+\\
    &\ls \iint_{|\t\eta_*|\ls 2^p,|\xi^+|\sim2^l} \<\t\eta_*\>^{-2(3+\gamma)}|\t\eta_*|^{-4s+2}|\xi^+|^{4s-2}\widehat{\F_lh}(\xi^+-\t \eta_*)|^2\d \t\eta_* \d \xi^+\ls 2^{(4s-2) l}\|\F_lh\|^2_{L^2},
\end{aligned}
\end{equation*}
where we use the fact that $-2(3+\gamma)-4s+2=-3-(\gamma+2)<-3$, so that 
\begin{equation*}
    \int_{|\tilde\eta_*|\ls2^p}\<\tilde{\eta}_*\>^{-2(3+\gamma)}|\tilde{\eta}_*|^{-4s+2}\d\tilde{\eta}_*\ls\int_{0}^{\infty}\frac{r^{4-4s}}{(1+r^2)^{\gamma+3}}\d r\ls1.
\end{equation*}
Noticing that when $\gamma\geq-\f32$, we have $\gamma+2s\geq -\f14$, thus $\tilde{\zeta} \leq 1$. Then we can get
\begin{equation}\label{D212}
\begin{aligned}
      D_{2,1,2}\ls I_1\times I_2\ls 2^{-p}2^{2s l}\|\F_p g\|_{L^2}\|\F_lh\|_{L^2}\|\t\F_l f\|_{L^2}\leq 2^{-\tilde{\zeta} p}2^{2s l}\|\F_p g\|_{L^2}\|\F_lh\|_{L^2}\|\t\F_l f\|_{L^2}.
\end{aligned}
\end{equation}

Combining \eqref{D_211}, \eqref{D212} and \eqref{D_22} and recall that $\tilde{\zeta}\leq \gamma+2s+\f32$, we get that
\begin{equation*}
\begin{aligned}
      D_{2}\ls 2^{-\t \zeta p}2^{2s l}\|\F_p g\|_{L^2}\|\F_lh\|_{L^2}\|\t\F_l f\|_{L^2}.
\end{aligned}
\end{equation*}
Combining the above with  \eqref{D_1} yields the desired result \eqref{ineq k=-1 M2}.

\medskip

{\bf Estimate of $\M_{k,p,l}^{2}$ and $\M_{k,p}^3,k\geq 0$:}
For \(k\ge 0\), the first minimum parts of \eqref{ineq kgeq0 M2} and \eqref{ineq kgeq0 M3} are given by \cite[Lemma 2.3]{he2018sharp}.  We remark that the term \(\mathcal{M}_{k,l}^2\) in \cite{he2018sharp} corresponds to \(\sum_{p\le l-N_0}\mathcal{M}_{k,p,l}^2\) here. 

Now we focus on the $L^2$ norm estimate , which is the second minimum parts of \eqref{ineq kgeq0 M2} and \eqref{ineq kgeq0 M3}. For $\M_{k,p}^3$, we have
\begin{equation*}
    \begin{aligned}
        |\M_{k,p}^3|\leq\& \Big|\iint_{\R^6\times\SS^2}\t \F_p(\Phi^\gamma_{k})(|v-v_*|)b_1(\cos\theta)( \F_pg)_*( \t\F_ph)
\big[( \t\F_pf)'-( \t\F_pf)\big]\d \sigma \d v_*\d v\Big|\\\&+\Big|\iint_{ \R^6\times\SS^2} \Phi^\gamma_{k}(|v-v_*|)b_2(\cos \theta)\left( \F_pg\right)_*( \t\F_ph)( \t\F_pf)'\d \sigma \d v_*\d v\Big|\\\&+\Big|\iint_{ \R^6\times\SS^2} \Phi^\gamma_{k}(|v-v_*|)b_2(\cos\theta)( \F_pg)_*( \t\F_ph)( \t\F_pf)\d \sigma \d v_*\d v\Big|=:\t I_1+\t I_2+\t I_3.
    \end{aligned}
\end{equation*}
For $\t I_1$, we can follow the same steps in Lemma \eqref{M14} for $k\geq 0$, then we have
\begin{equation*}
    \begin{aligned}
        \t I_1\ls C_N2^{-Nk}2^{-Np}\|\F_p g\|_{L^2}\|\t\F_p h\|_{L^2}\|\t\F_pf\|_{L^2}.
    \end{aligned}
\end{equation*}
For $\t I_2$ and $\t I_3$, we can follow the same steps in the proof of \eqref{I_2a}, and change $\sum_{l\leq p+N_0}\F_{l}(\Phi_{-1}^{\gamma})$ into $\Phi_{k}^{\gamma}$, then we have
\begin{equation*}
    \t I_2+\t I_3\ls\|\Phi_{k}^{\gamma}\|_{L^2}\|\F_pg\|_{L^2}\|\t\F_ph\|_{L^2}\|\t\F_pf\|_{L^2}.
\end{equation*}
Combining $\t I_1$, $\t I_2$, $\t I_3$ and using $\|\Phi_k^{\gamma}\|_{L^2}\ls2^{(\gamma+\frac{3}{2})k}$, we can get 
\begin{equation*}
    |\M_{k,p}^3|\ls2^{(\gamma+\frac{3}{2})k}\|\F_pg\|_{L^2}\|\t\F_ph\|_{L^2}\|\t\F_pf\|_{L^2},
\end{equation*}
which completes \eqref{ineq kgeq0 M3}.

For $\M_{k,p,l}^{2}$, 
we use the separation $b=b_{\theta\leq \delta_0 2^{p-l}}+b_{\theta> \delta_0 2^{p-l}}=\t b_1+\t b_2$ with
 suitably small $\delta_0$. When $\theta\leq \delta_0 2^{p-l}$, $\eta_*\sim2^p$ and $\xi\sim 2^p$ means we have \(|\eta_*-\xi^-|\sim |\eta_*|\sim 2^p\), so that 
 \begin{equation*}
     \begin{aligned}
         |\M_{k,p,l}^2|\leq\& \Big|\iint_{\R^6\times\SS^2}\t \F_p(\Phi^\gamma_{k})(|v-v_*|)\t b_1(\cos\theta)( \F_pg)_*( \F_lh)
\big[( \t\F_lf)'-( \t\F_lf)\big]\d \sigma \d v_*\d v\Big|\\\&+\Big|\iint_{ \R^6\times\SS^2} \Phi^\gamma_{k}(|v-v_*|)\t b_2(\cos \theta)\left( \F_pg\right)_*( \t\F_ph)( \t\F_pf)'\d \sigma \d v_*\d v\Big|\\\&+\Big|\iint_{ \R^6\times\SS^2} \Phi^\gamma_{k}(|v-v_*|)\t b_2(\cos\theta)( \F_pg)_*( \t\F_ph)( \t\F_pf)\d \sigma \d v_*\d v\Big|=:\hat I_1+\hat I_2+\hat I_3.
     \end{aligned}
 \end{equation*}
For $\hat I_1$, noticing that $\int_{\SS^2} b_{{\theta}\leq \delta_0 2^{p-l}} \sin^2\f{\theta}2 d\sigma \ls 2^{(2-2s)(p-l)},$ we can use the same method in \eqref{M14} and get the bound
\begin{equation*}
    \hat I_1\ls C_N2^{-Nk}2^{-Np}2^{(2s-2)(l-p)}\|\F_p g\|_{L^2}\|\F_lh\|_{L^2}\|\tF_l f\|_{L^2}.
\end{equation*}
For $\hat I_2$ and $\hat I_3$, using 
    $\int_{\SS^2} b_{\hat{\theta}> \delta_0 2^{p-l}}d\sigma \ls 2^{2s(l-p)},$
and the same argument as the estimate of $\t I_2$, we can obtain the bound
\begin{equation*}
\hat I_2 +\hat I_3\ls 2^{(\gamma+\f32)k}2^{2s(l-p)}\|\F_p g\|_{L^2}\|\F_lh\|_{L^2}\|\tF_l f\|_{L^2}.
\end{equation*}
Therefore, we obtain \eqref{ineq kgeq0 M2}.

We complete the proof of this lemma.
\end{proof}

Now we are ready to prove Proposition \ref{prop upper} and we begin with \eqref{upperbound1}.
\begin{proof}[Proof of Proposition \ref{prop upper} \eqref{upperbound1}]

For the case $k\geq0$, let $a,b\in\R_{\geq 0}$ and $a+b=2s$. Then, due to \eqref{ineq kgeq0 M1} in Lemma \ref{M14} and note that $\|\F_pg\|_{L^2}\ls 2^{\f32p}\|\F_p g\|_{L^1}$, we have 
\begin{equation*}
    \begin{aligned}
        &\sum_{l\leq p-N_0}|\M_{k,l,p}^1|
        \ls\sum_{l\leq p-N_0} C_N2^{-Nk-Np}\|\F_p g\|_{L^1}\|\F_l h\|_{L^2}\|\t\F_pf\|_{L^2} 
        \ls C_N2^{(\gamma+2s)k}\sum_{l\leq p-N_0}\|\F_pg\|_{L^1}\\
        &\times2^{-l}2^{al}\|\F_lh\|_{L^2}2^{-p}2^{bp}\|\t\F_pf\|_{L^2}
        \ls C_N2^{(\gamma+2s)k}\|g\|_{L^1}(\sum_{l=-1}^\infty 2^{2al}\|\F_l h\|^2_{L^2})^{\frac12}(\sum_{p=-1}^\infty 2^{2bl}\|\t\F_p f\|^2_{L^2})^{\frac12}\\
        &\ls C_N2^{(\gamma+2s)k}\|g\|_{L^1}\|h\|_{H^{a}}\|f\|_{H^{b}},
    \end{aligned}
\end{equation*}
where we choose suitably large $N$ and use the fact that $l\leq p$, $\|\F_p g\|_{L^1}\ls \|g\|_{L^1}$(see \eqref{a001}) for any $p\geq -1$ and Lemma \ref{le1.4}. 

Similarly, for $\M^4_{k,p,m}$, it follows from \eqref{ineq kgeq0 M4} that
\begin{equation*}
    \begin{aligned}
        &\sum_{m\leq p-N_0}|\M_{k,p,m}^4|
        \ls\sum_{m\leq p-N_0}C_N2^{-Nk-Np}\|\F_p g\|_{L^1}\|\t \F_p h\|_{L^2}\|\t\F_mf\|_{L^2}\ls C_N2^{(\gamma+2s)k}\sum_{m\leq p-N_0}\|\F_p g\|_{L^1}\\
        &\times 2^{-p}2^{ap}\|\t \F_p h\|_{L^2}2^{-m}2^{bm}\|\t\F_mf\|_{L^2}
    \ls C_N2^{(\gamma+2s)k}\|g\|_{L^1}(\sum_{l=-1}^\infty 2^{2ap}\|\t\F_ph\|^2_{L^2})^{\frac12}(\sum_{m=-1}^\infty 2^{2bm}\|\t\F_m f\|^2_{L^2})^{\frac12}\\
    &\ls C_N2^{(\gamma+2s)k}\|g\|_{L^1}\|h\|_{H^{a}}\|f\|_{H^{b}}.
    \end{aligned}
\end{equation*}

For $\M_{k,p,l}^2$ and $\M_{k,p}^3$ , thanks to  \eqref{ineq kgeq0 M2} and \eqref{ineq kgeq0 M3} in Lemma \ref{M23}, we have
\begin{equation*}
    \begin{aligned}
        &\sum_{l\geq -1}\big|\sum_{p\leq l-N_0}\M_{k,p,l}^2\big|+\sum_{p\geq -1}|\M_{k,p}^3|
        \ls\sum_{l\geq-1}2^{(\gamma+2s)k}2^{2sl}\|\cS_{l-N_0} g\|_{L^1}\|\F_l h\|_{L^2}\|\t\F_lf\|_{L^2}\\
        &+\sum_{p\geq -1}2^{(\gamma+2s)k}2^{2sp}\|\F_p g\|_{L^1}\|\t \F_p h\|_{L^2}\|\t\F_pf\|_{L^2}
        \ls~ 2^{(\gamma+2s)k}\|g\|_{L^1}\Big(\big(\sum_{l\geq-1}2^{2al}\|\F_l h\|^2_{L^2}\big)^{\frac12}\\
        &\times\big(\sum_{l\geq-1}2^{2bl}\|\t\F_l f\|^2_{L^2}\big)^{\frac12}+\big(\sum_{p\geq-1}2^{2ap}\|\F_p h\|^2_{L^2}\big)^{\frac12}\big(\sum_{p\geq-1}2^{2bp}\|\t\F_p f\|^2_{L^2}\big )^{\frac12}\Big)
        \ls2^{(\gamma+2s)k}\|g\|_{L^1}\|h\|_{H^{a}}\|f\|_{H^{b}},
    \end{aligned}
\end{equation*}
where we also use the fact $\|\cS_{l-N_0} g\|_{L^1}\ls \|g\|_{L^1}$ for any $l-N_0\geq -1$(see \eqref{a001}).

Therefore, plug the above upper bounds into \eqref{eq m1234}, we obtain 
\begin{equation}\label{Qk1}
|(Q_k(g,h),f)_v|\ls 2^{(\gamma+2s)k}\|g\|_{L^1}\|h\|_{H^{a}}\|f\|_{H^{b}}, \quad k\geq 0,
\end{equation}
with $a,b\in\R_{\geq 0}$ and $a+b=2s$.

\medskip

For the case $k=-1$, let $a,b\in\R_{\geq 0}$ with $a+b=2s$, by \eqref{ineq k=-1 M1} and \eqref{ineq k=-1 M4} in Lemma \ref{M14}, we have 
\begin{equation}\label{M14k=-1}
\begin{aligned}
\sum_{l\leq p-N_0}|\M_{-1,l,p}^1|+\sum_{m<p-N_0}|\M^4_{-1,p,m}|
\ls \sum_{l\leq p-N_0}(2^{-\tilde{\zeta} p}2^{2sl}+{2^{\frac 32 l}2^{-(\gamma+3)p}})\|\F_p g\|_{L^2}\|\t\F_l h\|_{L^2}\|\t\F_pf\|_{L^2}\\
+\sum_{m<p-N_0}2^{-\tilde{\zeta}  p}2^{2sm}\|\F_p g\|_{L^2}\|\t \F_p h\|_{L^2}\|\t\F_mf\|_{L^2}.
\end{aligned}  
\end{equation}
For the first term on the right hand side of \eqref{M14k=-1}, let $\a>\f12$, we can bounded it by
\begin{equation*}
\begin{aligned}
&\sum_{ p\geq-1}2^{-\tilde{\zeta} p}\|\F_p g\|_{L^2}2^{bp}\|\tF_p f\|_{L^2}\Big(\sum_{l\leq p-N_0}2^{al}\|\t\F_l h\|_{L^2}2^{b(l-p)}\Big)\\
\ls&\sum_{ p\geq-1}2^{-\tilde{\zeta} p}\|\F_p g\|_{L^2}2^{bp}\|\tF_p f\|_{L^2}\big(\sum_{l\leq p-N_0}2^{2al}\|\t\F_l h\|^2_{L^2}\big)^{\frac12}\big(\sum_{l\leq p-N_0} 2^{2b(l-p)}\big)^{\frac12}\\
\ls& \big(\sum_{ p\geq-1}p^{2(\a+\mathbf{1}_{\gamma+2s=-\f12})}2^{-2\zeta p}\|\F_p g\|^2_{L^2}\big)^{\f12}\big(\sum_{ p\geq-1}2^{2bp}\|\tF_p f\|^2_{L^2}\big)^{\f12}\|h\|_{H^a}
\ls  \|g\|_{H^{-\zeta,\a}}\|h\|_{H^a}\|f\|_{H^b},
\end{aligned}  
\end{equation*}
where we use $\sum_{l\leq p-N_0} 2^{2b(l-p)}\ls1$ when $b>0$; $\sum_{l\leq p-N_0} 2^{2b(l-p)}\ls p$ when $b=0$.

For the third term on the right hand side of \eqref{M14k=-1}, we can handle it by the same argument as above to get the same upper bound. While for the second term, let $\a>\f12$, for $\lambda_1,\lambda_2,\mu$ satisfying $\lambda_1+\lambda_2\geq -(\gamma+3),\mu\leq \frac{3}{2}$ and $\lambda_1+\lambda_2+\mu= -(\gamma+\frac{3}{2})$, we can bound it by
\begin{equation*}
\begin{aligned}
\sum_{l\leq p-N_0}2^{\frac{3}{2}l}2^{-(\gamma+3)p}\|\F_p g\|_{L^2}\|\t\F_l h\|_{L^2}\|\t\F_pf\|_{L^2}
&\ls \sum_{p\geq-1}2^{\lambda_1p}\|\F_pg\|_{L^2}2^{\lambda_2p}\|\t\F_p f\|_{L^2}\Big(\sum_{l\leq p-N_0}l^{2\a}2^{2\mu l}\|\t\F_l h\|_{L^2}\Big)^{\f12}\\&\times\Big(\sum_{l\leq p-N_0}l^{-2\a}\Big)^{\f12}2^{(\lambda_1+\lambda_2+\gamma+3)(l-p)}
\ls \|g\|_{H^{\lambda_1}}\|h\|_{H^{\mu,\a}}\|f\|_{H^{\lambda_2}}.
\end{aligned}  
\end{equation*}

Therefore, the above estimates imply that
\begin{equation*}
    \sum_{l\leq p-N_0}|\M_{-1,l,p}^1|+\sum_{m<p-N_0}|\M^4_{-1,p,m}|\ls\|g\|_{H^{-\zeta,\a}}\|h\|_{H^a}\|f\|_{H^b}+\|g\|_{H^{\lambda_1}}\|h\|_{H^{\mu,\a}}\|f\|_{H^{\lambda_2}},
\end{equation*}
with $a,b\in\R_{\geq 0}$, $a+b=2s$, $\a>\f12$ and $\lambda_1,\lambda_2,\mu$ satisfying $\lambda_1+\lambda_2\geq -(\gamma+3)$ , $\mu\leq \frac{3}{2}$ and $\lambda_1+\lambda_2+\mu= -(\gamma+\frac{3}{2})$. 

\medskip

For $\M_{-1,p,l}^2$ and $\M_{-1,p}^3$, thanks to  \eqref{ineq k=-1 M2} and \eqref{ineq k=-1 M3} in Lemma \ref{M23}, we have
\begin{equation}\label{M23k=-1}
\begin{aligned}
&\sum_{l\leq p-N_0}|\M_{-1,p,l}^2|+\sum_{p\geq-1}|\M^3_{-1,p}|\\
\ls& \sum_{l\leq p-N_0}2^{-\tilde{\zeta} p}2^{2sl}\|\F_p g\|_{L^2}\|\F_l h\|_{L^2}\|\t\F_lf\|_{L^2}+\sum_{p\geq-1}2^{-\tilde{\zeta}  p}2^{2sp}\|\F_p g\|_{L^2}\|\t \F_p h\|_{L^2}\|\t\F_pf\|_{L^2}.
\end{aligned}  
\end{equation}
Noticing that $2^{-\tilde{\zeta} p}\|\F_p g\|_{L^2}\ls \|g\|_{H^{-\zeta,\a}}$ for any $p\geq -1,\a>\f12$, 
then it is easy to see that the second term can be bounded by
\begin{equation*}
\begin{aligned}
\sum_{p\geq -1}2^{-\tilde{\zeta} p}\|\F_p g\|_{L^2}2^{ap}\|\t\F_p h\|_{L^2}2^{bp}\|\t\F_p f\|_{L^2}\ls \|g\|_{H^{-\zeta,\a}}\|h\|_{H^a}\|f\|_{H^b},\quad \mbox{with}\quad a,b\in\R_{\geq0},~a+b=2s,.
\end{aligned}  
\end{equation*}
For the first term, let $\a>\f12,a,b\in\R_{\geq0},a+b=2s$, it has the bound
\begin{equation*}
\begin{aligned}
&(\sum_{l\geq-1}2^{2al}\|\F_l h\|^2_{L^2})^{\f12}(\sum_{l\geq-1}2^{2bl}\|\t\F_lf\|^2_{L^2})^{\f12} (\sum_{p\geq-1}p^{2\a}2^{-2\zeta p}\|\F_p g\|^2_{L^2})^{\f12}(\sum_{p\geq-1}p^{-2\a})^{\f12}
\ls \|g\|_{H^{-\zeta,\a}}\|h\|_{H^a}\|f\|_{H^b}.
\end{aligned}  
\end{equation*}
Therefore, we obtain 
\begin{equation*}
\begin{aligned}
\sum_{l\leq p-N_0}|\M_{-1,p,l}^2|+\sum_{p\geq-1}|\M^3_{-1,p}|
\ls\|g\|_{H^{-\zeta,\a}}\|h\|_{H^a}\|f\|_{H^b},
\end{aligned}  
\end{equation*}
with $a,b\in\R_{\geq 0}$, $a+b=2s$ and $\a>\f12$.  Thus we conclude that
\begin{equation}\label{Q_-1a}
\begin{aligned}
|(Q_{-1}(g,h),f)_v|
\ls\|g\|_{H^{-\zeta,\a}}\|h\|_{H^{a}}\|f\|_{H^b}+\|g\|_{H^{\lambda_1}}\|h\|_{H^{\mu,\a}}\|f\|_{H^{\lambda_2}},
\end{aligned}  
\end{equation}
with $a,b\in\R_{\geq 0}$, $a+b=2s$, $\a>\f12$, $\lambda_1,\lambda_2,\mu$ satisfying $\lambda_1+\lambda_2\geq -(\gamma+3)$, $\mu\leq \frac{3}{2}$ and $\lambda_1+\lambda_2+\mu=-(\gamma+\frac{3}{2})$.

Take $\lambda_1=a$, $\lambda_2=b$, $\mu=-(\gamma+2s+\f32)$ and using the fact $\zeta\leq\gamma+2s+\f32$, we have
\begin{equation}\label{Q_-1}
\begin{aligned}
|(Q_{-1}(g,h),f)_v|
\ls(\|g\|_{H^{-\zeta,\a}}\|h\|_{H^a}+\|g\|_{H^a}\|h\|_{H^{-\zeta,\a}})\|f\|_{H^b},
\end{aligned}  
\end{equation}
with $a,b\in\R_{\geq 0}$, $a+b=2s$, $\a>\f12$ and $\zeta$ is defined in \eqref{zeta1}.

\bigskip

Now let $w_1,w_2\in\R$ and $w_1+w_2=\gamma+2s$. Combining  \eqref{Q_-1}, \eqref{Qk1} and using the decomposition in phase space(see \eqref{ubdecom}), we have 
\begin{equation*}
    \begin{aligned}
|(Q(g,h),f)_v|\ls \&\sum_{k\geq N_0-1}2^{(\gamma+2s)k}\|\cU_{k-N_0}g\|_{L^1}\|\t\cP_kh\|_{H^a}\|\t\cP_kf\|_{H^b} \\   
\&+\sum_{j\geq k+N_0,k\geq0}2^{(\gamma+2s)k}\|\cP_{j}g\|_{L^1}\|\t\cP_jh\|_{H^a}\|\t\cP_jf\|_{H^b} \\   
 \&+\sum_{|j-k|\leq N_0,k\geq0}2^{(\gamma+2s)k}\|\cP_{j}g\|_{L^1}\|\cU_{k+N_0}h\|_{H^a}\|\cU_{k+N_0}f\|_{H^b} \\
 &+\sum_{j\geq N_0-1}\big(\|\cP_jg\|_{H^{-\zeta,\a}}\|\t\cP_jh\|_{H^a}\|\t\cP_jf\|_{H^b}+\|\cP_jg\|_{H^a}\|\t\cP_jh\|_{H^{-\zeta,\a}}\|\t\cP_jf\|_{H^b}\big)\\   
 \&+\sum_{j\leq N_0+1}\Big(\|\cP_jg\|_{H^{-\zeta,\a}}\|\cU_{N_0}h\|_{H^a}\|\cU_{N_0}f\|_{H^b}+\|\cP_jg\|_{H^a}\|\cU_{N_0}h\|_{H^{-\zeta,\a}}\|\cU_{N_0}f\|_{H^b}\Big)\\
=:\& u_1+u_2+u_3+u_4+u_5.
\end{aligned}
\end{equation*}

For $u_1$, noticing that $\|\cU_{k-N_0}g\|_{L^1}\ls \|g\|_{L^1}$ for any $k\geq-1$, thus we have 
\begin{equation}\label{u_1}
\begin{aligned}
u_1\ls \|g\|_{L^1}(\sum_{k\geq N_0-1}2^{2\omega_1 k}\|\t\cP_k h\|^2_{H^a})^{\f12}(\sum_{k\geq N_0-1}2^{2\omega_2 k}\|\t\cP_k f\|^2_{H^b})^{\f12}\ls \|g\|_{L^1}\|h\|_{H^a_{\omega_1}}\|f\|_{H^b_{\omega_2}},
\end{aligned}  
\end{equation}
where we use Lemma \ref{le1.4} in the last step.

For $u_2$, noticing that $\gamma+2s< 0$, which implies $\sum_{k\leq j-N_0}2^{(\gamma+2s)k}\ls 1$, then we have 
\begin{equation}\label{u_2}
\begin{aligned}
u_2\ls2^{-(\gamma+2s)j}\|\cP_jg\|_{L^1}(\sum_{j\geq -1}2^{2\omega_1 j}\|\t\cP_j h\|^2_{H^a})^{\f12}(\sum_{j\geq -1}2^{2\omega_2 j}\|\t\cP_j f\|^2_{H^b})^{\f12}\ls \|g\|_{L^1_{-(\gamma+2s)}}\|h\|_{H^a_{\omega_1}}\|f\|_{H^b_{\omega_2}}.
\end{aligned}  
\end{equation}

For $u_3$, since $\|\cU_{k+N_0} h\|_{H^a}\ls 2^{k(-\omega_1)^+}\|h\|_{H^a_{\omega_1}}$(see \eqref{a001}) and $|j-k|\leq N_0$, we have
\begin{equation}\label{u_3}
\begin{aligned}
u_3&\ls \sum_{k\geq 0}2^{(\gamma+2s+(-\omega_1)^++(-\omega_2)^+)k}\|\t \cP_kg\|_{L^1}\|h\|_{H^a_{\omega_1}}\|f\|_{H^b_{\omega_2}}\ls\|g\|_{L^1_{\gamma+2s+(-\omega_1)^++(-\omega_2)^+}}\|h\|_{H^a_{\omega_1}}\|f\|_{H^b_{\omega_2}}.
\end{aligned}  
\end{equation}

Similarly, for $u_4$ and $u_5$, we can deduce that
\begin{equation}\label{u_4}
\begin{aligned}
u_4\ls \|g\|_{H^{-\zeta,\a}_{-(\gamma+2s)}}\|h\|_{H^a_{\omega_1}}\|f\|_{H^b_{\omega_2}}+\|g\|_{H^a_{\omega_1}}\|h\|_{H^{-\zeta,\a}_{-(\gamma+2s)}}\|f\|_{H^b_{\omega_2}},
\end{aligned}  
\end{equation}
and
\begin{equation}\label{u_5}
\begin{aligned}
u_5\ls C_N(\|g\|_{H^{-\zeta,\a}}\|h\|_{H^a_{-N}}\|f\|_{H^b_{-N}}+\|g\|_{H^a}\|h\|_{H^{-\zeta,\a}_{-N}}\|f\|_{H^b_{-N}}),\quad \forall N\in\N.
\end{aligned}  
\end{equation}

Combining the estimates \eqref{u_1}, \eqref{u_2}, \eqref{u_3}, \eqref{u_4} and \eqref{u_5}, we conclude that
\begin{equation}\label{u_45b}
\begin{aligned}
|(Q(g,h),f)_v|&\lesssim (\|g\|_{L^1_{\t w}}+\|g\|_{ H^{-\zeta,\a}_{-(\gamma+2s)}})\|h\|_{H^a_{w_1}}\|f\|_{H^b_{w_2}}+\|g\|_{H^a_{w_1}}\|h\|_{ H^{-\zeta,\a}_{-(\gamma+2s)}}\|f\|_{H^b_{w_2}}\\
+& C_N(\|g\|_{H^{-\zeta,\a}}\|h\|_{H^a_{-N}}\|f\|_{H^b_{-N}}+\|g\|_{H^a}\|h\|_{H^{-\zeta,\a}_{-N}}\|f\|_{H^b_{-N}})
\end{aligned}  
\end{equation}
for any $N\in\N$ and $\t w=\max\{-(\gamma+2s),\gamma+2s+(-\omega_1)^++(-\omega_2)^+\}$. This completes the proof of \eqref{upperbound1}.
\end{proof}




Next, we give the proof of \eqref{upperbound2}.
\begin{proof}[Proof of Proposition \ref{prop upper} \eqref{upperbound2}]
    On one hand, it follows from (3.21) in \cite{he2018sharp} that for $k\geq 0$,
    \begin{equation*}
         \begin{aligned}
|(Q_k(g,h),f)_v|&\ls2^{(\gamma+s)k}\|g\|_{L^1}\|h\|_{H^{a_1}}\|f\|_{H^{b_1}}\\
+&~2^{\gamma k}\|g\|_{L^1_{2s}}(\|(-\Delta_{\SS^2})^{\frac{a}{2}}h\|_{L^2}+\|h\|_{H^a})(\|(-\Delta_{\SS^2})^{\frac{b}{2}}f\|_{L^2}+\|f\|_{H^b}),
    \end{aligned}  
    \end{equation*}
    where $a+b=2s,a,b\geq 0$ and $a_1+b_1=s$. On the other hand, recall \eqref{Q_-1} that 
\begin{equation*}
\begin{aligned}
|(Q_{-1}(g,h),f)_v|
\ls\|g\|_{H^{-\zeta,\a}}\|h\|_{H^a}\|f\|_{H^b}+\|g\|_{H^a}\|h\|_{H^{-\zeta,\a}}\|f\|_{H^b}.
\end{aligned}  
\end{equation*}
The proof of \eqref{upperbound2} follows by exactly the same method as above  and  \cite[Theorem 1.4]{he2018sharp}.
\end{proof}

\section{Commutator estimates $[Q,\F_j]$}\label{commutatorestimatesQFj}
In this section, we will give estimates for the commutator $[Q,\F_j]$. We split it into two parts:
\begin{equation}\label{splitDkD-1}
\begin{aligned}
 &(\cF_jQ(g,h)-Q(g,\cF_jh),\cF_jf)_v=\sum_{k\geq0}(\cF_jQ_k(g,h)-Q_k(g,\cF_jh),\cF_jf)_v\\
 &+(\cF_jQ_{-1}(g,h)-Q_{-1}(g,\cF_jh),\cF_jf)_v=:\sum_{k\geq0}\fD_{k}+\fD_{-1}.
\end{aligned}
\end{equation}
Here $Q_k$ ($k\geq0$) is the part of the collision operator in which the relative velocity is localized at $|v-v_*|\sim 2^k$, away from the kinetic singularity, while $Q_{-1}$ carries the singular region $|v-v_*|\ls 1$; see \eqref{DefPhi}.

\smallskip

We first give the upper bound for $\fD_k$.
\begin{lemma}\label{fDk}
 Recall that the localized operator $\mF_j$ is defined in Definition \ref{Fj}.  It holds that
\begin{equation*}
\begin{aligned}
 \sum_{k\geq0}\fD_k\ls C_{N,\delta}\|g\|_{L^1_{(-\omega_1)^++(-\omega_2)^++\delta}}(\|\mF_jh\|_{H^a_{\omega_1}}+2^{-jN}\|h\|_{H^{-N}_{-N}})(\|\mF_jf\|_{H^b_{\omega_2}}+2^{-jN}\|f\|_{H^{-N}_{-N}}),
\end{aligned}
\end{equation*}
where $\omega_1,\omega_2\in \R,a,b\in[0,2s-1]$ satisfying $a+b=2s-1,\omega_1+\omega_2=\gamma+2s-1$. The constants $\delta>0$  and $N\in\N$ can be sufficiently small and large, respectively. 
\end{lemma}
\begin{proof}
Noticing that
\begin{equation*}
    (\cF_jQ_k(g,h)-Q_k(g,\cF_jh),\cF_jf)_v=(Q_k(g,h),\cF^2_jf)_v-(Q_k(g,\cF_jh),\cF_jf)_v.
\end{equation*}
Then due to \eqref{ubdecom}, \(\sum_{k\ge 0}\mathcal{\fD}_k\) admits the further decomposition
\begin{equation*}
    \begin{aligned}
      \sum_{k\geq0}\fD_k=A_1+A_2,
    \end{aligned}
\end{equation*}
where 
\begin{equation*}
    \begin{aligned}
A_1&:=\sum_{k\geq N_0-1}\big(\cF_j Q_k(\cU_{k-N_0}g,\tilde{\cP}_kh)- Q_k(\cU_{k-N_0}g,\cF_j\tilde{\cP}_kh),\cF_j\tilde{\cP}_kf\big)_v\\
    +&\sum_{l\geq k+N_0,k\geq 0}\big(\cF_j Q_k(\cP_lg,\tilde{\cP}_lh)- Q_k(\cP_lg,\cF_j\tilde{\cP}_lh),\cF_j\tilde{\cP}_lf\big)_v\\
    +&\sum_{|l-k|\leq N_0,k\geq 0}\big(\cF_j Q_k(\cP_lg,\cU_{k+N_0}h)- Q_k(\cP_lg,\cF_j\cU_{k+N_0}h),\cF_j\cU_{k+N_0}f\big)_v;\\
            \end{aligned}
\end{equation*}
\begin{equation*}
    \begin{aligned}
    A_2&:=\sum_{k\geq N_0-1}\Big[\big(Q_k
   (\cU_{k-N_0}g,\tilde{\cP}_kh),(\tilde{\cP}_k\cF^2_j-\cF^2_j\tilde{\cP}_k)f\big)_v\\
    +&\big(Q_k(\cU_{k-N_0}g,(\cF_j\tilde{\cP}_k-\tilde{\cP}_k\cF_j)h),\cF_j\tilde{P}_kf\big)_v+
    \big(Q_k(\cU_{k-N_0}g,\tilde{\cP}_k\cF_jh),(\cF_j\tilde{\cP}_k-\tilde{\cP}_k\cF_j)f\big)_v\Big]\\
    +&\sum_{l\geq k+N_0,k\geq 0}\Big[\big(Q_k(\cP_lg,\tilde{\cP}_lh),(\tilde{\cP}_l\cF^2_j-\cF^2_j\tilde{\cP}_l)f\big)_v\\
    +&\big(Q_k(\cP_lg,(\cF_j\tilde{\cP}_l-\tilde{\cP}_l\cF_j)h),\F_j\t\cP_lf)_v
    +(Q_k(\cP_lg,\tilde{\cP}_l\cF_jh),(\cF_j\tilde{\cP}_l-\tilde{\cP}_l\cF_j)f\big)_v\Big]\\
    +&\sum_{|l-k|\leq N_0,k\geq 0}\Big[\big( Q_k({\cP}_lg,\cU_{k+N_0}h),(\cU_{k+N_0}\cF^2_j-\cF^2_j\cU_{k+N_0})f\big)_v\\
    +&\big(Q_k({\cP}_lg,(\cF_j\cU_{k+N_0}-\cU_{k+N_0}\cF_j)h),\cF_j\cU_{k+N_0}f\big)_v
    +\big(Q_k({\cP}_lg,\cU_{k+N_0}\cF_jh),(\cF_j\cU_{k+N_0}-\cU_{k+N_0}\cF_j)f\big)_v\Big].
    \end{aligned}
\end{equation*}
Therefore, this result can be obtained by combining  \cite[Lemma 4.20 and Lemma 4.21]{CHJ}.
\end{proof}

We now focus on the estimate of $\fD_{-1}$. Similar to \eqref{2.2}, by Bobylev's formula, we have
\begin{equation*}
\begin{aligned}
    \fD_{-1}&=\sum_{|p-j|\leq 3N_0}\sum_{q\leq p+4N_0}\Big(\cF_jQ_{-1}(\cF_qg,\cF_ph)-Q_{-1}(\cF_qg,\cF_j\cF_ph),\cF_jf\Big)_v\\
    +&\sum_{p>j+3N_0}\Big(\cF_jQ_{-1}(\tilde{\cF}_pg,\cF_ph),\cF_jf\Big)_v
    +\sum_{p<j-3N_0}\Big(\cF_jQ_{-1}(\tilde{\cF}_jg,\cF_ph),\cF_jf\Big)_v\\
    &=:\sum_{|p-j|\leq 3N_0}\sum_{q\leq p+4N_0}\fD_{-1}^{q<p,j}(g,h,f)+\sum_{p>j+3N_0}\fD_{-1}^{p>j}(g,h,f)+\sum_{p<j-3N_0}\fD_{-1}^{p<j}(g,h,f).
\end{aligned}    
\end{equation*}
We now give the estimate for each term.

\begin{lemma}\label{fD-1}
    For smooth functions $g,h$ and $f$, it holds that
    
    \medskip
    
\noindent $\bullet$ For $p>j+3N_0$,
    \begin{equation}\label{p>j}
        |\fD_{-1}^{p>j}(g,h,f)|\ls2^{-(\gamma+4)p}2^{\frac{5}{2}j}\|\tilde{\cF}_pg\|_{L^2}\|\cF_ph\|_{L^2}\|\cF_jf\|_{L^2}.
    \end{equation}
    
\noindent $\bullet$ For $p<j-3N_0$
    \begin{equation}\label{p<j}
        |\fD_{-1}^{p<j}(g,h,f)|\ls2^{-(\gamma+\f32)j}\|\tilde{\cF}_jg\|_{L^2}\|\cF_ph\|_{L^2}\|\cF_jf\|_{L^2}.
    \end{equation}
    
\noindent $\bullet$ For $|p-j|\leq 3N_0$ and $q\leq p+4N_0$,
    \begin{equation}\label{p=j}
    \begin{aligned}
        |\fD_{-1}^{q<p, j}(g,h,f)|
        \ls 2^{-(\t\zeta-1)q}2^{(2s-1)j}\|\cF_qg\|_{L^2}\|\cF_ph\|_{L^2}\|\cF_jf\|_{L^2}.
    \end{aligned}
    \end{equation}
   Here
   \begin{equation*}
\t\zeta=\left\{\begin{array}{ll}
1,  &-\frac{1}{2}<\gamma+2s<0,\\
1-\frac{\log_2 q}{q}, & \gamma+2s=-\frac{1}{2},\\
\gamma+2s+\frac{3}{2}, & -1<\gamma+2s<-\frac{1}{2}.
\end{array}\right.
\end{equation*}
\end{lemma}
\begin{remark}\label{rem:tzeta}
In the case $\gamma+2s=-\f12$, the factor $2^{-(\t\zeta-1)q}=q$ in \eqref{p=j} produces a logarithmic loss in the frequency $2^q$. Exactly as in Proposition \ref{prop upper}, this loss is absorbed by replacing $\a$ with $\a+1$, and we shall therefore not distinguish $\t\zeta$ from $\zeta$ in the notation below.
\end{remark}

\begin{proof}
    \textbf{Step 1}: Proof of \eqref{p>j} for case $j\geq0$. Using Bobylev's formula \eqref{bobylev} and Taylor expansion, we have
\begin{equation*}
\begin{aligned}
        &|(\cF_jQ_{-1}(\tilde{\cF}_pg,\cF_ph),\cF_jf)_v|\\
        =&\Big|\iint_{\R^6\times\SS^2}b(\frac{\xi}{|\xi|}\cdot\sigma)\big[\widehat{\Phi^{\gamma}_{-1}}(\eta-\xi^-)-\widehat{\Phi^{\gamma}_{-1}}(\eta)\big]\widehat{\tilde{\cF}_pg}(\eta)\widehat{\cF_ph}(\xi-\eta)\varphi(2^{-j}\xi)\overline{\widehat{\F_jf}}(\xi)\d\sigma\d\eta\d\xi\Big|\\
        \leq&\Big|\iint_{\R^6\times\SS^2}b(\frac{\xi}{|\xi|}\cdot\sigma)\xi^-\cdot\nabla\widehat{\Phi^{\gamma}_{-1}}(\eta)\widehat{\tilde{\cF}_pg}(\eta)\widehat{\cF_ph}(\xi-\eta)\varphi(2^{-j}\xi)\overline{\widehat{\F_jf}}(\xi)\d\sigma\d\eta\d\xi\Big|\\
        &+\Big|\int_0^1\iint_{\R^6\times\SS^2}b(\frac{\xi}{|\xi|}\cdot\sigma)\big[\nabla^2\widehat{\Phi^{\gamma}_{-1}}(\eta-t\xi^-):\xi^-\otimes\xi^-\big]\widehat{\tilde{\cF}_pg}(\eta)\widehat{\cF_ph}(\xi-\eta)\varphi(2^{-j}\xi)\overline{\widehat{\F_jf}}(\xi)\d\sigma\d\eta\d\xi\d t\Big|\\
        =:&~T_1+T_2.
\end{aligned}
\end{equation*}

    For $T_1$, note that $|\xi|\sim 2^j,|\eta|\sim 2^p$ with $p>j+3N_0$ and $|\partial^{\alpha}\widehat{\Phi^\gamma_{-1}}(\eta)|\ls\<\eta\>^{-(\gamma+3+|\alpha|)}$ for $|\alpha|\leq 2$, by \eqref{bcdotxi_-}, we have
\begin{equation*}
    \begin{aligned}
        T_1\ls&\int_{|\xi|\sim 2^j,|\eta|\sim 2^p}\<\eta\>^{-(\gamma+4)}|\xi||\widehat{\tilde{\cF}_pg}(\eta)\widehat{\cF_ph}(\xi-\eta)\varphi(2^{-j}\xi)\overline{\widehat{\cF_jf}}(\xi)|\d\xi\d\eta\\
        \ls&~2^{-(\gamma+4)p}2^{j}\Big(\int_{|\xi|\sim2^j}\int_{\R^3}|\widehat{\tilde{\cF}_pg}(\eta)|^2\d\xi\d\eta\Big)^{\frac{1}{2}}\Big(\iint_{\R^6}|\widehat{\cF_ph}(\xi-\eta)|^2|\varphi(2^{-j}\xi)\overline{\widehat{\cF_jf}}(\xi)|^2\d\xi\d\eta\Big)^{\frac{1}{2}}\\
        \ls&~2^{-(\gamma+4)p}2^{\frac{5}{2}j}\|\tilde{\cF}_pg\|_{L^2}\|\cF_ph\|_{L^2}\|\cF_jf\|_{L^2}.
    \end{aligned}
\end{equation*}
    
    For $T_2$, noticing that $|\eta-t\xi^-|\sim 2^p$ for any $t\in[0,1]$, which implies
    $$
    \begin{aligned}
        \big|\nabla^2\widehat{\Phi^{\gamma}_{-1}}(\eta-t\xi^-):\xi^-\otimes\xi^-\big|\leq\<\eta\>^{-(\gamma+5)}\theta^2|\xi|^2.
    \end{aligned}
    $$
    Then by the same argument, we can derive that
    $$
    \begin{aligned}
        T_2\ls2^{-(\gamma+5)p}2^{\frac{7}{2}j}\|\tilde{\cF}_pg\|_{L^2}\|\cF_ph\|_{L^2}\|\cF_jf\|_{L^2}\ls 2^{-(\gamma+4)p}2^{\frac{5}{2}j}\|\tilde{\cF}_pg\|_{L^2}\|\cF_ph\|_{L^2}\|\cF_jf\|_{L^2}.
    \end{aligned}
    $$
    Therefore, \eqref{p>j} holds true for $j\geq 0$ and $j=-1$ can be handled similarly. 
    
\bigskip

      \textbf{Step 2}: Proof of \eqref{p<j}. Denote
    \begin{equation*}
        \begin{aligned}
        P&:=|(\cF_j Q_{-1}(\tilde{\cF}_jg,\cF_ph),\cF_jf)_v|\\
        &=\Big|\iint_{\R^6\times\SS^2}b(\frac{\xi}{|\xi|}\cdot\sigma)\big[\widehat{\Phi^{\gamma}_{-1}}(\eta-\xi^-)-\widehat{\Phi^{\gamma}_{-1}}(\eta)\big]\widehat{\tilde{\cF}_jg}(\eta)\widehat{\cF_ph}(\xi-\eta)\overline{\widehat{\F_jf}}(\xi)\varphi(2^{-j}\xi)\d\sigma\d\eta\d\xi\Big|.
    \end{aligned}
    \end{equation*}
    Note that
    \begin{equation*}
        \begin{aligned}
        |\eta|\sim2^j,|\xi-\eta|\sim2^p,|\xi^+|\sim |\xi|\sim2^j~\mbox{with}~p<j-3N_0,
    \end{aligned}
    \end{equation*}
    which implies $|\eta-\xi^-|=|\xi^++(\eta-\xi)|\sim2^j$. We split the integration domain of $P$ into two parts: $2|\xi^-|\leq \<\eta\>$ and  $2|\xi^-|>\<\eta\>$ and denote them by $P_1$ and $P_2$. Then for $P_1$, we have $|\eta-t\xi^-|\sim \<\eta\>$ for $t\in[0,1]$ and one may copy the argument of $T$ to get that 
    \begin{equation*}
           \begin{aligned}
        P_1\ls2^{-(\gamma+\frac{3}{2})j}\|\tilde{\cF}_jg\|_{L^2}\|\cF_ph\|_{L^2}\|\cF_jf\|_{L^2}.
    \end{aligned} 
    \end{equation*}
   For $P_2$, noticing that $\theta\gs \<\eta\>/|\xi|$ at this point, which implies that
$       \int_{\SS^2}b(\frac{\xi}{|\xi|}\cdot\sigma)d\sigma \ls |\xi|^{2s}\<\eta\>^{-2s}\sim 1.$
It leads to
   \begin{equation*}
   \begin{aligned}
        P_2&\ls \Big|\iint_{\R^6\times\SS^2}b(\frac{\xi}{|\xi|}\cdot\sigma)\widehat{\Phi^{\gamma}_{-1}}(\eta-\xi^-)\widehat{\tilde{\cF}_jg}(\eta)\widehat{\cF_ph}(\xi-\eta)\overline{\widehat{\F_jf}}(\xi)\varphi(2^{-j}\xi)\d\sigma\d\eta\d\xi\Big|\\
        &+ \Big|\iint_{\R^6\times\SS^2}b(\frac{\xi}{|\xi|}\cdot\sigma)\widehat{\Phi^{\gamma}_{-1}}(\eta)\widehat{\tilde{\cF}_jg}(\eta)\widehat{\cF_ph}(\xi-\eta)\overline{\widehat{\F_jf}}(\xi)\varphi(2^{-j}\xi)\d\sigma\d\eta\d\xi\Big|\\
        &\ls2^{-(\gamma+3)j}\Big|\int_{\SS^2}b(\frac{\xi}{|\xi|}\cdot\sigma)\d\sigma\iint_{\R^6}\big|\widehat{\tilde{\cF}_jg}(\eta)\widehat{\cF_ph}(\xi-\eta)\overline{\widehat{\F_jf}}(\xi)\varphi(2^{-j}\xi)\big|\d\sigma\d\eta\d\xi\Big|\\
        &\ls 2^{-(\gamma+\frac{3}{2})j}\|\tilde{\cF}_jg\|_{L^2}\|\cF_ph\|_{L^2}\|\cF_jf\|_{L^2}.
   \end{aligned}
   \end{equation*}
Therefore, we obtain \eqref{p<j}.

    \bigskip

    \textbf{Step 3}: Proof of \eqref{p=j}. Denote
    \begin{equation*}
        \begin{aligned}
        R:=&~|(\cF_j Q_{-1}(\cF_qg,\cF_ph)-Q_{-1}(\cF_qg,\cF_j\cF_ph),\cF_jf)_v|
        =\Big|\iint_{\R^6\times\SS^2}b(\frac{\xi}{|\xi|}\cdot\sigma)\big[\widehat{\Phi^{\gamma}_{-1}}(\eta-\xi^-)-\widehat{\Phi^{\gamma}_{-1}}(\eta)\big]\\
        &\times\widehat{\cF_qg}(\eta)\widehat{\cF_ph}(\xi-\eta)\overline{\widehat{\F_jf}}(\xi)\left(\varphi(2^{-j}\xi)-\varphi(2^{-j}(\xi-\eta))\right)\d\sigma\d\eta\d\xi\Big|.
    \end{aligned}
    \end{equation*}
    Recall that
    $$
    \begin{aligned}
        |\xi^-|\sim\theta|\xi|,|\eta|\sim2^q,|\xi-\eta|\sim2^p,|\xi|\sim2^j\quad\mbox{with}\quad|p-j|\leq3N_0,q\leq p+4N_0.
    \end{aligned}
    $$
    Similar to the decomposition of $P$, we split $b(\cos\theta)=b_{\theta\leq\delta_02^{q-j}}+b_{\theta>\delta_02^{q-j}}=:\hat b_1+\hat b_2,$
    where $\delta_0$ is chosen to be sufficiently small such that
$|\eta-t\xi^-|\sim|\eta-\xi^-|\sim2^q$ when $ \theta\leq \delta_0 2^{q-j}\sim \delta_0|\eta|/|\xi|, ~t\in[0,1].$
    So we can separate $R$ into $R_1+R_2$ using the separation of $b$.
    
    For  $R_1$, using Taylor expansion, we have
    \begin{equation*}
        \begin{aligned}
        |R_1|\ls&\Big|\int_{\sigma\in\SS^2}\iint_{\R^6}\hat b_1(\frac{\xi}{|\xi|}\cdot\sigma)\big[\nabla\widehat{\Phi^{\gamma}_{-1}}(\eta)\cdot\xi^-\big]\widehat{\cF_qg}(\eta)\widehat{\cF_ph}(\xi-\eta)\overline{\widehat{\F_jf}}(\xi)\left(\varphi(2^{-j}\xi)-\varphi(2^{-j}(\xi-\eta))\right)\d\sigma\d\eta\d\xi\Big|\\
        +&\Big|\int_0^1\iint_{\R^6\times\SS^2}\hat b_1(\frac{\xi}{|\xi|}\cdot\sigma)\big[\nabla^2\widehat{\Phi^{\gamma}_{-1}}(\eta-t\xi^-):\xi^-\otimes\xi^-\big]\widehat{\cF_qg}(\eta)\widehat{\cF_ph}(\xi-\eta)\overline{\widehat{\F_jf}}(\xi)\\
        &\times\left(\varphi(2^{-j}\xi)-\varphi(2^{-j}(\xi-\eta))\right)\d\sigma\d\eta\d\xi \d t \Big|
        =:~R_{1,1}+R_{1,2}.
         \end{aligned}
    \end{equation*}
Using the same method in the proof of \eqref{p>j} and the fact that $ |\eta|\sim 2^q, ~|\xi|\sim 2^j,~\mbox{and}~
       |\varphi(2^{-j}(\xi))-\varphi(2^{-j}(\xi-\eta))|\ls2^{-j}|\eta|,$
    we have
    \begin{equation*}
    \begin{aligned}
        R_{1,1}\ls&~2^{-(\gamma+4)q}2^{-j}\int_{\theta\ls 2^{q-j}}\theta^{1-2s}\d\theta \iint_{|\eta|\sim2^q}|\widehat{\cF_qg}(\eta)||\widehat{\cF_ph}(\xi-\eta)||\overline{\widehat{\cF_jf}}(\xi)||\xi||\eta|\d\xi\d\eta\\
        \ls&~2^{-(\gamma+3)q}2^{(2-2s)(q-j)}\Big(\iint_{|\eta|\sim2^q}\left|\overline{\widehat{\cF_jf}}(\xi)\right|^2\d\xi\d\eta\Big)^{\frac{1}{2}}\Big(\iint_{\R^6}|\widehat{\cF_qg}(\eta)|^2|\widehat{\cF_ph}(\xi-\eta)|^2\d\xi\d\eta\Big)^{\frac{1}{2}}\\
        \ls&~2^{-(\gamma+2s-\frac{1}{2})q}2^{(2s-2)j}\|\cF_qg\|_{L^2}\|\cF_ph\|_{L^2}\|\cF_jf\|_{L^2}.
    \end{aligned}
    \end{equation*}
    Similarly, using $|\eta-t\xi^-|\sim2^q$ and $|\nabla^2\widehat{\Phi^{\gamma}_{-1}}(\eta-t\xi^-):\xi^-\otimes\xi^-|\ls2^{-(\gamma+5)q}\theta^2|\xi|^2$,we can derive that
    $$
    \begin{aligned}
        R_{1,2}\ls2^{-(\gamma+2s+\frac{1}{2})q}2^{(2s-1)j}\|\cF_qg\|_{L^2}\|\cF_ph\|_{L^2}\|\cF_jf\|_{L^2}.
    \end{aligned}
    $$
Since $q \leq p+4N_0\leq j+7N_0$, combining the estimates of $R_{1,1}$ and $R_{1,2}$, we obtain 
    \begin{equation}\label{R1}
    \begin{aligned}
        R_{1}\ls2^{-(\gamma+2s+\frac{1}{2})q}2^{(2s-1)j}\|\cF_qg\|_{L^2}\|\cF_ph\|_{L^2}\|\cF_jf\|_{L^2}.
    \end{aligned}        
    \end{equation}

\medskip

    For $R_2$, using $|\varphi(2^{-j}(\xi))-\varphi(2^{-j}(\xi-\eta))|\ls2^{-j}|\eta|$, we split it as follows:
    \begin{equation*}
        \begin{aligned}
        R_2\ls&~2^{-j}\int_{\SS^2}\iint_{\R^6}|\eta|\hat b_2(\frac{\xi}{|\xi|}\cdot\sigma)\big|\widehat{\Phi^{\gamma}_{-1}}(\eta)\big|\big|\widehat{\cF_qg}(\eta)\big|\big|\widehat{\cF_ph}(\xi-\eta)\big|\big|\widehat{\cF_jf}(\xi)\big|\d\sigma\d\xi\d\eta\\
        +&2^{-j}\int_{\SS^2}\iint_{\R^6}|\eta|\hat b_2(\frac{\xi}{|\xi|}\cdot\sigma)\big|\widehat{\Phi^{\gamma}_{-1}}(\eta-\xi^-)\big|\big|\widehat{\cF_qg}(\eta)\big|\big|\widehat{\cF_ph}(\xi-\eta)\big|\big|\widehat{\cF_jf}(\xi)\big|\d\sigma\d\xi\d\eta
        =:R_{2,1}+R_{2,2}.
    \end{aligned}    
    \end{equation*}

For  $R_{2,1}$, we have 
    \begin{equation}\label{R21}
        \begin{aligned}
        R_{2,1}\ls&~2^{-j}\int_{\theta \gs 2^{q-j}} \theta ^{-1-2s}\d \theta \iint_{\R^6}\<\eta\>^{-(\gamma+3)}|\eta|\big|\widehat{\cF_qg}(\eta)\big|\big|\widehat{\cF_ph}(\xi-\eta)\big|\big|\widehat{\cF_jf}(\xi)\big|\d\xi\d\eta\\
        \ls&~2^{-j}2^{2s(j-q)}2^{-(\gamma+2)q}\Big(\iint_{|\eta|\sim2^q}\left|\widehat{\cF_jf}(\xi)\right|^2\d\xi\d\eta\Big)^{\frac{1}{2}}\Big(\iint_{\R^6}|\widehat{\cF_qg}(\eta)|^2|\widehat{\cF_ph}(\xi-\eta)|^2\d\xi\d\eta\Big)^{\frac{1}{2}}\\
        \ls&~2^{-(\gamma+2s+\frac{1}{2})q}2^{(2s-1)j}\|\cF_qg\|_{L^2}\|\cF_ph\|_{L^2}\|\cF_jf\|_{L^2}.
    \end{aligned} 
    \end{equation}

   For $R_{2,2}$, if $|\eta-\xi^-|\geq \f12|\eta|$, we have $|\widehat{\Phi^{\gamma}_{-1}}(\eta-\xi^-)|\ls \<\eta\>^{-(\gamma+3)}$ and we can get the same bound as $R_{2,1}$. If $|\eta-\xi^-|< \f12|\eta|$, by the method used in the estimates of $D_{2,1,2}$ in Lemma \ref{M23}, i.e., do the change of variables $\eta-\xi^-\rightarrow \t \eta$ (thus $\xi-\eta\rightarrow\xi^+-\t\eta$ and $|\t \eta|\ls 2^q$), we have
   $$
\begin{aligned}
   R_{2,2}\ls&~2^{-j}2^{q}\Big(\int_{\sigma\in\SS^2}\iint \mathbf{1}_{\theta\gs 2^{q-j}} b^\kappa(\frac{\xi}{|\xi|}\cdot\sigma)|\widehat{\F_qg}(\eta)|^2|\overline{\widehat{\t\F_jf}}(\xi)|^2\d \sigma\d \eta \d \xi\Big)^{\frac{1}{2}}\\&\times
    \Big(\int_{\sigma\in\SS^2}\iint_{|\xi^+|\sim 2^j}\mathbf{1}_{\theta\gs \f{|\t \eta|}{|\xi^+|}} b^{2-\kappa}(\frac{\xi^+}{|\xi^+|}\cdot\sigma)|\widehat{\Phi^\gamma_{-1}}(\t \eta)|^2|\widehat{\F_ph}(\xi^+-\t \eta)|^2\d \sigma\d \t\eta \d \xi^+\Big)^{\frac{1}{2}}.
\end{aligned}
$$
It yields that
    \begin{equation}\label{R22}
    \begin{aligned}
     R_{2,2}\ls 2^{-(\t\zeta-1)q}2^{(2s-1)j}\|\cF_qg\|_{L^2}\|\cF_ph\|_{L^2}\|\cF_jf\|_{L^2}.
    \end{aligned} 
    \end{equation}
Therefore, combining the estimates \eqref{R1}, \eqref{R21} and \eqref{R22}, we can obtain \eqref{p=j}.

We complete the proof of this lemma. 
\end{proof}

In what follows we shall use the phase decomposition \eqref{ubdecom} and apply  Lemma \ref{fD-1} to obtain a more precise estimate. 

\begin{lemma}\label{precisebound}
For smooth functions \(g,h\) and \(f\),  more precise bounds for \(\mathcal{\fD}_{-1}^{p>j}(g,h,f)\), \(\mathcal{\fD}_{-1}^{p<j}(g,h,f)\) and \(\mathcal{\fD}_{-1}^{q<p,j}(g,h,f)\) are given in \eqref{bound1}, \eqref{bound2} and \eqref{bound3}.
\end{lemma}
\begin{proof}
It follows from  \eqref{ubdecom} that
\begin{equation}\label{recall2.6}
\begin{aligned}
( Q_{-1}(g,h), f)_v &=
\sum_{l\ge N_0-1}( Q_{-1}(\mathcal{P}_{l} g, \tilde{\mathcal{P}}_lh), \tilde{\mathcal{P}}_lf )_v+\sum_{l\le N_0+1}( Q_{-1}( \mathcal{P}_{l} g, \mathcal{U}_{N_0}h), \mathcal{U}_{N_0}f )_v.  
\end{aligned}
\end{equation}

\textbf{Step 1:} Estimate of $\fD_{-1}^{p>j}$.  From \eqref{recall2.6}, we have the following decomposition
    \begin{equation*}
     \begin{aligned}
        &(\cF_jQ_{-1}(\tilde{\cF}_pg,\cF_ph),\cF_jf)_v=(Q_{-1}(\tilde{\cF}_pg,\cF_ph),\cF^2_jf)_v\\
        =&\sum_{l\ge N_0-1}( Q_{-1}(\mathcal{P}_{l} \tilde{\cF}_pg, \tilde{\mathcal{P}}_l\cF_ph), \tilde{\mathcal{P}}_l\cF^2_jf )_v+\sum_{l\le N_0+1}( Q_{-1}( \mathcal{P}_{l} \tilde{\cF}_pg, \mathcal{U}_{N_0}\cF_ph), \mathcal{U}_{N_0}\cF^2_jf )_v
        =:G_1+G_2,
    \end{aligned}       
    \end{equation*}
    where
    \begin{equation*}
         \begin{aligned}
        &G_1=\sum_{l\geq N_0-1}(\cF_jQ_{-1}(\cP_l\tilde{\cF}_pg, \cF_p\tilde{
        \cP}_lh),\cF_j\tilde{\cP}_lf)_v+
        \sum_{l\leq N_0+1}(\cF_jQ_{-1}(\cP_l\tilde{\cF}_pg,\cF_p\cU_{N_0}h),\cF_j\cU_{N_0}f)_v,\\
        G_2&=\sum_{l\geq N_0-1}\Big((Q_{-1}(\cP_l\tilde{\cF}_pg,\tilde{\cP}_l\cF_ph),(\cF_j^2\tilde{\cP}_l-\tilde{\cP}_l\cF_j^2)f)_v
        +(Q_{-1}(\cP_l\tilde{\cF}_pg,(\tilde{\cP}_l\cF_p-\cF_p\tilde{\cP}_l)h),\cF_j^2\tilde{\cP}_lf)_v\Big)\\
        +&\sum_{l\leq N_0+1}\Big((Q_{-1}(\cP_l\tilde{\cF}_pg,\cU_{N_0}\cF_ph),(\cF_j^2\cU_{N_0}-\cU_{N_0}\cF_j^2)f)_v
        +(Q_{-1}(\cP_l\tilde{\cF}_pg,(\cF_p\cU_{N_0}-\cU_{N_0}\cF_p)h),\cF_j^2\cU_{N_0}f)_v\Big).
    \end{aligned}
    \end{equation*}
     We remark that the terms $G_1$ matching the form of \(\fD_{-1}^{p>j}\) are constructed deliberately, while the remaining terms arise naturally as commutators.

    \medskip
    
    {\bf\noindent $\bullet$ Estimate of $G_1$.} It follows from \eqref{bobylev} that $\<\eta\>\sim 2^p$ since $|\xi|\sim 2^j$, $\<\eta-\xi\>\sim 2^p$ and $p>j+3N_0$, which leads to
\begin{equation*}
    \begin{aligned}
        G_1=&\sum_{l\geq N_0-1}(\cF_jQ_{-1}(\tilde{\cF}_p\cP_l\tilde{\cF}_pg,\cF_p\tilde{
        \cP}_lh),\cF_j\tilde{\cP}_lf)_v
        +\sum_{l\leq N_0+1}(\cF_jQ_{-1}(\tilde{\cF}_p\cP_l\tilde{\cF}_pg,\cF_p\cU_{N_0}h),\cF_j\cU_{N_0}f)_v.
    \end{aligned}
\end{equation*}
    Due to \eqref{p>j} in Lemma \ref{fD-1}, we have
    \begin{equation}\label{G1}
       \begin{aligned}
        |G_1|\ls &\sum_{l\geq N_0-1}2^{-(\gamma+4)p}2^{\frac{5}{2}j}\|\tilde{\cF}_p\cP_l\tilde{\cF}_pg\|_{L^2}\|\cF_p\tilde{
        \cP}_lh\|_{L^2}\|\cF_j\tilde{\cP}_lf\|_{L^2}\\
        +&\sum_{l\leq N_0+1}2^{-(\gamma+4)p}2^{\frac{5}{2}j}\|\tilde{\cF}_p\cP_l\tilde{\cF}_pg\|_{L^2}\|\cF_p\cU_{N_0}h\|_{L^2}\|\cF_j\cU_{N_0}f\|_{L^2}\\
        \ls &\sum_{l\geq N_0-1} 2^{-(\gamma+4)p}2^{\frac{5}{2}j}\|\mP_l\mF_pg\|_{L^2}\|\mF_p
        \mP_lh\|_{L^2}\|\mF_j\mP_lf\|_{L^2}\\
        +&2^{-(\gamma+4)p}2^{\frac{5}{2}j}\|\mU_{N_0}\mF_pg\|_{L^2}\|\mF_p\mU_{N_0}h\|_{L^2}\|\mF_j\mU_{N_0}f\|_{L^2},
    \end{aligned} 
    \end{equation}
    where we use the fact that $\|\t\cF_p f\|_{L^2}\ls \|f\|_{L^2}$ for any $p\geq-1$. And for convenience, we uniformly use the notation \(\mF,\mP\) and \(\mU\) (see Definition \ref{Fj}).

    \medskip
    
     {\bf\noindent $\bullet$ Estimate of $G_2$.} We denote the terms on right-hand side of $G_2$ by $G_{2,1}$ to $G_{2,4}$ and first split $G_{2,1}$ as follows:
    $$
    \begin{aligned}
        G_{2,1}=&\sum_{i=1}^{3}G_{2,1}^{(i)}\quad\mbox{with}\quad
        G_{2,1}^{(1)}=\sum_{l\geq N_0-1}\sum_{|a-p|>N_0}(Q_{-1}(\cP_l\tilde{\cF}_pg,\cF_a\tilde{\cP}_l\cF_ph),(\cF_j^2\tilde{\cP}_l-\tilde{\cP}_l\cF_j^2)f)_v,\\
        G_{2,1}^{(2)}=&\sum_{l\geq N_0-1}\sum_{|a-p|\leq N_0}\sum_{|m-j|>N_0}(Q_{-1}(\cP_l\tilde{\cF}_pg,\cF_a\tilde{\cP}_l\cF_ph),\cF_m(\cF_j^2\tilde{\cP}_l-\tilde{\cP}_l\cF_j^2)f)_v,\\
        G_{2,1}^{(3)}=&\sum_{l\geq N_0-1}\sum_{|a-p|\leq N_0}\sum_{|m-j|\leq N_0}(Q_{-1}(\cP_l\tilde{\cF}_pg,\cF_a\tilde{\cP}_l\cF_ph),\cF_m(\cF_j^2\tilde{\cP}_l-\tilde{\cP}_l\cF_j^2)f)_v.\\
    \end{aligned}
    $$
    From \eqref{Q_-1}, we have (choose $a=2s,b=0$) 
    $$
    \begin{aligned}
        &|G_{2,1}^{(1)}|\leq \sum_{l\geq N_0-1}\sum_{|a-p|>N_0}|(Q_{-1}(\cP_l\tilde{\cF}_pg,\cF_a\tilde{\cP}_l\cF_ph),(\cF_j^2\tilde{\cP}_l-\tilde{\cP}_l\cF_j^2)f)_v|\\
        \ls&\sum_{l\geq N_0-1}\sum_{|a-p|>N_0}\left(\|\cP_l\tilde{\cF}_pg\|_{H^{-\zeta,\a}}\|\cF_a\tilde{\cP}_l\cF_ph\|_{H^{2s}}+\|\cP_l\tilde{\cF}_pg\|_{H^{2s}}\|\cF_a\tilde{\cP}_l\cF_ph\|_{H^{-\zeta,\a}}\right)\|(\cF_j^2\tilde{\cP}_l-\tilde{\cP}_l\cF_j^2)f\|_{L^2}.
    \end{aligned}
    $$
 For the second term on the right-hand side, noticing that
    \begin{equation*}
       \begin{aligned}
        \|\cP_l\tilde{\cF}_pg\|_{H^{2s}}\leq \sum_{|b-p|>N_0}\|\cF_{b}\cP_l\tilde{\cF}_pg\|_{H^{2s}}+\sum_{|b-p|\leq N_0}\|\cF_{b}\cP_l\tilde{\cF}_pg\|_{H^{2s}}.
    \end{aligned}  
    \end{equation*}
Thanks to \eqref{m-p>N0} and \eqref{pkfj} in Lemma \ref{PFcommutator} respectively, we have 
  \begin{equation*}
       \begin{aligned}
        &\sum_{|b-p|>N_0}\|\cF_{b}\cP_l\tilde{\cF}_pg\|_{H^{2s}}+\sum_{|b-p|\leq N_0}\|\cF_{b}\cP_l\tilde{\cF}_pg\|_{H^{2s}}\\
        \leq& ~C_N\sum_{|b-p|>N_0}2^{-(b+l+p)N}2^{2sb}\|\tilde{\cF}_pg\|_{L_{-N}^2}+\sum_{|b-p|\leq N_0}2^{2sp}(\|\mF_p\mP_lg\|_{L^2}+2^{-lN}2^{-pN}\|g\|_{H^{-N}_{-N}})\\
        \ls&~C_N2^{2sp}(\|\mF_p\mP_lg\|_{L^2}+2^{-lN}2^{-pN}\|g\|_{H^{-N}_{-N}}).
    \end{aligned}  
    \end{equation*}
Also by Lemma \ref{PFcommutator}, we can estimate the term $\|\cF_a\tilde{\cP}_l\cF_ph\|_{H^{-\zeta,\a}}$ and $\|(\cF_j^2\tilde{\cP}_l-\tilde{\cP}_l\cF_j^2)f\|_{L^2}$ to obtain that
   \begin{equation*}
   \begin{aligned}
       &\sum_{l\geq N_0-1}\sum_{|a-p|>N_0}\|\cP_l\tilde{\cF}_pg\|_{H^{2s}}\|\cF_a\tilde{\cP}_l\cF_ph\|_{H^{-\zeta,\a}}\|(\cF_j^2\tilde{\cP}_l-\tilde{\cP}_l\cF_j^2)f\|_{L^2}\\
       \ls&~C_N\sum_{l\geq N_0}\sum_{|a-p|>N_0}2^{-(a+l+p)N}(\|\mF_p\mP_lg\|_{L^2}+2^{-lN}2^{-pN}\|g\|_{H^{-N}_{-N}})\|\mF_p h\|_{L^2_{-N}}\\
        &\times(2^{-j}2^{-l}\|\mF_j\mP_lf\|_{L^2}+2^{-lN}2^{-jN}\|f\|_{H^{-N}_{-N}}).
   \end{aligned}
   \end{equation*}
   
    We can apply the same method to estimate the first term and combine them all to get
    \begin{equation}\label{G211}
        \begin{aligned}
        |G_{2,1}^{(1)}|\ls &~C_N\sum_{l\geq N_0-1}2^{-(l+p)N}(\|\mF_p\mP_lg\|_{L^2}+2^{-lN}2^{-pN}\|g\|_{H^{-N}_{-N}})\|\mF_p h\|_{L_{-N}^{2}}\\
        &\times(2^{-j}2^{-l}\|\mF_j\mP_lf\|_{L^2}+2^{-lN}2^{-jN}\|f\|_{H^{-N}_{-N}}).
    \end{aligned}
    \end{equation}
 
   For $G_{2,1}^{(2)}$, note that $|a-p|\leq N_0$  and $|m-j|>N_0$ implies $\F_m\F_j^2=0$, then from \eqref{Q_-1}, we have (choose $a=0,b=2s$)
\begin{equation*}
    \begin{aligned}
        |G_{2,1}^{(2)}|\leq &\sum_{l\geq N_0-1}\sum_{|m-j|>N_0}|(Q_{-1}(\cP_l\tilde{\cF}_pg,\tilde{\cF}_p\tilde{\cP}_l\cF_ph),\cF_m\tilde{\cP}_l\cF_j^2f)_v|\\
        \ls&\sum_{l\geq N_0-1}\sum_{|m-j|>N_0} \left(\|\cP_l\tilde{\cF}_pg\|_{H^{-\zeta,\a}}\|\t\cF_p\tilde{\cP}_l\cF_ph\|_{L^2}+\|\cP_l\tilde{\cF}_pg\|_{L^{2}}\|\t\cF_p\tilde{\cP}_l\cF_ph\|_{H^{-\zeta,\a}}\right)\|\F_m\tilde{\cP}_l\cF_j^2f\|_{H^{2s}}.
    \end{aligned}
\end{equation*}
For the first term, thanks to  Lemma \ref{PFcommutator},  we have
\begin{equation*}
    \begin{aligned}
&\|\cP_l\tilde{\cF}_pg\|_{H^{-\zeta,\a}}\ls 2^{-\zeta p} p^\a\|\mF_p\mP_lg\|_{L^2}+C_N2^{-lN}2^{-pN}\|g\|_{H^{-N}_{-N}},\\
     &\|\t\cF_p\tilde{\cP}_l\cF_ph\|_{L^{2}}\ls \|\mF_p\mP_lh\|_{L^2}+C_N2^{-lN}2^{-pN}\|h\|_{H^{-N}_{-N}},\\
     &~~~~\mbox{and}~~ \|\F_m\tilde{\cP}_l\cF_j^2f\|_{L^2}\ls C_N2^{-mN-lN-jN}\|\mF_jf\|_{L^2_{-N}}.
    \end{aligned}
\end{equation*}
The second term has a similar estimate and  it leads to 
   \begin{equation}\label{G212}
        \begin{aligned}
        |G_{2,1}^{(2)}|\ls&~C_N \sum_{l\geq N_0-1}2^{-(l+j)N}2^{-\zeta p}p^\a(\|\mF_p\mP_lg\|_{L^2}+2^{-lN}2^{-pN}\|g\|_{H^{-N}_{-N}})\\
        &\times(\|\mF_p\mP_lh\|_{L^2}+2^{-lN}2^{-pN}\|h\|_{H^{-N}_{-N}})\|\mF_j f\|_{L^2_{-N}}.
    \end{aligned}
   \end{equation}

 For $G_{2,1}^{(3)}$, since $|a-p|\leq N_0$ and $|m-j|\leq N_0$, then from \eqref{Q_-1} (choose $a=0,b=2s$) and the above argument, we can deduce that 
    \begin{equation}\label{G213}
        \begin{aligned}
        |G_{2,1}^{(3)}|\ls& \sum_{l\geq N_0-1}\left(\|\cP_l\tilde{\cF}_pg\|_{H^{-\zeta,\a}}\|\t \cF_p\tilde{\cP}_l\cF_ph\|_{L^2}+\|\cP_l\tilde{\cF}_pg\|_{L^2}\|\t\cF_p\tilde{\cP}_l\cF_ph\|_{H^{-\zeta,\a}}\right)\|\t\F_j(\cF_j^2\tilde{\cP}_l-\tilde{\cP}_l\cF_j^2)f\|_{H^{2s}}\\
        \ls &\sum_{l\geq N_0-1} C_N2^{-\zeta p}p^\a 2^{(2s-1)j}2^{-l} (\|\mF_p\mP_l g\|_{L^2}+2^{-pN-lN}\|g\|_{H^{-N}_{-N}})\\
        &\times (\|\mF_p\mP_lh\|_{L^2}+2^{-pN-lN}\|h\|_{H^{-N}_{-N}})(\|\mF_j\mP_l f\|_{L^2}+2^{-jN-lN}\|f\|_{H^{-N}_{-N}}).
        \end{aligned}
    \end{equation}

Replacing \(N\) by \(2N\) in \eqref{G211} and \eqref{G212}, and noting that \(2^{-2jN}\|\mF_j f\|_{L^2_{-2N}}\lesssim 2^{-jN}\|f\|_{H^{-N}_{-N}}\), we combine this with \eqref{G213} to obtain that
    \begin{equation}\label{G21}
       \begin{aligned}
        |G_{2,1}|\ls& \sum_{l\geq N_0-1} C_N2^{-\zeta p}p^\a 2^{(2s-1)j}2^{-l} (\|\mF_p\mP_l g\|_{L^2}+2^{-pN-lN}\|g\|_{H^{-N}_{-N}})\\
        &\times (\|\mF_p\mP_lh\|_{L^2}+2^{-pN-lN}\|h\|_{H^{-N}_{-N}})(\|\mF_j\mP_l f\|_{L^2}+2^{-jN-lN}\|f\|_{H^{-N}_{-N}}).
        \end{aligned}  
    \end{equation}

\medskip

   For $G_{2,2}$, similar to $G_{2,1}$, we can split it as follows:
    \begin{equation*}
        \begin{aligned}
        G_{2,2}=&\sum_{l\geq N_0-1}(Q_{-1}(\cP_l\tilde{\cF}_pg,(\tilde{\cP}_l\cF_p-\cF_p\tilde{\cP}_l)h),\cF_j^2\tilde{\cP}_lf)_v\\
        =&\sum_{l\geq N_0-1}\sum_{|a-p|>N_0}(Q_{-1}(\cP_l\tilde{\cF}_pg,\cF_a(\tilde{\cP}_l\cF_p-\cF_p\tilde{\cP}_l)h),\cF_j^2\tilde{\cP}_lf)_v\\
        +&\sum_{l\geq N_0-1}\sum_{|a-p|\leq N_0}(Q_{-1}(\cP_l\tilde{\cF}_pg,\cF_a(\tilde{\cP}_l\cF_p-\cF_p\tilde{\cP}_l)h),\cF_j^2\tilde{\cP}_lf)_v
        =:G_{2,2}^{(1)}+G_{2,2}^{(2)}.
    \end{aligned}
    \end{equation*}
    
    The estimate of $G_{2,2}^{(1)}$ is similar with that of $G_{2,1}^{(1)}$. Indeed, we have 
    \begin{equation*}
        \begin{aligned}
          |G_{2,2}^{(1)}| \ls \sum_{l\geq N_0-1}\sum_{|a-p|>N_0}\left(\|\cP_l\tilde{\cF}_pg\|_{H^{-\zeta,\a}}\|\cF_a\tilde{\cP}_l\cF_ph\|_{L^2}+\|\cP_l\tilde{\cF}_pg\|_{L^{2}}\|\cF_a\tilde{\cP}_l\cF_ph\|_{H^{-\zeta,\a}}\right)\|\F_j^2\tilde{\cP}_lf\|_{H^{2s}}.
        \end{aligned}
    \end{equation*}
 Taking the first term as an example, we have
     \begin{equation*}
        \begin{aligned}
        &\|\cP_l\tilde{\cF}_pg\|_{H^{-\zeta,\a}}\ls 2^{-\zeta p}p^\a \|\mF_p\mP_l g\|_{L^2}+C_N 2^{-lN-pN}\|g\|_{H^{-N}_{-N}},\quad \|\F_j^2\t\cP_l f\|_{H^{2s}}\ls 2^{2sj}\|\mF_j\mP_l f\|_{L^2}\\
&~\mbox{and}\quad \|\cF_a\tilde{\cP}_l\cF_ph\|_{L^2}\ls C_N 2^{-aN-lN-pN}\|\cF_p h\|_{L^2_{-N}},\quad \mbox{for}\quad |a-p|>N_0.
        \end{aligned}
    \end{equation*}   
The second term admits a similar estimate, whence we obtain
      \begin{equation*}
        \begin{aligned}
          |G_{2,2}^{(1)}| \ls  \sum_{l\geq N_0-1}C_N 2^{-(l+p)N}\|g\|_{H^{-N}_{-N}}\|h\|_{H^{-N}_{-N}}\|f\|_{H^{-N}_{-N}}.
        \end{aligned}
    \end{equation*}  
    
The estimate of $G_{2,2}^{(2)}$ is similar with that of $G_{2,1}^{(3)}$, we can derive that
      \begin{equation*}
        \begin{aligned}
          |G_{2,2}^{(2)}| \ls&  \sum_{l\geq N_0-1}C_N 2^{(-\zeta-1)p}p^\a 2^{2sj}2^{-l}(\|\mF_p\mP_l g\|_{L^2}+2^{-pN-lN}\|g\|_{H^{-N}_{-N}})\\
        &\times (\|\mF_p\mP_lh\|_{L^2}+2^{-pN-lN}\|h\|_{H^{-N}_{-N}})(\|\mF_j\mP_l f\|_{L^2}+2^{-jN-lN}\|f\|_{H^{-N}_{-N}}).
        \end{aligned}
    \end{equation*}
Therefore, we obtain that
      \begin{equation}\label{G22}
        \begin{aligned}
          |G_{2,2}| \ls&  \sum_{l\geq N_0-1}C_N 2^{(-\zeta-1)p}p^\a 2^{2sj}2^{-l}(\|\mF_p\mP_l g\|_{L^2}+2^{-pN-lN}\|g\|_{H^{-N}_{-N}})\\
        &\times (\|\mF_p\mP_lh\|_{L^2}+2^{-pN-lN}\|h\|_{H^{-N}_{-N}})(\|\mF_j\mP_l f\|_{L^2}+2^{-jN-lN}\|f\|_{H^{-N}_{-N}}).
        \end{aligned}
    \end{equation}

 Observing that \(G_{2,3}\) and \(G_{2,4}\) have the same structure as \(G_{2,1}\) and \(G_{2,2}\), respectively, with \(\t\cP_l\) replaced by \(\mathcal{U}_{N_0}\),  we can still handle them by the same method. We therefore omit the details and obtain
     \begin{equation}\label{G234}
        \begin{aligned}
          |G_{2,3}+G_{2,4}| \ls& ~ C_N 2^{-\zeta p}p^\a 2^{(2s-1)j}(\|\mF_p \mU_{N_0}g\|_{L^2}+2^{-pN}\|g\|_{H^{-N}_{-N}})\\
        &\times (\|\mF_p\mU_{N_0}h\|_{L^2}+2^{-pN}\|h\|_{H^{-N}_{-N}})(\|\mF_j\mU_{N_0}f\|_{L^2}+2^{-jN}\|f\|_{H^{-N}_{-N}}).
        \end{aligned}
    \end{equation}

Patching together \eqref{G21}, \eqref{G22} and \eqref{G234} and note that $p>j+3N_0$, we obtain
  \begin{equation}\label{G2}
        \begin{aligned}
          |G_{2}| \ls& \sum_{l\geq N_0-1} C_N2^{-\zeta p}p^\a 2^{(2s-1)j}2^{-l} (\|\mF_p\mP_l g\|_{L^2}+2^{-pN-lN}\|g\|_{H^{-N}_{-N}})\\
        &\times (\|\mF_p\mP_lh\|_{L^2}+2^{-pN-lN}\|h\|_{H^{-N}_{-N}})(\|\mF_j\mP_l f\|_{L^2}+2^{-jN-lN}\|f\|_{H^{-N}_{-N}})\\
        &+~ C_N 2^{-\zeta p}p^\a 2^{(2s-1)j}(\|\mF_p \mU_{N_0}g\|_{L^2}+2^{-pN}\|g\|_{H^{-N}_{-N}})\\
        &\times (\|\mF_p\mU_{N_0}h\|_{L^2}+2^{-pN}\|h\|_{H^{-N}_{-N}})(\|\mF_j\mU_{N_0}f\|_{L^2}+2^{-jN}\|f\|_{H^{-N}_{-N}}).
        \end{aligned}
    \end{equation}
    
Further combining \eqref{G1}  and \eqref{G2}, we obtain
     \begin{equation}\label{bound1}
        \begin{aligned}
        &|\fD_{-1}^{p>j}(g,h,f)|\ls \sum_{l\geq N_0-1} 2^{-(\gamma+4)p}2^{\frac{5}{2}j}\|\mP_l\mF_pg\|_{L^2}\|\mF_p
        \mP_lh\|_{L^2}\|\mF_j\mP_lf\|_{L^2}
        +2^{-(\gamma+4)p}2^{\frac{5}{2}j}\|\mU_{N_0}\mF_pg\|_{L^2}\\&\times\|\mF_p\mU_{N_0}h\|_{L^2}
        \|\mF_j\mU_{N_0}f\|_{L^2}
        +\sum_{l\geq N_0-1} C_N2^{-\zeta p}p^\a 2^{(2s-1)j}2^{-l} (\|\mF_p\mP_l g\|_{L^2}+2^{-pN-lN}\|g\|_{H^{-N}_{-N}})\\
         &(\|\mF_p\mP_lh\|_{L^2}
         +2^{-pN-lN}\|h\|_{H^{-N}_{-N}})(\|\mF_j\mP_l f\|_{L^2}+2^{-jN-lN}\|f\|_{H^{-N}_{-N}})+
        C_N 2^{-\zeta p}p^\a 2^{(2s-1)j}\\&(\|\mF_p \mU_{N_0}g\|_{L^2}+2^{-pN}\|g\|_{H^{-N}_{-N}})
         (\|\mF_p\mU_{N_0}h\|_{L^2}+2^{-pN}\|h\|_{H^{-N}_{-N}})(\|\mF_j\mU_{N_0}f\|_{L^2}+2^{-jN}\|f\|_{H^{-N}_{-N}}).
        \end{aligned}
    \end{equation}

    \medskip
    
    \textbf{Step 2:}  Estimate of $\fD_{-1}^{p<j}$. Since the estimates for \(\fD_{-1}^{p<j}\)  follow a similar pattern to that for \(\fD_{-1}^{p>j}\), we only outline the main steps and omit the details. We have the following decomposition: 
    \begin{equation*}
      \begin{aligned}
        (\cF_jQ_{-1}(\tilde{\cF_j}g,\cF_ph),\cF_jf)_v=(Q_{-1}(\tilde{\cF_j}g,\cF_ph),\cF^2_jf)_v=X_1+X_2,
    \end{aligned}  
    \end{equation*}
    where
     \begin{equation*}
    \begin{aligned}
        &X_1=\sum_{l\geq N_0-1}(\cF_jQ_{-1}(\cP_l\tilde{\cF}_jg,\cF_p\tilde{
        \cP}_lh),\cF_j\tilde{\cP}_lf)_v
        +\sum_{l<N_0+1}(\cF_jQ_{-1}(\cP_l\tilde{\cF}_jg,\cF_p \cU_{N_0}h),\cF_j\cU_{N_0}f)_v,\\
        &X_2=\sum_{l\geq N_0-1}\Big((Q_{-1}(\cP_l\tilde{\cF}_jg,\tilde{\cP}_l\cF_ph),(\tilde{\cP}_l\cF_j^2-\cF_j^2\tilde{\cP}_l)f)_v
        +(Q_{-1}(\cP_l\t\cF_jg,(\t\cP_l\F_p-\F_p\t\cP_l)h),\cF_j^2\tilde{\cP}_lf)\Big)\\
        &+\sum_{l< N_0+1}\Big((Q_{-1}(\cP_l\tilde{\cF}_jg,\cU_{N_0}\cF_ph),(\cU_{N_0}\cF_j^2-\cF_j^2\cU_{N_0})f)_v
        +(Q_{-1}(\cP_l\t\cF_jg,(\cU_{N_0}\cF_p-\cF_p\cU_{N_0})h),\cF_j^2\cU_{N_0}f)_v\Big).
    \end{aligned}
    \end{equation*}

     {\bf\noindent $\bullet$ Estimate of  $X_1$.} It follows from \eqref{bobylev} that $\<\eta\>\sim 2^j$ since $|\xi|\sim 2^j$, $\<\eta-\xi\>\sim 2^p$ and $p<j-3N_0$, which leads to
 \begin{equation*}
    \begin{aligned}
        X_1=&\sum_{l\geq N_0-1}(\cF_jQ_{-1}(\t\F_j\cP_l\tilde{\cF}_jg,\cF_p\tilde{
        \cP}_lh),\cF_j\tilde{\cP}_lf)_v
        +\sum_{l<N_0+1}(\cF_jQ_{-1}(\t\F_j\cP_l\tilde{\cF}_jg,\cF_p \cU_{N_0}h),\cF_j\cU_{N_0}f)_v.
\end{aligned}
    \end{equation*}
Then due to \eqref{p<j} in Lemma \ref{fD-1}, we have 
 \begin{equation*}
    \begin{aligned}
        |X_1|\ls &\sum_{l\geq N_0-1}2^{-(\gamma+\f32)j}\|\mP_l \mF_jg\|_{L^2}\|\mF_p\mP_l h\|\|\mF_j \mP_l f\|_{L^2}
    +2^{-(\gamma+\f32)j}\|\mU_{N_0}\mF_j g\|_{L^2}\|\mF_p \mU_{N_0}h\|\|\mF_j \mU_{N_0} f\|_{L^2}.
\end{aligned}
    \end{equation*}

 {\bf \noindent $\bullet$  Estimate of $X_{2}$:} We denote the terms on the right-hand side by $X_{2,1}$ to $X_{2,4}$. Note that they have the same structure as $G_{2,1}$ to $G_{2,4}$ apart from differences in frequency, and we use \eqref{Q_-1a} instead of \eqref{Q_-1} in the estimate(choose $a=2s$ and $b=0$). Thus we can use the same argument to obtain that
  \begin{equation}\label{X2}
        \begin{aligned}
          |X_{2}| \ls& \sum_{l\geq N_0-1} C_N(2^{-\zeta j}j^\a 2^{(2s-1)p}+2^{\frac{1}{2}p}p^\a 2^{-(\gamma+3)j})2^{-l} (\|\mF_j\mP_l g\|_{L^2}+2^{-pN-lN}\|g\|_{H^{-N}_{-N}})\\
        &\times (\|\mF_p\mP_lh\|_{L^2}+2^{-pN-lN}\|h\|_{H^{-N}_{-N}})(\|\mF_j\mP_l f\|_{L^2}+2^{-jN-lN}\|f\|_{H^{-N}_{-N}})\\
        &+~ C_N  (2^{-\zeta j}j^\a 2^{(2s-1)p}+2^{\frac{1}{2}p}p^\a 2^{-(\gamma+3)j})(\|\mF_j \mU_{N_0}g\|_{L^2}+2^{-pN}\|g\|_{H^{-N}_{-N}})\\
        &\times (\|\mF_p\mU_{N_0}h\|_{L^2}+2^{-pN}\|h\|_{H^{-N}_{-N}})(\|\mF_j\mU_{N_0}f\|_{L^2}+2^{-jN}\|f\|_{H^{-N}_{-N}}).
        \end{aligned}
    \end{equation}
Then we conclude that
     \begin{equation}\label{bound2}
        \begin{aligned}
        &|\fD_{-1}^{p<j}(g,h,f)|\ls \sum_{l\geq N_0-1}2^{-(\gamma+\f32)j}\|\mP_l \mF_jg\|_{L^2}\|\mF_p\mP_l h\|_{L^2}
        \|\mF_j \mP_l f\|_{L^2}
    +2^{-(\gamma+\f32)j}\|\mU_{N_0}\mF_j g\|_{L^2}\|\mF_p \mU_{N_0}h\|_{L^2}\\
    &\times\|\mF_j \mU_{N_0} f\|_{L^2}
    +\sum_{l\geq N_0-1} C_N(2^{-\zeta j}j^\a 2^{(2s-1)p}+2^{\frac{1}{2}p}p^\a 2^{-(\gamma+3)j})2^{-l} (\|\mF_j\mP_l g\|_{L^2}+2^{-pN-lN}\|g\|_{H^{-N}_{-N}})
        (\|\mF_p\mP_lh\|_{L^2}\\
        &+2^{-pN-lN}\|h\|_{H^{-N}_{-N}})(\|\mF_j\mP_l f\|_{L^2}
        +2^{-jN-lN}\|f\|_{H^{-N}_{-N}})
        +C_N(2^{-\zeta j}j^\a 2^{(2s-1)p}+ 2^{\frac{1}{2}p}p^\a 2^{-(\gamma+3)j})\\&(\|\mF_j \mU_{N_0}g\|_{L^2}+2^{-pN}\|g\|_{H^{-N}_{-N}})
        (\|\mF_p\mU_{N_0}h\|_{L^2}+2^{-pN}\|h\|_{H^{-N}_{-N}})(\|\mF_j\mU_{N_0}f\|_{L^2}+2^{-jN}\|f\|_{H^{-N}_{-N}}).
        \end{aligned}
    \end{equation}

    \textbf{Step 3:}  Estimate of $\fD_{-1}^{q<p,j}$. The main method remains the same as before and we only point out the differences. We split it as follows:
    \begin{equation*}
         \begin{aligned}
        (\cF_jQ_{-1}(\cF_qg,\cF_ph)-Q_{-1}(\cF_qg,\cF_j\cF_ph),\cF_jf)_v=\sum_{i=1}^{4}Y_i,
    \end{aligned}   
    \end{equation*}
    where
    \begin{equation*}
      \begin{aligned}
        &Y_1=\sum_{l\geq N_0-1}(\cF_jQ_{-1}(\cP_l\cF_qg,\cF_p\tilde{\cP}_lh)-Q_{-1}(\cP_l\cF_qg,\cF_j\cF_p\tilde{\cP}_lh),\cF_j\tilde{\cP}_lf)_v,\\
        &Y_2=\sum_{l<N_0+1}(\cF_jQ_{-1}(\cP_l\cF_qg,\cF_p\cU_{N_0}h)-Q_{-1}(\cP_l\cF_qg,\cF_j\cF_p\cU_{N_0}h),\cF_j\cU_{N_0}f)_v,\\
        &Y_3=\sum_{l\geq N_0-1}\Big((Q_{-1}(\cP_l\cF_qg,\tilde{\cP}_l\cF_ph),(\tilde{\cP}_l\cF_j^2-\cF_j^2\tilde{\cP}_l)f)_v+(Q_{-1}(\cP_l\cF_qg,\tilde{\cP}_l\cF_j\cF_ph),(\tilde{\cP}_l\cF_j-\cF_j\tilde{\cP}_l)f)_v\\
        &+(Q_{-1}(\cP_l\cF_qg,(\tilde{\cP}_l\cF_p-\cF_p\tilde{\cP}_l)\cF_ph),\cF_j^2\tilde{\cP}_lf)_v
        +(Q_{-1}(\cP_l\cF_qg,(\tilde{\cP}_l\cF_j\cF_p-\cF_j\cF_p\tilde{\cP}_l)\cF_ph),\cF_j\tilde{\cP}_lf)_v
        \Big),\\
        &Y_4=\sum_{l<N_0+1}\Big((Q_{-1}(\cP_l\cF_qg,\cU_{N_0}\cF_ph),(\cU_{N_0}\cF_j^2-\cF_j^2\cU_{N_0})f)_v+(Q_{-1}(\cP_l\cF_qg,\cU_{N_0}\cF_j\cF_ph),(\cU_{N_0}\cF_j-\cF_j\cU_{N_0})f)_v\\
        &+(Q_{-1}(\cP_l\cF_qg,(\cU_{N_0}\cF_p-\cF_p\cU_{N_0})h),\cF^2_j\cU_{N_0}f)_v+(Q_{-1}(\cP_l\cF_qg,(\cF_j\cF_p\cU_{N_0}-\cU_{N_0}\cF_j\cF_p)h),\cF_j\cU_{N_0}f)_v
        \Big).
    \end{aligned}  
 \end{equation*}
   Note that $\<\eta\>\ls 2^j$ since $\<\xi\>\sim 2^j$ and $\<\xi-\eta\>\sim 2^p$ with $|p-j|\leq 3N_0$, then by \eqref{p=j} in Lemma \ref{fD-1}, we have
    \begin{equation*}
        \begin{aligned}
            |Y_1+Y_2|\ls& \sum_{l\geq N_0-1}\sum_{a\leq j+7N_0}2^{-(\zeta-1)a}2^{(2s-1)j}\|\mF_a\mP_l\mF_qg\|_{L^2}\|\mF_j\mP_lh\|_{L^2}\|\mF_j\mP_lf\|_{L^2}\\
            &+\sum_{a\leq j+7N_0}2^{-(\zeta-1)a}2^{(2s-1)j}\|\mF_a\mU_{N_0}\mF_qg\|_{L^2}\|\mF_j\mU_{N_0}h\|_{L^2}\|\mF_j\mU_{N_0}\|_{L^2}.
        \end{aligned}
    \end{equation*}
After split it into $|a-q|\leq N_0$ and $|a-q|>N_0$, we can further obtain from \eqref{m-p>N0} that
    \begin{equation*}
        \begin{aligned}
            |Y_1+Y_2|\ls& \sum_{l\geq N_0-1}C_N\big(2^{-(\zeta-1)q}2^{(2s-1)j}\|\mP_l\mF_qg\|_{L^2}+2^{-qN-lN}\|g\|_{H^{-N}_{-N}}\big)\|\mF_j\mP_lh\|_{L^2}\|\mF_j\mP_lf\|_{L^2}\\
            &+C_N\big(2^{-(\zeta-1)q}2^{(2s-1)j}\|\mU_{N_0}\mF_qg\|_{L^2}+2^{-qN}\|g\|_{H^{-N}_{-N}}\big)\|\mF_j\mU_{N_0}h\|_{L^2}\|\mF_j\mU_{N_0}\|_{L^2}.
        \end{aligned}
    \end{equation*}

    For $Y_3$ and $Y_4$, similar to the estimates of $G_2$ and $X_2$ and note that $|p-j|\leq 3N_0$, we can derive that
      \begin{equation*}
        \begin{aligned}
            |Y_3+Y_4|\ls& \sum_{l\geq N_0-1} C_N2^{-\zeta q}q^\a 2^{(2s-1)j}2^{-l} (\|\mF_q\mP_l g\|_{L^2}+2^{-qN-lN}\|g\|_{H^{-N}_{-N}})\\
        &\times (\|\mF_j\mP_lh\|_{L^2}+2^{-jN-lN}\|h\|_{H^{-N}_{-N}})(\|\mF_j\mP_l f\|_{L^2}+2^{-jN-lN}\|f\|_{H^{-N}_{-N}})\\
        &+~ C_N 2^{-\zeta q}q^\a 2^{(2s-1)j}(\|\mF_q \mU_{N_0}g\|_{L^2}+2^{-jN}\|g\|_{H^{-N}_{-N}})\\
        &\times (\|\mF_j\mU_{N_0}h\|_{L^2}+2^{-jN}\|h\|_{H^{-N}_{-N}})(\|\mF_j\mU_{N_0}f\|_{L^2}+2^{-jN}\|f\|_{H^{-N}_{-N}}).
        \end{aligned}
    \end{equation*}  
Therefore, we conclude that
       \begin{equation}\label{bound3}
        \begin{aligned}
            &|\fD_{-1}^{q<p,j}(g,h,f)|
            \ls\sum_{l\geq N_0-1}C_N\big(2^{-(\zeta-1)q}2^{(2s-1)j}\|\mP_l\mF_qg\|_{L^2}+2^{-qN-lN}\|g\|_{H^{-N}_{-N}}\big)\|\mF_j\mP_lh\|_{L^2}\|\mF_j\mP_lf\|_{L^2}\\
            &+C_N\big(2^{-(\zeta-1)q}2^{(2s-1)j}\|\mU_{N_0}\mF_qg\|_{L^2}+2^{-qN}\|g\|_{H^{-N}_{-N}}\big)\|\mF_j\mU_{N_0}h\|_{L^2}\|\mF_j\mU_{N_0}f\|_{L^2}\\
            &+\sum_{l\geq N_0-1} C_N2^{-\zeta q}q^\a 2^{(2s-1)j}2^{-l} (\|\mF_q\mP_l g\|_{L^2}+2^{-qN-lN}\|g\|_{H^{-N}_{-N}})
         (\|\mF_j\mP_lh\|_{L^2}+2^{-jN-lN}\|h\|_{H^{-N}_{-N}})\\
         &\times(\|\mF_j\mP_l f\|_{L^2}+2^{-jN-lN}\|f\|_{H^{-N}_{-N}})
        + C_N 2^{-\zeta q}q^\a 2^{(2s-1)j}(\|\mF_q \mU_{N_0}g\|_{L^2}+2^{-qN}\|g\|_{H^{-N}_{-N}})\\
        &\times (\|\mF_j\mU_{N_0}h\|_{L^2}+2^{-jN}\|h\|_{H^{-N}_{-N}})(\|\mF_j\mU_{N_0}f\|_{L^2}+2^{-jN}\|f\|_{H^{-N}_{-N}}).
        \end{aligned}
    \end{equation}  
Summarizing the above results, we complete the proof of this lemma.
\end{proof}
 
 From  \eqref{splitDkD-1}, Lemma \ref{fDk} and Lemma \ref{precisebound}, we finally conclude
\begin{proposition}[Commutator of $Q$ and $\cF_j$]\label{QFj}
 Suppose $-1<\gamma+2s<0$ and $\gamma+4s=1$. For any smooth functions $g,f$ and $h$, it holds that
\begin{equation}\label{QFjghf}
\begin{aligned}
    &\Big|\big(\cF_jQ(g,h)-Q(g,\cF_jh),\cF_jf\big)_v\Big|\\
    \ls&~\|g\|_{L^1_{(-\omega_1)^++(-\omega_2)^++\delta}}(\|\mF_jh\|_{H^a_{\omega_1}}+2^{-jN}\|h\|_{H^{-N}_{-N}})(\|\mF_jf\|_{H^b_{\omega_2}}+2^{-jN}\|f\|_{H^{-N}_{-N}})\\
    &+\sum_{p>j+3N_0}|\fD_{-1}^{p>j}(g,h,f)|+\sum_{p<j-3N_0}|\fD_{-1}^{p<j}(g,h,f)|+\sum_{|p-j|\leq 3N_0}\sum_{q\leq p+4N_0}|\fD_{-1}^{q<p,j}(g,h,f)|,
\end{aligned}
\end{equation}
where $\omega_1,\omega_2\in \R,a,b\in[0,2s-1]$ satisfying $a+b=2s-1,\omega_1+\omega_2=\gamma+2s-1$. The constants $\delta>0$  and $N\in\N$ can be sufficiently small and large, respectively. The upper bounds for the last three terms are given in \eqref{bound1}, \eqref{bound2} and \eqref{bound3}.
\end{proposition}

\smallskip

We end this section with a corollary that follows from Proposition \ref{QFj}.
\begin{corollary}\label{cor1}
  For smooth function $g$ and $f$, it holds that
        \begin{equation}\label{QFjf}
    \begin{aligned}
       &\big|(\cF_jQ(g,\cP_kf)-Q(g,\cF_j\cP_kf),\cF_j\cP_kf)_v\big|\\
       \ls&~ 2^{(\gamma+2s-1)k}2^{(2s-1)j}\|g\|_{L^1_2}(\|\mF_j\mP_kf\|_{L^2}+2^{-jN-kN}\|f\|_{H^{-N}_{-N}})^2+\sum_{i=1}^3 \mB_i(g,f)
    \end{aligned}
    \end{equation} 
   with $\mB_1(g,f)=\sum_{p>j+3N_0}|\fD_{-1}^{p>j}(g,\cP_kf,\cP_kf)|$, $\mB_2(g,f)=\sum_{p<j-3N_0}$ $|\fD_{-1}^{p<j}(g,\cP_kf,\cP_kf)|$ and $\mB_3(g,f)=\sum_{q<j+7N_0}|\fD_{-1}^{q<p,j}(g,\cP_kf,\cP_kf)|$. 
\end{corollary}
\begin{proof}
We plug $(g,h,f)\rightarrow (g,\cP_kf,\cP_kf)$ into \eqref{QFjghf} and replace $N$ by $2N$, then for the first term on the right-hand side, from Lemmas \ref{le1.4} and \ref{PFcommutator}, one may easy to check that
            \begin{equation*}
    \begin{aligned}
       &\|\mF_j\mP_kf\|_{H^a_{\omega_1}}+2^{-2jN}\|\cP_kf\|_{H^{-2N}_{-2N}}\ls 2^{\omega_1 k}2^{aj}\|\mF_j\mP_kf\|_{L^2}+2^{-jN-kN}\|f\|_{H^{-N}_{-N}},\\
       &\|\mF_j\mP_kf\|_{H^b_{\omega_2}}+2^{-2jN}\|\cP_kf\|_{H^{-2N}_{-2N}}\ls 2^{\omega_2 k}2^{bj}\|\mF_j\mP_kf\|_{L^2}+2^{-jN-kN}\|f\|_{H^{-N}_{-N}}
    \end{aligned}
    \end{equation*} 
with $a+b=2s-1$ and $\omega_1+\omega_2=\gamma+2s-1$, then we choose $\omega_1=0$ and $\omega_2=\gamma+2s-1$ which implies that $(-\omega_1)^++(-\omega_2)^++\delta<2$, thus we can get the bound
  \begin{equation}
    \begin{aligned}
      2^{(\gamma+2s-1)k}2^{(2s-1)j}\|g\|_{L^1_2}(\|\mF_j\mP_kf\|_{L^2}+2^{-jN-kN}\|f\|_{H^{-N}_{-N}})^2.
    \end{aligned}
    \end{equation}  
This ends the proof of this corollary.
\end{proof}

\section{Commutator estimates $[Q,\cP_k]$}\label{commutatorestimatesQPk}

In this section, we will give estimates for the commutator $[Q,\cP_k]$.  Specifically, we will give the upper bound for $(\cP_kQ(g,h)-Q(g,\cP_kh),\F^2_j\cP_kf)_v$.  To obtain the sharpest  estimates, we need to decompose it simultaneously in both phase space and frequency space.

Due to \eqref{ubdecom}, we first have the following decomposition:
\begin{equation*}
\begin{aligned}
    &(\cP_kQ(g,h),\F_j^2\cP_kf)_v=\sum_{\substack{m\geq N_0-1\\|m- k|\leq 2N_0}}(Q_{m}(\cU_{m-N_0}g,\tilde{\cP}_mh),\tilde{\cP}_m\cP_k\F_j^2\cP_kf)_v\\
    &+\sum_{\substack{l\geq m+N_0\\|l- k|\leq 2N_0}}(Q_{m}(\cP_{l}g,\tilde{\cP}_lh),\tilde{\cP}_l\cP_k\F_j^2\cP_kf)_v
    +\sum_{\substack{|l-m|\leq N_0\\k\leq m+2N_0}}(Q_{m}(\cP_lg,\cU_{m+N_0}h),\cU_{m+N_0}\cP_k\F_j^2\cP_kf)_v,\\
    &(Q_{m}(g,\cP_kh),\F_j^2\cP_kf)_v=\sum_{\substack{m\geq N_0-1\\|m- k|\leq 2N_0}}(Q_{m}(\cU_{m-N_0}g,\tilde{\cP}_m\cP_kh),\tilde{\cP}_m\F_j^2\cP_kf)_v\\
    &+\sum_{\substack{l\geq m+N_0\\|l- k|\leq 2N_0}}(Q_{m}(\cP_{l}g,\tilde{\cP}_l\cP_kh),\tilde{\cP}_l\F_j^2\cP_kf)_v
    +\sum_{\substack{|l-m|\leq N_0\\k\leq m+2N_0}}(Q_{m}(\cP_lg,\cU_{m+N_0}\cP_kh),\cU_{m+N_0}\F_j^2\cP_kf)_v,~~m\geq-1,
\end{aligned}    
\end{equation*}
where we use the fact that $\t\cP_m\cP_k=0$ if $|m-k|>2N_0$ and $\cU_{m+N_0}\cP_k=0$ if $k>m+2N_0$. Noticing that $\t\cP_l\cP_k=\cP_k\t\cP_l$  and $\cU_{m+N_0}\cP_k=\cP_k\cU_{m+N_0}$, we have
\begin{equation}\label{commutator2}
    \begin{aligned}
   (\cP_kQ(g,h)-Q(g,\cP_kh),\F^2_j\cP_kf)_v=\Delta_1+\Delta_2+\Delta_3,  
    \end{aligned}
\end{equation}
with
\begin{equation}\label{00Delta1}
\begin{aligned}
    \Delta_1:=&\sum_{m\geq N_0-1,|m-k|\leq 2N_0}(\cP_kQ_m(\cU_{m-N_0}g,\tilde{\cP}_mh)-Q_m(\cU_{m-N_0}g,\cP_k\tilde{\cP}_mh),\cF_j^2\tilde{\cP}_m\cP_kf)_v\\
    +&\sum_{m\geq N_0-1,|m-k|\leq 2N_0}(\cP_kQ_m(\cU_{m-N_0}g,\tilde{\cP}_mh)-Q_m(\cU_{m-N_0}g,\cP_k\tilde{\cP}_mh),\big[\tilde{\cP}_m,\cF_j^2\big]\cP_kf)_v\\
    =:~&\Delta_{1,1}+\Delta_{1,2},
    \end{aligned}
\end{equation}
where we further decompose it into $\Delta_{1,1}$ and $\Delta_{1,2}$. Similarly, 
\begin{equation*}
\begin{aligned}
    \Delta_2:=
    &\sum_{l\geq m+N_0,|l-k|\leq 2N_0}(\cP_kQ_{m}(\cP_lg,\tilde{\cP}_lh)-Q_{m}(\cP_lg,\cP_k\tilde{\cP}_lh),\cF_j^2\tilde{\cP}_l\cP_kf)_v\\
    +&\sum_{l\geq m+N_0,|l-k|\leq 2N_0}(\cP_kQ_{m}(\cP_lg,\tilde{\cP}_lh-Q_{m}(\cP_lg,\cP_k\tilde{\cP}_lh),\big[\tilde{\cP}_l,\cF^2_j\big]\cP_kf))_v
    \\
    =:~&\Delta_{2,1}+\Delta_{2,2},\\
    \Delta_3:=
    &\sum_{|l-m|\leq N_0,k\leq  m+2N_0}(\cP_kQ_{m}(\cP_lg,\cU_{m+N_0}h)-Q_{m}(\cP_lg,\cP_k\cU_{m+N_0}h),\cF_j^2\cU_{m+N_0}\cP_kf)_v\\
    +&\sum_{|l-m|\leq N_0,k\leq m+2N_0}(\cP_kQ_{m}(\cP_lg,\cU_{m+N_0}h)-Q_{m}(\cP_lg,\cP_k\cU_{m+N_0}h),\big[\cU_{m+N_0},\cF_j^2\big]\cP_kf)_v\\
    =:~&\Delta_{3,1}+\Delta_{3,2}.
\end{aligned}
\end{equation*}

We next continue with the frequency decomposition. Recall from \eqref{2.2} that
\begin{equation*}
   \begin{aligned}
    &(Q_m(g,h),\cF^2_jf)_v=\sum_{|b-j|\leq 2N_0,a\leq b+3N_0}(Q_m(\cF_ag,\cF_bh),\cF^2_jf)_v\\
    +&\sum_{b>j+2N_0,|a-b|\leq N_0}(Q_m(\cF_ag,\cF_bh),\cF^2_jf)_v
    +\sum_{b<j-2N_0,|a-j|\leq 2N_0}(Q_m(\cF_ag,\cF_bh),\cF^2_jf)_v.
\end{aligned} 
\end{equation*}
And we denote
\begin{equation}
    \begin{aligned}
        \Delta^{\text{Freq}}_{m}(g,h,f)
    :=&\sum_{\substack{|b-j|\leq 2N_0,a\leq b+3N_0}}(\cP_kQ_m(\F_ag,\F_bh)-Q_m(\F_ag,\cP_k\F_bh),\cF_j^2\cP_kf)_v\\
    &+\sum_{\substack{b>j+2N_0,|a-b|\leq N_0}}(\cP_kQ_m(\F_ag,\F_bh)-Q_m(\F_ag,\cP_k\F_bh),\cF_j^2\cP_kf)_v\\
    &+\sum_{\substack{b<j-2N_0,|a-j|\leq 2N_0}}(\cP_kQ_m(\F_ag,\F_bh)-Q_m(\F_ag,\cP_k\F_bh),\cF_j^2\cP_kf)_v,\\ 
    \Delta^{\text{Comm}}_{m}(g,h,f)
    :=&(Q_m(g,h),[\cP_k,\cF_j^2]\cP_kf)_v+\sum_{\substack{|b-j|\leq 2N_0\\a\leq b+3N_0}}(Q_m(\F_ag,[\cP_k,\F_b]h),
    \cF_j^2\cP_kf)_v\\&+\sum_{\substack{b>j+2N_0\\|a-b|\leq N_0}}(Q_m(\F_ag,[\cP_k,\F_b]h),\cF_j^2\cP_kf)_v
    +\sum_{\substack{b<j-2N_0\\|a-j|\leq 2N_0}}(Q_m(\F_ag,
    [\cP_k,\F_b]h),\cF_j^2\cP_kf)_v
    \\&+\sum_{\substack{|b-j|\leq 2N_0\\a\leq b+3N_0}}(Q_m(\F_ag,\F_bh),[\cF_j^2,\cP_k]\cP_kf)_v
    +\sum_{\substack{b>j+2N_0\\|a-b|\leq N_0}}(Q_m(\F_ag,\F_bh),[\cF_j^2,\cP_k]\cP_kf)_v\\&+\sum_{\substack{b<j-2N_0\\|a-j|\leq 2N_0}}(Q_m(\F_ag,\F_bh),[\cF_j^2,\cP_k]\cP_kf)_v.\label{Deltacomm}
    \end{aligned}
\end{equation}
For $\Delta_{1,1}$, we can split it as  $\Delta_{1,1}=\Delta_{1,1}^{1}+\Delta_{1,1}^2$ with 
\begin{equation}\label{Delta1,11}
 \begin{aligned}
    \Delta^1_{1,1}
    :=&\sum_{\substack{m\geq N_0-1,|m-k|\leq 2N_0}}\big(\Delta^{\text{Freq}}_{m}(\cU_{m-N_0}g,\tilde{\cP}_mh,\tilde{\cP}_mf)\big)
\end{aligned}   
\end{equation}
which has the same form as the left-hand side of \eqref{commutator2} and will be the main object we need to estimate, as well as the following commutators of \(\F\) and \(\mathcal{P}\) introduced to match the terms above,
\begin{equation}\label{Delta1,12}
 \begin{aligned}
    \Delta^2_{1,1}
    :=&\sum_{\substack{m\geq N_0-1,|m-k|\leq 2N_0}}\Big(\Delta^{\text{Comm}}_{m}(\cU_{m-N_0}g,\tilde{\cP}_mh,\t{\cP}_mf)\Big).
\end{aligned}   
\end{equation}

Similarly, we list here the analogous decompositions \(\Delta_{2,1}=\Delta_{2,1}^1+\Delta_{2,1}^2\) and \(\Delta_{3,1}=\Delta_{3,1}^1+\Delta_{3,1}^2\), i.e.,
\begin{equation}\label{Delta211}
 \begin{aligned}
    \Delta^1_{2,1}
    :=&\sum_{\substack{l\geq m+N_0,|l-k|\leq 2N_0}}\big(\Delta^{\text{Freq}}_{m}(\cP_lg,\t{\cP}_lh,\t{\cP}_lf)\big),
\end{aligned}   
\end{equation} 
\begin{equation*}
 \begin{aligned}
    \Delta^2_{2,1}
    :=&\sum_{\substack{l\geq m+N_0,|l-k|\leq 2N_0}}\Big(\Delta^{\text{Comm}}_{m}(\cP_lg,\t{\cP}_lh,\t{\cP}_lf)\Big),
\end{aligned}   
\end{equation*}
\begin{equation}\label{Delta311}
 \begin{aligned}
    \Delta^1_{3,1}
    :=&\sum_{\substack{|l-m|\leq N_0,m\geq k-2N_0}}\big(\Delta^{\text{Freq}}_{m}(\cP_lg,\cU_{m+N_0}h,\cU_{m+N_0}f)\big),
\end{aligned}   
\end{equation}
and 
\begin{equation*}
 \begin{aligned}
    \Delta^2_{3,1}
    :=&\sum_{\substack{|l-m|\leq N_0, m\geq k-2N_0}}\Big(\Delta^{\text{Comm}}_{m}(\cP_lg,\cU_{m+N_0}h,\cU_{m+N_0}f)\Big).
\end{aligned}   
\end{equation*}

We summarize that
   \begin{equation}\label{Deltadecomposition}
    \begin{aligned}
   (\cP_kQ(g,h)-Q(g,\cP_kh),\F^2_j\cP_kf)_v=  \sum_{i=1}^3\Delta^1_{i,1}+\sum_{i=1}^3(\Delta_{i,1}^2+\Delta_{i,2}).
    \end{aligned}
\end{equation} 

We first deal with the main terms \(\Delta_{i,1}^1, i=1,2,3\). Noting that, ignoring the localization in phase space, they share a unified structure and we give the following lemma.

\begin{lemma}\label{mgeq0}
Recall the localized operator $\mF_b,\mF_j$ and $\mS_j$ in Definition \ref{Fj}. For smooth functions $g,h$ and $f$, $m\geq 0$, it holds that:
    \\
    \noindent $\bullet$ For $|b-j|\leq 2N_0$,
\begin{equation}\label{b=j}
\Big|\sum_{a\leq b+3N_0}\Big(\cP_kQ_{m}(\cF_ag,\cF_bh)-Q_{m}(\cF_ag,\cP_k\cF_bh),\cF_jf\Big)_v\Big|\ls2^{-sk}2^{sj}2^{(\gamma+2s)m}\|\mS_jg\|_{L^1}\|\cF_bh\|_{L^2}\|\cF_jf\|_{L^2}.
    \end{equation}
Meanwhile, for any $\tau\in[0,1]$
    \begin{equation}\label{b=j2}
        \begin{aligned}
&~\Big|\sum_{a\leq b+3N_0}\Big(\cP_kQ_{m}(\cF_ag,\cF_bh)-Q_{m}(\cF_ag,\cP_k\cF_bh),\cF_jf\Big)_v\Big|\\\ls&~\sum_{a\leq b+3N_0}\Big(2^{(2s-\tau)(j-a)}2^{(\gamma+\frac{3}{2})m+\tau(m-k)}\|\cF_ag\|_{L^2}\|\cF_bh\|_{L^2}\|\cF_jf\|_{L^2}\\
    &+C_N 2^{-k}2^{(2s-1)j}2^{-2sa}2^{(\gamma+\frac{3}{2})m}\|\cF_ag\|_{L^2}\|\cF_bh\|_{L^2}(\|\mF _j\mP_kf\|_{L^2}+2^{-jN}2^{-kN}\|f\|_{H^{-N}_{-N}})\\&+C_N 2^{-k}2^{(2s-1)j}2^{-2sa}2^{(\gamma+\frac{3}{2})m}\|\cF_ag\|_{L^2}(\|\mF _j\mP_kh\|_{L^2}+2^{-jN}2^{-kN}\|h\|_{H^{-N}_{-N}})\|\cF_j f\|_{L^2}\Big).
        \end{aligned}
    \end{equation}
 
    \noindent $\bullet$ For $b>j+2N_0$,
    \begin{equation}\label{b>j}
        \begin{aligned}
            \big|\sum_{|a-b|\leq N_0}(Q_{m}(\cF_ag,\cF_bh),\cF_jf)_v\big|\ls C_N2^{-bN}2^{-mN}\|\mF_bg\|_{L^2}\|\cF_bh\|_{L^2}\|\cF_jf\|_{L^2},~\forall N\in\N.
        \end{aligned}
    \end{equation}
  
   \noindent $\bullet$ For $b<j-2N_0$,
    \begin{equation}\label{b<j}
        \begin{aligned}
            \big|\sum_{|a-j|\leq 2N_0}(Q_{m}(\cF_ag,\cF_bh),\cF_jf)_v\big|\ls C_N2^{-jN}2^{-mN}\|\mF_jg\|_{L^2}\|\cF_bh\|_{L^2}\|\cF_jf\|_{L^2},~\forall N\in\N.
        \end{aligned}
    \end{equation}
\end{lemma}
\begin{proof}
    Noticing that \eqref{b>j} and \eqref{b<j} are straight results of Lemma \ref{M14} since they have the same structure as \(\M^4\) and \(\M^1\) in \eqref{eq m1234}, respectively. Thus we only need to give the proof of \eqref{b=j} and \eqref{b=j2}.

    \medskip
    
    First we give the proof of \eqref{b=j}. By definition, we have
    \begin{equation*}
      \begin{aligned}
        &\Big|\sum_{a\leq b+3N_0}(\cP_k Q_{m}(\cF_ag,\cF_bh)-Q_{m}(\cF_ag,\cP_k\cF_bh),\cF_jf)_v\Big|\\
        \leq&~\Big|\iint_{\R^6\times\SS^2}b(\cos\theta)\Phi_{m}^{\gamma}(|v-v_*|)\mS_jg_*\cF_bh(\cF_jf)'(\varphi(2^{-k}v')-\varphi(2^{-k}v))\d v\d v_*\d\sigma\Big|\\
        \leq&~\Big|\iint_{\R^6\times\SS^2}b(\cos\theta)\psi(2^{\frac{1}{2}(j-k)}(v'-v))\Phi_{m}^{\gamma}(|v-v_*|)\mS_jg_*\cF_bh(\cF_jf)'(\varphi(2^{-k}v')-\varphi(2^{-k}v))\d v\d v_*\d\sigma\Big|\\
        +&\Big|\iint_{\R^6\times\SS^2}b(\cos\theta)\big(1-\psi(2^{\frac{1}{2}(j-k)}(v'-v))\big)\Phi_{m}^{\gamma}(|v-v_*|)\mS_jg_*\cF_bh(\cF_jf)'(\varphi(2^{-k}v')-\varphi(2^{-k}v))\d v\d v_*\d\sigma\Big|\\
        =:&~A_1+A_2.
    \end{aligned}  
    \end{equation*}
   
    For $A_1$, let $\kappa(v)=v'+\kappa(v-v')$, by Taylor expansion
    \begin{equation}\label{Taylorv'}
        \vphi(2^{-k}v')-\vphi(2^{-k}v)=2^{-k}\nabla\vphi(2^{-k}v')\cdot(v'-v)+2^{-2k}\int_0^1(1-\kappa)(\nabla^2\vphi(2^{-k}\kappa(v))):(v'-v)\otimes(v'-v)d\kappa,
    \end{equation}
 we have
    \begin{equation*}
      \begin{aligned}
        A_1\ls
        &~\Big|\iint_{\R^6\times\SS^2}b(\cos\theta)\psi(2^{\frac{1}{2}(j-k)}(v'-v))\Phi_{m}^{\gamma}(|v-v_*|)\mS_jg_*\cF_bh(\cF_jf)'2^{-k}\nabla\varphi(2^{-k}v')\cdot(v'-v)\d v\d v_*\d\sigma\Big|\\
        &+\Big|\int_0^1\iint_{\R^6\times\SS^2}b(\cos\theta)\psi(2^{\frac{1}{2}(j-k)}(v'-v))\Phi_{m}^{\gamma}(|v-v_*|)\mS_jg_*\cF_bh(\cF_jf)'\\
        &\times~2^{-2k}(1-\kappa)\nabla^2\varphi(2^{-k}\kappa(v)):(v'-v)\otimes(v'-v)\d v\d v_*\d\sigma\d \kappa\Big|
        =:A_{1,1}+A_{1,2}.
          \end{aligned}
    \end{equation*}
   For $A_{1,1}$, using the fact (see (2.7) in \cite{he2018sharp}) that
\begin{equation}\label{symmetry}
    \begin{aligned}
        \int_{\sigma\in \SS^2,v\in \R^3}\Gamma(|v-v_*|)b(\frac{v-v_*}{|v-v_*|}\cdot\sigma)\omega(|v'-v|)(v'-v)\rho(v')\d\sigma\d v=0
    \end{aligned}
\end{equation}
    for any smooth functions $\Gamma, \omega$ and $\rho$, we have 
    \begin{equation*}
      \begin{aligned}
        |A_{1,1}|\leq&~2^{-k}\Big|\iint_{\R^6\times\SS^2}b(\cos\theta)\psi(2^{\frac{1}{2}(j-k)}(v'-v))\Phi_{m}^{\gamma}(|v-v_*|)\mS_jg_*(\cF_bh-(\cF_bh)')(\cF_jf)'\\
        &\times~\nabla\varphi(2^{-k}v')\cdot(v'-v)\d v\d v_*\d\sigma\Big|.
    \end{aligned}  
    \end{equation*}
Using Taylor expansion
$
    \cF_bh-(\cF_bh)'=\int_0^1\nabla\F_bh(\kappa(v))d\kappa\cdot (v-v'),
$
the fact $|v-v'|=|v-v_*|\sin \f \theta 2$ and Cauchy-Schwarz inequality, we can get 
    \begin{equation*}
      \begin{aligned}
        |A_{1,1}|
        \ls&~2^{-k}\Big|\int_0^1\iint_{\R^6\times\SS^2}b(\cos\theta)\sin^2\f\theta 2\psi(2^{\frac{1}{2}(j-k)}(v'-v))\Phi_{m}^{\gamma}(|v-v_*|)|\mS_jg_*|\\
        &\times~|\nabla\cF_bh(\kappa(v))(\cF_jf)'|\nabla\varphi(2^{-k}v')||v-v_*|^2\d v\d v_*\d\sigma\d \kappa\Big|\\
        \ls&~2^{-k}\Big(\int_0^1\iint_{\R^6\times\SS^2}b(\cos\theta)\sin^2\f\theta2\psi(2^{\frac{1}{2}(j-k)}(v'-v))\Phi_{m}^{\gamma}(|v-v_*|)|\mS_jg_*|\\
        &\times~|\nabla\cF_bh(\kappa(v))|^2|v-v_*|^2\d v\d v_*\d\sigma\d \kappa\Big)^{\frac{1}{2}}\Big(\int_0^1\iint_{\R^6\times\SS^2}b(\cos\theta)\sin^2\f\theta2\\
        &\times ~\psi(2^{\frac{1}{2}(j-k)}(v'-v))\Phi_{m}^{\gamma}(|v-v_*|)|\mS_jg_*|(\cF_jf)'|^2|v-v_*|^2\d v\d v_*\d\sigma\d \kappa\Big)^{\frac{1}{2}}.
    \end{aligned}  
    \end{equation*}
Using change of variables 
\begin{equation}\label{changeofvariables}
    (v_*,v)\rightarrow(v_*,u_1=v')\quad\mbox{and}\quad (v_*,v)\rightarrow(v_*,u_2=\kappa(v)),
\end{equation}
respectively. Thanks to the fact that $\Big|\f{\partial \kappa(v)}{\partial v}\Big|>\f18$(see \eqref{Jacobi}),
we derive that
    \begin{equation*}
      \begin{aligned}
        |A_{1,1}|
        \ls&~2^{-k}\Big(\int_0^1\iint_{\R^6\times\SS^2}b(\cos\theta)\sin^2\f\theta2\psi\big(2^{\frac{1}{2}(j-k)}|v-v_*|\sin \f \theta 2\big)\Phi_{m}^{\gamma}(|v-v_*|)|\mS_jg_*|\\
        &\times~|\nabla\cF_bh(u_2)|^2|u_2-v_*|^2\d u_2\d v_*\d\sigma\d \kappa\Big)^{\frac{1}{2}}\Big(\int_0^1\iint_{\R^6\times\SS^2}b(\cos\theta)\sin^2\f\theta2\\
        &\times ~\psi\big(2^{\frac{1}{2}(j-k)}|v-v_*|\sin\f\theta 2\big)\Phi_{m}^{\gamma}(|v-v_*|)|\mS_jg_*|\cF_jf(u_1)|^2|u_1-v_*|^2\d u_1\d v_*\d\sigma\d \kappa\Big)^{\frac{1}{2}},
    \end{aligned}  
    \end{equation*}
where we also use the fact $|v-v_*|\sim |u_1-v_*|\sim |u_2-v_*|\sim 2^m$. It is easy to check that
    \begin{equation*}
      \begin{aligned}
        &\int_{\sigma\in\SS^2}b(\cos\theta)\sin^2\f\theta2\psi\big(2^{\frac{1}{2}(j-k)}|v-v_*|\sin\f\theta2 \big)\d\sigma\\
        \ls& \int_{\tilde\theta\ls |v-v_*|^{-1}2^{-\f12(j-k)}}b(\cos\theta)\theta^2\sin\t\theta d\t\theta
        \ls~2^{(s-1)j}2^{(1-s)k}|v-v_*|^{2s-2},
    \end{aligned}  
    \end{equation*}
    where $\t\theta$ is defined by $\cos\t\theta=\f{u_2-v_*}{|u_2-v_*|}\cdot\sigma$ and $\theta/2\leq \t\theta\leq \theta$.  Since $|b-j|\leq 2N_0$, we can derive that
    \begin{equation*}
      \begin{aligned}
        |A_{1,1}|\ls&~2^{(1-s)k}2^{-k}2^{(s-1)j}2^{(\gamma+2s)m}\|\mS_jg\|_{L^1}\||\nabla\cF_bh|\|_{L^2}\|\cF_jf\|_{L^2}\\
        \ls&~2^{-sk}2^{sj}2^{(\gamma+2s)m}\|\mS_jg\|_{L^1}\|\cF_bh\|_{L^2}\|\cF_jf\|_{L^2}.
    \end{aligned}  
    \end{equation*}
    
    Similarly, we can get
    \begin{equation*}
     \begin{aligned}
|A_{1,2}|\ls&~2^{-2k}\Big(\int_0^1\iint_{\R^6\times\SS^2}b(\cos\theta)\sin^2\f\theta2\psi(2^{\frac{1}{2}(j-k)}(v'-v))\Phi_{m}^{\gamma}(|v-v_*|)|\mS_jg_*|\\
        &\times~|\cF_bh(\kappa(v))|^2|v-v_*|^2\d v\d v_*\d\sigma\d \kappa\Big)^{\frac{1}{2}}\Big(\int_0^1\iint_{\R^6\times\SS^2}b(\cos\theta)\sin^2\f\theta2\\
        &\times ~\psi(2^{\frac{1}{2}(j-k)}(v'-v))\Phi_{m}^{\gamma}(|v-v_*|)|\mS_jg_*|(\cF_jf)'|^2|v-v_*|^2\d v\d v_*\d\sigma\d \kappa\Big)^{\frac{1}{2}}\\
        \ls&~2^{-(1+s)k}2^{(s-1)j}2^{(\gamma+2s)m}\|\mS_jg\|_{L^1}\|\cF_bh\|_{L^2}\|\cF_jf\|_{L^2}.
    \end{aligned}   
    \end{equation*}
    Then we conclude that
    \begin{equation}\label{A1}
        \begin{aligned}
            |A_1|\ls 2^{-sk}2^{sj}2^{(\gamma+2s)m}\|\mS_jg\|_{L^1}\|\cF_bh\|_{L^2}\|\cF_jf\|_{L^2}.
        \end{aligned}
    \end{equation}

 For  $A_2$, we directly have
\begin{equation*}
     \begin{aligned}
        |A_2|\ls&~\Big|\iint_{\R^6\times\SS^2}b(\cos\theta)\big(1-\psi\big(2^{\frac{1}{2}(j-k)}|v-v_*|\sin\f \theta 2\big)\big)\Phi_{m}^{\gamma}(|v-v_*|)\mS_jg_*\cF_bh(\cF_jf)'\d v\d v_*\d\sigma\Big|.
    \end{aligned}
\end{equation*}
Observe that
    \begin{equation*}
         \begin{aligned}
        \int_{\sigma\in\SS^2}b(\cos\theta)\big(1-\psi\big(2^{\frac{1}{2}(j-k)}|v-v_*|\sin\f\theta 2\big )\big)\d\sigma
        \ls2^{-sk}2^{sj}|v-v_*|^{2s},
    \end{aligned}
    \end{equation*}
then by Cauchy-Schwarz inequality and change of variables as above, we have
    \begin{equation}\label{A2}
      \begin{aligned}
        |A_2|\ls2^{-sk}2^{sj}2^{(\gamma+2s)m}\|\mS_jg\|_{L^1}\|\cF_bh\|_{L^2}\|\cF_jf\|_{L^2}.
    \end{aligned}  
    \end{equation}
    Combining estimates \eqref{A1} and \eqref{A2},  we obtain \eqref{b=j}.

\medskip
    
Next we give the proof of \eqref{b=j2}. Using the decomposition that $b=\hat b_1+\hat b_2=b_{\theta\leq \delta_02^{a-j}}+b_{\theta>\delta_02^{a-j}}$, where $\delta_0$ is sufficiently small defined in the proof of \eqref{p=j}, we have 
    \begin{equation*}
        \begin{aligned}
            &(\cP_k Q_{m}(\cF_ag,\cF_bh)-Q_{m}(\cF_ag,\cP_k\cF_bh),\cF_jf)_v\\
            =&\iint_{\R^6\times\SS^2}\hat b_1(\cos\theta)\Phi_{m}^{\gamma}(|v-v_*|)\cF_ag_*\cF_bh(\cF_jf)'(\varphi(2^{-k}v')-\varphi(2^{-k}v))\d v\d v_*\d\sigma\\
            &+\iint_{\R^6\times\SS^2}\hat b_2(\cos\theta)\Phi_{m}^{\gamma}(|v-v_*|)\cF_ag_*\cF_bh(\cF_jf)'(\varphi(2^{-k}v')-\varphi(2^{-k}v))\d v\d v_*\d\sigma
            =:\tilde{A}_1+\tilde{A}_2.
        \end{aligned}
    \end{equation*}
    \textbf{Estimate of $\tilde{A}_1$:} we have the following decomposition,
    $$
    \begin{aligned}
        |\tilde{A}_1|\leq&~\Big|\iint_{\R^6\times\SS^2}\hat b_1(\cos\theta)\tilde{\cF}_a\Phi_{m}^{\gamma}(|v-v_*|)\cF_ag_*\cF_bh(\cF_jf)'(\varphi(2^{-k}v')-\varphi(2^{-k}v))\d v\d v_*\d\sigma\Big|\\
        +&~\Big|\iint_{\R^6\times\SS^2}\hat b_1(\cos\theta)\Phi_{m}^{\gamma}(|v-v_*|)\cF_ag_*\cF_bh\left((\left[\cF_j,\cP_k\right]f)'-\left[\cF_j,\cP_k\right]f\right)\d v\d v_*\d\sigma\Big|\\
        +&~\Big|\iint_{\R^6\times\SS^2}\hat b_1(\cos\theta)\Phi_{m}^{\gamma}(|v-v_*|)\cF_ag_*(\left[\cF_b,\cP_k\right]h)\left((\cF_jf)'-\cF_jf\right)\d v\d v_*\d\sigma\Big|\\
        +&~\Big|\iint_{\R^6\times\SS^2}\hat b_1(\cos\theta)\tilde{\cF}_a\Phi_{m}^{\gamma}(|v-v_*|)\cF_ag_*\cF_bh\left((\left[\cF_j,\cP_k\right]f)'-\left[\cF_j,\cP_k\right]f\right)\d v\d v_*\d\sigma\Big|\\
        +&~\Big|\iint_{\R^6\times\SS^2}\hat b_1(\cos\theta)\tilde{\cF}_a\Phi_{m}^{\gamma}(|v-v_*|)\cF_ag_*(\left[\cF_b,\cP_k\right]h)\left((\cF_jf)'-\cF_jf\right)\d v\d v_*\d\sigma\Big|\\
        =:&~\tilde{A}_{1,1}+\tilde{A}_{1,2}+\tilde{A}_{1,3}+\t A_{1,4}+\tilde{A}_{1,5}.
    \end{aligned}
    $$
    Note that we use $\theta\leq \delta_02^{a-j}$ to control the frequency of $\Phi_m^{\gamma}$ in
    $\tilde{A}_{1,1}$, and $\tilde{A}_{1,2}-\tilde{A}_{1,5}$ are commutator terms.
    
    \noindent\textbf{Estimate of $\tilde{A}_{1,1}$:} Let $\kappa(v)=v'+\kappa(v-v')$, by Taylor expansion
    \begin{equation}\label{Taylorv''}
        \vphi(2^{-k}v')-\vphi(2^{-k}v)=2^{-k}\nabla\vphi(2^{-k}v')\cdot(v'-v)+2^{-2k}\int_0^1(1-\kappa)(\nabla^2\vphi(2^{-k}\kappa(v))):(v'-v)\otimes(v'-v)d\kappa,
    \end{equation}
    we have
    \begin{equation*}
      \begin{aligned}
        \tilde{A}_{1,1}\ls
        &~\Big|\iint_{\R^6\times\SS^2}\hat b_1(\cos\theta)\tilde{\cF}_a\Phi_{m}^{\gamma}(|v-v_*|)\cF_ag_*\cF_bh(\cF_jf)'2^{-k}\nabla\varphi(2^{-k}v')\\
        &\cdot(v'-v)\d v\d v_*\d\sigma\Big|+\Big|\int_0^1\iint_{\R^6\times\SS^2}\hat b_1(\cos\theta)\tilde{\cF}_a\Phi_{m}^{\gamma}(|v-v_*|)\cF_ag_*\cF_bh(\cF_jf)'\\
        &\times~2^{-2k}(1-\kappa)\nabla^2\varphi(2^{-k}\kappa(v)):(v'-v)\otimes(v'-v)\d v\d v_*\d\sigma\d \kappa\Big|
        =:\tilde{A}_{1,1}^{(1)}+\tilde{A}_{1,1}^{(2)}.
          \end{aligned}
    \end{equation*}
    For $\tilde{A}_{1,1}^{(1)}$, using the fact \eqref{symmetry}, we have 
\begin{equation*}
    \begin{aligned}
       \tilde{A}_{1,1}^{(1)}\ls&~\Big|\iint_{\R^6\times\SS^2}\hat b_1(\cos\theta)\tilde{\cF}_a\Phi_{m}^{\gamma}(|v-v_*|)\cF_ag_*(\cF_bh-(\cF_bh)')(\cF_jf)'2^{-k}\nabla\varphi(2^{-k}v')(v'-v)\d v\d v_*\d\sigma\Big|\\
        \leq&~\Big|\int_0^1\iint_{\R^6\times\SS^2}\hat b_1(\cos\theta)\tilde{\cF}_a\Phi_{m}^{\gamma}(|v-v_*|)|\cF_ag_*||\nabla\cF_bh(\kappa(v))||v'-v||(\cF_jf)'|2^{-k}|\nabla\varphi(2^{-k}v')\\
        &\cdot(v'-v)|\d v\d v_*\d\sigma\d\kappa\Big|\\
   \end{aligned}
\end{equation*}
By Cauchy-Schwarz inequality, using the same change of variables \eqref{changeofvariables} and the fact \eqref{tFpphigammak}, \eqref{Jacobi},
we derive that
\begin{equation*}
    \begin{aligned}
    \tilde{A}_{1,1}^{(1)}\ls&~2^{(2s-2)(j-a)}2^{-k}\||\cdot|^2\tilde{\cF}_a\Phi_{m}^{\gamma}\|_{L^2}\|\cF_ag_*\|_{L^2}\|\nabla\cF_bh\|_{L^2}\|\cF_jf\|_{L^2}\\
    \ls&~C_N2^{(2s-1)j}2^{(2-2s-N)a}2^{(\gamma+\frac{3}{2}-N)m}2^{-k}\|\cF_ag_*\|_{L^2}\|\cF_bh\|_{L^2}\|\cF_jf\|_{L^2}.\\
\end{aligned}
\end{equation*}
As for $\tilde{A}_{1,1}^{(2)}$, we use similar change of variables with $\tilde{A}_{1,1}^{(1)}$ and Cauchy-Shwarz inequality,
\begin{equation*}
    \begin{aligned}
    \tilde{A}_{1,1}^{(2)}
    \ls&~C_N2^{(2s-2)j}2^{(2-2s-N)a}2^{(\gamma+\frac{3}{2}-N)m}2^{-2k}\|\cF_ag\|_{L^2}\|\cF_bh\|_{L^2}\|\cF_jf\|_{L^2}.\\
\end{aligned}
\end{equation*}
Combine $\tilde{A}_{1,1}^{(1)}$ and $\tilde{A}_{1,1}^{(2)}$, we have
\begin{equation*}
    \begin{aligned}
    \tilde{A}_{1,1}\ls C_N2^{(2s-1)j}2^{(2-2s-N)a}2^{(\gamma+\frac{3}{2}-N)m}2^{-k}\|\cF_ag\|_{L^2}\|\cF_bh\|_{L^2}\|\cF_jf\|_{L^2}.\\
\end{aligned}
\end{equation*}
Now we estimate $\tilde{A}_{1,2}$ and $\tilde{A}_{1,3}$, using Lemma \ref{M14}, Lemma \ref{M23} and Lemma \ref{PFcommutator}, we have
\begin{equation*}
    \begin{aligned}
    \tilde{A}_{1,2}+\tilde{A}_{1,3}\ls \&C_N 2^{-k}2^{(2s-1)j}2^{-2sa}2^{(\gamma+\frac{3}{2})m}\|\cF_ag\|_{L^2}\|\cF_bh\|_{L^2}(\|\mF _j\mP_kf\|_{L^2}+2^{-jN}2^{-kN}\|f\|_{H^{-N}_{-N}})\\
    +&C_N2^{-k}2^{(2s-1)j}2^{-2sa}2^{(\gamma+\frac{3}{2})m}\|\cF_ag\|_{L^2}(\|\mF _j\mP_kh\|_{L^2}+2^{-jN}2^{-kN}\|h\|_{H^{-N}_{-N}})\|\cF_jf\|_{L^2}.
\end{aligned}
\end{equation*}
Note that $\tilde{A}_{1,4}$ and $\tilde{A}_{1,5}$ has the same structure with $\tilde{A}_{1,2}$ and $\tilde{A}_{1,3}$, we can replace $\||\cdot|^2\Phi^{\gamma}_{m}\|_{L^2}$ by $\||\cdot|^2\cF_a\Phi_{m}^{\gamma}\|_{L^2}$ in the proof of \eqref{ineq k=-1 M3}, and use \eqref{tFpphigammak} so that \\
\begin{equation*}
    \begin{aligned}
    \tilde{A}_{1,4}+\tilde{A}_{1,5}&\ls2^{(2s-1)j}2^{-2sa}\||\cdot|^2\cF_a\Phi_{m}^{\gamma}\|_{L^2}\|\cF_ag\|_{L^2}\\&\times(\|\cF_bh\|_{L^2}\|[\cF_j,\cP_k]f\|_{L^2}+\|[\cF_b,\cP_k]h\|_{L^2}\|\cF_j f\|_{L^2})\\
    &\ls C_N 2^{-k}2^{(2s-1)j}2^{-2sa}2^{(\gamma+\frac{3}{2})m}\|\cF_ag\|_{L^2}\|\cF_bh\|_{L^2}(\|\mF _j\mP_kf\|_{L^2}+2^{-jN}2^{-kN}\|f\|_{H^{-N}_{-N}})\\
    &+C_N2^{-k}2^{(2s-1)j}2^{-2sa}2^{(\gamma+\frac{3}{2})m}\|\cF_ag\|_{L^2}(\|\mF _j\mP_kh\|_{L^2}+2^{-jN}2^{-kN}\|h\|_{H^{-N}_{-N}})\|\cF_jf\|_{L^2}.
\end{aligned}
\end{equation*}
Combine the estimates of $\tilde{A}_{1,1}$ to $\tilde{A}_{1,5}$, we have
\begin{equation}\label{tildeA1}
\begin{aligned}
    \tilde{A}_{1}\ls &~C_N 2^{-k}2^{(2s-1)j}2^{-2sa}2^{(\gamma+\frac{3}{2})m}\|\cF_ag\|_{L^2}\|\cF_bh\|_{L^2}(\|\mF _j\mP_kf\|_{L^2}+2^{-jN}2^{-kN}\|f\|_{H^{-N}_{-N}})\\+&C_N 2^{-k}2^{(2s-1)j}2^{-2sa}2^{(\gamma+\frac{3}{2})m}\|\cF_ag\|_{L^2}(\|\mF _j\mP_kh\|_{L^2}+2^{-jN}2^{-kN}\|h\|_{H^{-N}_{-N}})\|\cF_j f\|_{L^2}.
\end{aligned}    
\end{equation}
\textbf{Estimate of $\tilde{A}_2$:} Recall that
\begin{equation*}
    \begin{aligned}
        \tilde{A}_2=\iint_{\R^6\times\SS^2}\hat b_2(\cos\theta)\Phi_{m}^{\gamma}(|v-v_*|)\cF_ag_*\cF_bh(\cF_jf)'(\varphi(2^{-k}v')-\varphi(2^{-k}v))\d v\d v_*\d\sigma.
    \end{aligned}
\end{equation*}
    Note that $|\varphi(2^{-k}v')-\varphi(2^{-k}v)|\ls 2^{-k}\theta|v-v_*|$ and following the same steps during the proof of \eqref{ineq kgeq0 M2}, we have
\begin{equation*}
    \begin{aligned}
        |\tilde{A}_2|\ls2^{(2s-1)(j-a)-k}\|\Phi_{m}^{\gamma+1}\|_{L^2}\|\cF_ag\|_{L^2}\|\cF_bh\|_{L^2}\|\cF_jf\|_{L^2}
        \ls2^{(2s-1)(j-a)+(\gamma+\frac{5}{2})m-k}\|\cF_ag\|_{L^2}\|\cF_bh\|_{L^2}\|\cF_jf\|_{L^2}.
    \end{aligned}
\end{equation*}
    If using $|\varphi(2^{-k}v')-\varphi(2^{-k}v)|\ls 1$, we can get another bound
\begin{equation*}
    \begin{aligned}
        |\tilde{A}_2|\ls 2^{2s(j-a)}\|\Phi_{m}^{\gamma}\|_{L^2}\|\cF_ag\|_{L^2}\|\cF_bh\|_{L^2}\|\cF_jf\|_{L^2}
        \ls2^{2s(j-a)}2^{(\gamma+\frac{3}{2})m}\|\cF_ag\|_{L^2}\|\cF_bh\|_{L^2}\|\cF_jf\|_{L^2}.
    \end{aligned}
\end{equation*}
Therefore, by interpolation, we obtain
    \begin{equation}\label{tildeA2}
    \begin{aligned}
        |\tilde{A}_2|\ls2^{(2s-\tau)(j-a)}2^{(\gamma+\frac{3}{2})m+\tau(m-k)}\|\cF_ag\|_{L^2}\|\cF_bh\|_{L^2}\|\cF_jf\|_{L^2},\quad \forall \tau\in[0,1].
    \end{aligned}    
    \end{equation}
    Combining \eqref{tildeA1} and \eqref{tildeA2}, we obtain \eqref{b=j2}.

    We complete the proof of this lemma.
\end{proof}

It remains to treat the case \(m=-1\), for which we have
\begin{lemma}\label{m=-1}
   Recall $\mF_j,\mF_b$ and $\mS_j$ defined in Definition \ref{Fj} and the constant $\a>\f12$. For any smooth functions $g,h$ and $f$, it holds that
    
    \noindent $\bullet$  For $|b-j|\leq 2N_0$,
    \begin{equation}\label{b=j,-1}
        \big|\sum_{a\leq b+3N_0}(\cP_kQ_{-1}(\cF_ag,\cF_bh)-Q_{-1}(\cF_ag,\cP_k\cF_bh),\cF_jf)_v\big|\ls\\
        2^{-k}2^{(2s-1)j}\|\mS_jg\|_{L^2}\|\cF_bh\|_{L^2}\|\cF_jf\|_{L^2}.
    \end{equation}
   
     \noindent $\bullet$ For $b>j+2N_0$,
    \begin{equation}\label{b>j,-1}
        \begin{aligned}
            &\big|\sum_{|a-b|\leq N_0}(\cP_k Q_{-1}(\cF_ag,\cF_bh)-Q_{-1}(\cF_ag,\cP_k\cF_bh),\cF_jf)_v\big|
            \ls C_N2^{-k}(2^{-(\gamma+\f52)b}+2^{-\zeta b}b^\a2^{(2s-1)j})\|\mF_bg\|_{L^2}\\
            &\times(\|\mF_bh\|_{L^2}+2^{-bN-kN}\|h\|_{H^{-N}_{-N}})(\|\mF_j f\|_{L^2}+2^{-jN-kN}\|f\|_{H^{-N}_{-N}}),\quad\forall N\in\N.
        \end{aligned}
    \end{equation}

     \noindent $\bullet$ For $b<j-2N_0$,
    \begin{equation}\label{b<j,-1}
        \begin{aligned}
              &\big|\sum_{|a-j|\leq 2N_0}(\cP_k Q_{-1}(\cF_ag,\cF_bh)-Q_{-1}(\cF_ag,\cP_k\cF_bh),\cF_jf)_v\big|
            \ls C_N2^{-k}(2^{-(\gamma+\f52)j}+2^{\f12 b}b^\a2^{-(\gamma+3)j})\|\mF_jg\|_{L^2}\\
            &\times(\|\mF_bh\|_{L^2}+2^{-bN-kN}\|h\|_{H^{-N}_{-N}})(\|\mF_j f\|_{L^2}+2^{-jN-kN}\|f\|_{H^{-N}_{-N}}),\quad\forall N\in\N.
        \end{aligned}
    \end{equation}
\end{lemma}
\begin{proof}
We give the proof in several steps.

     \smallskip
     
\textbf{Step 1}: Proof of \eqref{b=j,-1}. By definition, we have
    \begin{equation*}
      \begin{aligned}
        &\sum_{a\leq b+3N_0}(\cP_kQ_{-1}(\cF_ag,\cF_bh)-Q_{-1}(\cF_ag,\cP_k\cF_bh),\cF_jf)_v\\
        =&\iint_{\R^6\times\SS^2}b(\cos\theta)\psi(2^{j}(v'-v))\Phi_{-1}^{\gamma}(|v-v_*|)\mS_jg_*\cF_bh(\cF_jf)'(\varphi(2^{-k}v')-\varphi(2^{-k}v))\d v\d v_*\d\sigma\\
        +&\iint_{\R^6\times\SS^2}b(\cos\theta)(1-\psi(2^{j}(v'-v)))\Phi_{-1}^{\gamma}(|v-v_*|)\mS_jg_*\cF_bh(\cF_jf)'(\varphi(2^{-k}v')-\varphi(2^{-k}v))\d v\d v_*\d\sigma\\
        =:&~G_1+G_2.
    \end{aligned}  
    \end{equation*}

  For $G_1$, by the same argument as the estimate of $A_1$ in Lemma \ref{mgeq0},
    \begin{equation*}
      \begin{aligned}
        G_1=&\iint_{\R^6\times\SS^2}b(\cos\theta)\psi(2^{j}(v'-v))\Phi_{-1}^{\gamma}(|v-v_*|)\mS_jg_*\cF_bh(\cF_jf)'2^{-k}\nabla\varphi(2^{-k}v')\\
        &\cdot(v'-v)\d v\d v_*\d\sigma
        +\int_{0}^{1}\iint_{\R^6\times\SS^2}b(\cos\theta)\psi(2^{j}(v'-v))\Phi_{-1}^{\gamma}(|v-v_*|)\mS_jg_*\cF_bh(\cF_jf)'\\
        &\times~2^{-2k}\nabla^2\varphi(2^{-k}\kappa(v)):(v'-v)\otimes(v'-v)\d v\d v_*\d\sigma\d\kappa
        =:G_{1,1}+G_{1,2}.
    \end{aligned}   
    \end{equation*}
    
    For $G_{1,1}$, similar to the estimate of $A_{1,1}$ in \eqref{b=j}, we have
    \begin{equation*}
            \begin{aligned}
        &|G_{1,1}|\leq2^{-k}\Big|\iint_{\R^6\times\SS^2}b(\cos\theta)\psi(2^{j}(v'-v))\Phi_{-1}^{\gamma}(|v-v_*|)\mS_jg_*(\cF_bh-(\cF_bh)')(\cF_jf)'\\
        &\times\nabla\varphi(2^{-k}v')\cdot(v'-v)\d v\d v_*\d\sigma\Big|
        =\Big|\int_{0}^{1}\iint_{\R^6\times\SS^2}b(\cos\theta)\psi(2^{j}(v'-v))\Phi_{-1}^{\gamma}(|v-v_*|)\mS_jg_*\\
        &\times\nabla\cF_bh(\kappa(v))\cdot(v'-v)(\cF_jf)'2^{-k}\nabla\varphi(2^{-k}v')\cdot(v'-v)\d v\d v_*\d\sigma\d\kappa\Big|\\
        \ls&~2^{-k}\bigg(\iint_{\R^6\times\SS^2}b(\cos\theta)\sin^2\f\theta 2\psi\big(2^{j}|v-v_*|\sin\f\theta 2\big)|\Phi_{-1}^{\gamma}(|u_1-v_*|)|^2|\cF_jf(u_1)|^2|u_1-v_*|^{2+2s}\d u_1\d v_*\d\sigma\bigg)^{\frac{1}{2}}\\
        &\times\bigg(\int_{0}^{1}\iint_{\R^6\times\SS^2}b(\cos\theta)\sin^2\f\theta 2\psi\big(2^{j}|v-v_*|\sin\f\theta 2\big)|\mS_jg_*|^2|\nabla\cF_bh(u_2)|^2|u_2-v_*|^{2-2s}\d u_2\d v_*\d\sigma\d\kappa\bigg)^{\frac{1}{2}}\\
        \ls&~2^{-k}2^{(2s-1)j}\|\mS_jg\|_{L^2}\|\cF_bh\|_{L^2}\|\cF_jf\|_{L^2},
    \end{aligned}
    \end{equation*}
    where we use change of variables $(v_*,v')\rightarrow (v_*,u_1)$, $(v_*,\kappa(v))\rightarrow (v_*,u_2)$, $|\kappa(v)-v_*|\sim|v-v_*|$ and 
    \begin{equation*}
      \begin{aligned}
        &\int_{\SS^2}b(\cos\theta)\sin^2\f\theta 2\psi\big(2^{j}|v-v_*|\sin\f\theta 2\big)\d\sigma\ls 2^{(2s-2)j}|v-v_*|^{2s-2},\\
        &\int_{\R^3}|\Phi_{-1}^{\gamma}(|u|)|^2|u|^{4s}\d u\ls r^{2(\gamma+2s+\frac{3}{2})}\big|_{0}^{1}=1,\quad \mbox{since}\quad \gamma+2s>-1.
    \end{aligned}  
    \end{equation*}
    
    For $G_{1,2}$, the same argument leads to
    \begin{equation*}
      \begin{aligned}
        |G_{1,2}|\ls&~2^{-2k}\bigg(\iint_{\R^6\times\SS^2}b(\cos\theta)\sin^2\f\theta 2\psi(2^{j}(v'-v))|\Phi_{-1}^{\gamma}(|v-v_*|)|^2|(\cF_jf)'|^2|v-v_*|^{2+2s}\d v\d v_*\d\sigma\bigg)^{\frac{1}{2}}\\
        \times&\bigg(\int_{\SS^2}\iint_{\R^6}b(\cos\theta)\sin^2\f\theta 2\psi(2^{j}(v'-v))|\cF_ag_*|^2|\cF_bh|^2|v-v_*|^{2-2s}\d v\d v_*\d\sigma\bigg)^{\frac{1}{2}}\\
        \ls&~2^{-2k}2^{(2s-2)j}\|\mS_jg\|_{L^2}\|\cF_bh\|_{L^2}\|\cF_jf\|_{L^2}.
    \end{aligned}.   
    \end{equation*}
   Then using change of variables $(v,v_*)\mapsto(v',v_*)$ and $|v'-v_*|\sim|v-v_*|$,we conclude that
   \begin{equation}\label{G1new}
       |G_1|\ls 2^{-k}2^{(2s-1)j}\|\mS_jg\|_{L^2}\|\cF_bh\|_{L^2}\|\cF_jf\|_{L^2}.
   \end{equation}
    
    For $G_2$, similar to the estimate of $A_2$ in the proof of \eqref{b=j}, using $|\varphi(2^{-k}v')-\varphi(2^{-k}v)|\ls2^{-k}|v-v'|\ls 2^{-k}|v-v_*|\sin \theta $ and Cauchy-Schwarz inequality, we have 
    \begin{equation*}
        \begin{aligned}
    |G_2|\ls&~2^{-k}\bigg(\iint_{\R^6\times\SS^2}| b(\cos\theta)\sin\theta ||1-\psi(2^{j}(v'-v))||\Phi_{-1}^{\gamma}(|v-v_*|)|^2|(\cF_jf)'|^2|v-v_*|^{1+2s}\d v\d v_*\d\sigma\bigg)^{\frac{1}{2}}\\\times&\bigg(\iint_{\R^6\times\SS^2}|b(\cos\theta)\sin\theta||1-\psi(2^{j}(v'-v))||\mS_jg_*|^2|\cF_bh|^2|v-v_*|^{1-2s}\d v\d v_*\d\sigma\bigg)^{\frac{1}{2}}.
    \end{aligned}
    \end{equation*}
    It is easy to check that
    \begin{equation*}
      \begin{aligned}
        &\int_{\SS^2}| b(\cos\theta)\sin\theta||1-\psi\big(2^{j}|u-v_*|\sin\f\theta2 \big)|\d\sigma\ls2^{(2s-1)j}|u-v_*|^{2s-1},\\
        &\int_{\R^3}|\Phi_{-1}^{\gamma}(|\t u|)|^2|\t u|^{4s}\d \t u\ls r^{2(\gamma+2s+\frac{3}{2})}\big|_{0}^{1}=1.
    \end{aligned}  
    \end{equation*}
    Thus by change of variables $(v,v_*)\mapsto(u=v',v_*)$ and $|v'-v_*|\sim|v-v_*|$, we obtain
    \begin{equation}\label{G2a}
       \begin{aligned}
        |G_2|\ls2^{-k}2^{(2s-1)j}\|\mS_jg\|_{L^2}\|\cF_bh\|_{L^2}\|\cF_jf\|_{L^2}.
    \end{aligned} 
    \end{equation}
 Combining \eqref{G1new} and \eqref{G2a}, we get the desired result \eqref{b=j,-1}.   

    \medskip

     \textbf{Step 2}: Proof of \eqref{b>j,-1}. Observing that it has a structure similar to \(\mathcal{M}^4_{-1,b,j},b>j+2N_0\) (see \eqref{eq m1234}), and in order to obtain the optimal estimate, we make the following decomposition.
    \begin{equation*}
       \begin{aligned}
        &\Big|\sum_{|a-b|\leq N_0}(\cP_k Q_{-1}(\cF_ag,\cF_bh)-Q_{-1}(\cF_ag,\cP_k\cF_bh),\cF_jf)_v\Big|\\
        \leq&~\Big|\iint_{\R^6\times\SS^2}b(\cos\theta)\tilde{\cF}_b\Phi_{-1}^{\gamma}(|v-v_*|)\tF_bg_*\cF_bh(\cF_jf)'(\varphi(2^{-k}v')-\varphi(2^{-k}v))\d v\d v_*\d\sigma\Big|\\
        +&~\Big|\iint_{\R^6\times\SS^2}b(\cos\theta)\tilde{\cF}_b\Phi_{-1}^{\gamma}(|v-v_*|)\tF_bg_*\cF_bh\left((\left[\cF_j,\cP_k\right]f)'-\left[\cF_j,\cP_k\right]f\right)\d v\d v_*\d\sigma\Big|\\
        +&~\Big|\iint_{\R^6\times\SS^2}b(\cos\theta)\tilde{\cF}_b\Phi_{-1}^{\gamma}(|v-v_*|)\tF_bg_*(\left[\cF_b,\cP_k\right]h)\left((\cF_jf)'-\cF_jf\right)\d v\d v_*\d\sigma\Big|\\
        +&~|(Q_{-1}(\tF_bg,\cF_bh),\left[\cP_k,\cF_j\right]f)_v|+|(Q_{-1}(\tF_bg,\left[\cF_b,\cP_k\right]h),\cF_jf)_v|\\
        =:&~H_1+H_2+H_3+H_4+H_5.
    \end{aligned} 
    \end{equation*}
   
    \noindent $\bullet$ Estimate of $H_{1}$.  By Taylor's expansion (compared with \eqref{Taylorv'}, the expansion here is performed at the point \(v\))
     \begin{equation}\label{Taylorv}
        \vphi(2^{-k}v')-\vphi(2^{-k}v)=2^{-k}\nabla\vphi(2^{-k}v)\cdot(v'-v)+2^{-2k}\int_0^1(1-\kappa)(\nabla^2\vphi(2^{-k}\t\kappa(v))):(v'-v)\otimes(v'-v)d\kappa,
    \end{equation}
    where $\t\kappa(v)=v+\kappa(v'-v)$. Then we have
    \begin{equation*}
       \begin{aligned}
        H_1\leq&\Big|\iint_{\R^6\times\SS^2}b(\cos\theta)\tilde{\cF}_b\Phi_{-1}^{\gamma}(|v-v_*|)\tF_bg_*\cF_bh(\cF_jf)'2^{-k}\nabla\varphi(2^{-k}v)\\
        &\cdot(v'-v)\d v\d v_*\d\sigma\Big|
        +\Big|\int_0^1\iint_{\R^6\times\SS^2}b(\cos\theta)\tilde{\cF}_b\Phi_{-1}^{\gamma}(|v-v_*|)\tF_bg_*\cF_bh(\cF_jf)'\\
        &\times(1-\kappa)2^{-2k}\nabla^2\varphi(2^{-k}\t\kappa(v)):(v'-v)\otimes(v'-v)\d v\d v_*\d\sigma\d \kappa\Big|=:H_{1,1}+H_{1,2}.
    \end{aligned} 
    \end{equation*}
    For $H_{1,1}$, we further have
    \begin{equation*}
         \begin{aligned}
        H_{1,1}\ls&~2^{-k}\Big|\iint_{\R^6\times\SS^2}b(\cos\theta)\tilde{\cF}_b\Phi^{\gamma}_{-1}(|v-v_*|)\tF_bg_*\cF_bh\left[(\cF_jf)'-\cF_jf\right]\nabla\varphi(2^{-k}v)\cdot(v'-v)\d v\d v_*\d\sigma\Big|\\
        +&2^{-k}\Big|\iint_{\R^6\times\SS^2}b(\cos\theta)\tilde{\cF}_b\Phi^{\gamma}_{-1}(|v-v_*|)\tF_bg_*\cF_bh\cF_jf\nabla\varphi(2^{-k}v)\cdot(v'-v)\d v\d v_*\d\sigma\Big|
        =:H_{1,1}^{(1)}+H_{1,1}^{(2)}.
    \end{aligned}   
    \end{equation*}
By the techniques developed above(change of variables), the desired bound follows readily.  
    \begin{equation*}
      \begin{aligned}
H_{1,1}^{(1)}\ls&~2^{-k}\Big|\int_0^1\iint_{\R^6\times\SS^2}b(\cos\theta)\tilde{\cF}_b\Phi^{\gamma}_{-1}(|v-v_*|)\tF_bg_*\cF_bh\nabla\cF_jf(\tilde{\kappa}(v))\cdot (v'-v)\\&\times\nabla\varphi(2^{-k}v)\cdot(v'-v)\d v\d v_*\d\sigma\d \kappa\Big|\\
        \ls&~2^{-k}\Big(\iint_{\R^6}\big(|v-v_*|^2\tilde{\cF}_b\Phi^{\gamma}_{-1}(|v-v_*|)\big)^2(\cF_bh)^2\d v\d v_*\Big)^{\frac{1}{2}}\Big(\int_{\R^6}(\tF_bg_*)^2|\nabla\cF_jf(\tilde{\kappa}(v))|^2\d v\d v_*\Big)^{\frac{1}{2}}\\
        \ls&~2^{-k}\||\cdot|^2\t\F_b\Phi_{-1}^{\gamma}\|_{L^2}\|\tF_bg\|_{L^2}\|\cF_bh\|_{L^2}\|\nabla\cF_jf\|_{L^2}
        \ls2^{-k}2^{-(\gamma+\frac{7}{2})b}2^j\|\tF_bg\|_{L^2}\|\cF_bh\|_{L^2}\|\cF_jf\|_{L^2},
    \end{aligned}   
    \end{equation*}
where we use the fact $b>j+2N_0$ and 
\begin{equation*}
    \int_{\R^3} \big(|u|^2\t\F_b\Phi^{\gamma}_{-1}(u)\big)^2du=\int_{\R^3}|\nabla^2\widehat{\t\F_b\Phi^{\gamma}_{-1}}|^2 d\xi\ls \int_{|\xi|\sim 2^b}\<\xi\>^{-2(5+\gamma)}\d\xi\ls 2^{-2(\gamma+\f72)b}.
\end{equation*}
    For $H_{1,1}^{(2)}$, by \eqref{symmetric}, we have
    $$
    \begin{aligned}
        H_{1,1}^{(2)}\ls&~2^{-k}\Big(\iint_{\R^6}\left(|v-v_*|\tilde{\cF}_b\Phi^{\gamma}_{-1}(|v-v_*|)|\right)^2|\cF_bh|^2\d v\d v_*\Big)^{\frac{1}{2}}\Big(\iint_{\R^6}|\tF_bg_*|^2|\cF_jf(v)|^2\d v\d v_*\Big)^{\f12}\\
        \ls&~2^{-k}\||\nabla\varphi_b\widehat{\Phi^{\gamma}_{-1}}|\|_{L^2}\|\cF_bh\|_{L^2}\|\tF_bg\|_{L^2}\|\cF_jf\|_{L^2}
        \ls2^{-k}2^{-(\gamma+\frac{5}{2})b}\|\tF_bg\|_{L^2}\|\cF_bh\|_{L^2}\|\cF_jf\|_{L^2}.
    \end{aligned}
    $$
    Therefore, we obtain 
    \begin{equation}\label{H11}
        H_{1,1}\ls 2^{-k}2^{-(\gamma+\frac{5}{2})b}\|\tF_bg\|_{L^2}\|\cF_bh\|_{L^2}\|\cF_jf\|_{L^2}.
    \end{equation}
    
    We can get the estimate of $H_{1,2}$ using the same method in the estimate of $H_{1,1}^{(1)}$ ,\\so that we get
    \begin{equation}\label{H12}
         \begin{aligned}
        H_{1,2}\ls2^{-2k}2^{-(\gamma+\frac{7}{2})b}\|\tF_bg\|_{L^2}\|\cF_bh\|_{L^2}\|\cF_jf\|_{L^2}.
    \end{aligned}
    \end{equation}
   
    Combine \eqref{H11} and \eqref{H12}, we have
    \begin{equation}\label{H1}
       \begin{aligned}
        H_{1}\ls2^{-k}2^{-(\gamma+\frac{5}{2})b}\|\tF_bg\|_{L^2}\|\cF_bh\|_{L^2}\|\cF_jf\|_{L^2}.
    \end{aligned}  
    \end{equation}

     \noindent $\bullet$ Estimate of $H_{2}$ and $H_3$. Recall that
    \begin{equation*}
        \begin{aligned}
        H_2=&~\Big|\iint_{\R^6\times\SS^2}b(\cos\theta)\tilde{\cF}_b\Phi_{-1}^{\gamma}(|v-v_*|)\tF_bg_*\cF_bh\left((\left[\cF_j,\cP_k\right]f)'-\left[\cF_j,\cP_k\right]f\right)\d v\d v_*\d\sigma\Big|,\\
        H_3=&~\Big|\iint_{\R^6\times\SS^2}b(\cos\theta)\tilde{\cF}_b\Phi_{-1}^{\gamma}(|v-v_*|)\tF_bg_*\left[\cF_b,\cP_k\right]h\left((\cF_jf)'-\cF_jf\right)\d v\d v_*\d\sigma\Big|.
    \end{aligned}
    \end{equation*}
 For $H_{2}$ and $H_3$, we can apply \eqref{Taylorv} and \eqref{symmetric} to functions $[\F_j,\cP_k]f$ and $\F_j f$, then using Cauchy-Schwarz inequality, change of variables \eqref{changeofvariables} and \eqref{pkfj}, we omit the details and give the bound as follows
    \begin{equation}\label{H23}
         \begin{aligned}
        |H_{2}|\ls &~C_N2^{-k}2^{-(\gamma+\f52) b}\|\tF_bg\|_{L^2}\|\cF_bh\|_{L^2}\big(\|\mF_j\mP_kf\|_{L^2}+2^{-jN-kN}\|f\|_{H^{-N}_{-N}}\big),\\
        |H_3|\ls&~C_N2^{-k}2^{-(\gamma+\f72)b}2^{j}\|\tF_bg\|_{L^2}\big(\|\mF_b\mP_k h\|_{L^2}+2^{-Nb}2^{-Nk}\|h\|_{H^{-N}_{-N}}\big)\|\cF_jf\|_{L^2}.
    \end{aligned}
    \end{equation}

    \noindent $\bullet$  Estimate of $H_{4}$ and $H_5$. Recall
    \begin{equation*}
        H_4=|(Q_{-1}(\cF_ag,\cF_bh),\left[\cP_k,\cF_j\right]f)_v|,~
        H_5=|(Q_{-1}(\cF_ag,\left[\cF_b,\cP_k\right]h),\cF_jf)_v|.
    \end{equation*}
    In view of \eqref{Q_-1} or \eqref{ineq k=-1 M4} since $H_4$ and $H_5$ has similar structure with $\M_{-1,p,m}^4$, we have
    \begin{equation}\label{H45}
        \begin{aligned}
        |H_{4}|\ls &~2^{-\zeta b}b^{\a} \|\tF_bg\|_{L^2}\|\cF_bh\|_{L^2}\|\left[\cF_j,\cP_k\right]f\|_{H^{2s}}\\
        \ls&~C_N2^{-\zeta b}b^{\a}2^{-k}2^{(2s-1)j}\|\tF_bg\|_{L^2}\|\cF_bh\|_{L^2}\big(\|\mF_j\mP_k f\|_{L^2}+2^{-Nj}2^{-Nk}\|f\|_{H^{-N}_{-N}}\big),\\
        |H_{5}|\ls &~2^{-\zeta b}2^{2sj}b^\a\|\tF_bg\|_{L^2}\|\left[\cF_b,\cP_k\right]h\|_{L^2}\|\cF_jf\|_{L^2}\\
        \ls &C_N2^{-(\zeta+1)b}2^{-k}2^{2sj}b^\a\|\tF_bg\|_{L^2}\left(\|\mF_b\mP_k h\|_{L^2}+2^{-Nb}2^{-Nk}\|h\|_{H^{-N}_{-N}}\right)\|\cF_jf\|_{L^2}.
    \end{aligned} 
    \end{equation}

    Combining the estimates \eqref{H1}, \eqref{H23} and \eqref{H45} and noticing that
    $b>j+2N_0$, we conclude \eqref{b>j,-1}.
    
     \textbf{Step 3}: Proof of \eqref{b<j,-1}. It is similar to the proof of \eqref{b>j,-1}, the difference is that it has a structure similar to \(\M^1_{-1,j,a}\) instead of \(\M^4_{-1,p,m}\) (see \eqref{eq m1234}), so that we can use the estimate \eqref{ineq k=-1 M1}. We therefore omit the repetition of the argument and obtain \eqref{b<j,-1}.

     We complete the proof of this lemma.
\end{proof}

Combining the above lemmas, we obtain
\begin{lemma}\label{Delta1}
We have the following bound 
\begin{equation*}
   \sum_{i=1}^3\Delta_{i,1}^1 \leq \sum\limits_{i=1}^{3}\mC_i(g,h,f),
\end{equation*}
where $\mC_i(g,h,f),i=1,2,3$ are given in \eqref{bound1ofD}, \eqref{bound2ofD} and  \eqref{bound3ofD}.
\end{lemma}
\begin{proof}
    We take $\Delta_{2,1}^1$ as a typical one. Making the substitution \((g,h,f)\to(\cP_lg,\tilde{\mathcal{P}}_l h,\F_j\tilde{\mathcal{P}}_l\mathcal{P}_k f)\) and noticing that $l\geq m+N_0$, $|l-k|\leq 2N_0$ and $\gamma+2s<0$, we can derive from \eqref{Delta211}, Lemma \ref{mgeq0} and Lemma \ref{m=-1} that
\begin{equation}\label{bound2ofD}
    \begin{aligned}
    \Delta_{2,1}^1\ls&\min\big\{\Delta_{2,1}^{1,(1)},\Delta_{2,1}^{1,(2)}\big\}+2^{-k}2^{(2s-1)j} \|\mS_j\mP_{k}g\|_{L^2}\|\mF_j\mP_kh\|_{L^2}\|\mF_j\mP_kf\|_{L^2}\\
            &+C_N\sum_{b>j+2N_0}2^{-k}(2^{-(\gamma+\f52)b}+2^{-\zeta b}b^\a 2^{(2s-1)j})\|\mF_b\mP_{k}g\|_{L^2}(\|\mF_b\mP_kh\|_{L^2}+2^{-bN-kN}\|h\|_{H^{-N}_{-N}})\\&\times(\|\mF_j\mP_kf\|_{L^2}
            +2^{-jN-kN}\|f\|_{H^{-N}_{-N}})+C_N\sum_{b<j-2N_0}2^{-k}(2^{-(\gamma+\f52)j}+2^{\f12 b}b^\a 2^{-(\gamma+3)j})\\&\|\mF_j\mP_{k}g\|_{L^2}(\|\mF_b\mP_kh\|_{L^2}
            +2^{-bN-kN}\|h\|_{H^{-N}_{-N}})(\|\mF_j\mP_kf\|_{L^2}+2^{-jN-kN}\|f\|_{H^{-N}_{-N}})=:~\mC_2(g,h,f),
    \end{aligned}
\end{equation}
where
\begin{equation*}
    \begin{aligned}
        &\Delta_{2,1}^{1,(1)}=2^{-sk}2^{sj}\|\mS_j\mP_{k}g\|_{L^1}\|\mF_j\mP_kh\|_{L^2}\|\mF_j\mP_kf\|_{L^2}+C_N\sum_{b>j+2N_0}2^{-bN}\|\mF_b\mP_{k}g\|_{L^1}\|\mF_b\mP_kh\|_{L^2}\|\mF_j\mP_kf\|_{L^2}\\
            &\,\,+C_N\sum_{b<j-2N_0}2^{-jN}\|\mF_j\mP_{k}g\|_{L^1}\|\mF_b\mP_kh\|_{L^2}\|\mF_j\mP_kf\|_{L^2},\\
        &\Delta_{2,1}^{1,(2)}=\sum_{a\leq j+5N_0}2^{(2s-1)(j-a)}2^{-k}\sum_{m\leq k+3N_0}2^{(\gamma+\f52)m}\|\mF_a\mP_kg\|_{L^2}\|\mF _j\mP_kh\|_{L^2}\|\mF _j\mP_kf\|_{L^2}
    \\&\,\,+\sum_{a\leq j+5N_0}C_N 2^{-k}2^{(2s-1)j}2^{-2sa}\sum_{m\leq k+3N_0}2^{(\gamma+\f32)m}\|\mF_a\mP_kg\|_{L^2}(\|\mF _j\mP_kh\|_{L^2}+2^{-jN}2^{-kN}\|h\|_{H^{-N}_{-N}})\\
    &\,\,\times(\|\mF _j\mP_kf\|_{L^2}+2^{-jN}2^{-kN}\|f\|_{H^{-N}_{-N}})+C_N\sum_{b>j+2N_0}2^{-bN}\|\mF_b\mP_{k}g\|_{L^2}\|\mF_b\mP_kh\|_{L^2}\|\mF_j\mP_kf\|_{L^2}\\
            &\,\,+C_N\sum_{b<j-2N_0}2^{-jN}\|\mF_j\mP_{k}g\|_{L^2}\|\mF_b\mP_kh\|_{L^2}\|\mF_j\mP_kf\|_{L^2}
    \end{aligned}
\end{equation*}
    \medskip
Similarly, from \eqref{Delta211} and \eqref{Delta311} and making the substitution $$
(g,h,f)\to(\mathcal{U}_{m-N_0}g,\tilde{\mathcal{P}}_m h,\F_j\tilde{\mathcal{P}}_m\mathcal{P}_k f), \quad (g,h,f)\to(\cP_lg,\cU_{m+N_0} h,\F_j\cU_{m+N_0}\mathcal{P}_k f)$$ respectively. Noting that for \(\Delta_{1,1}^1\), we always have \(m\ge 0\) and \(|m-k|\le 2N_0\), while for \(\Delta_{3,1}^1\) we have $|l-m|\leq N_0$ and  \(k\le m+2N_0\), thus we can obtain
    \begin{equation}\label{bound1ofD}
        \begin{aligned}
            \Delta_{1,1}^1\ls&\min\big\{\Delta_{1,1}^{1,(1)},\Delta_{1,1}^{1,(2)}\big\}
            =:~\mC_1(g,h,f),
        \end{aligned}
    \end{equation}
where
\begin{equation*}
    \begin{aligned}
    \Delta_{1,1}^{1,(1)}=&2^{(\gamma+s)k}2^{sj}\|\mS_j\mU_{k}g\|_{L^1}\|\mF_j\mP_kh\|_{L^2}\|\mF_j\mP_kf\|_{L^2}+C_N\sum_{b>j+2N_0}2^{-bN-kN}\|\mF_b\mU_{k}g\|_{L^1}\\
            &\times \|\mF_b\mP_kh\|_{L^2}\|\mF_j\mP_kf\|_{L^2}
            +C_N\sum_{b<j-2N_0}2^{-jN-kN}\|\mF_j\mU_{k}g\|_{L^1}\|\mF_b\mP_kh\|_{L^2}\|\mF_j\mP_kf\|_{L^2},\\
    \Delta_{1,1}^{1,(2)}=&\sum_{a\leq j+5N_0}2^{(2s-1)(j-a)}2^{(\gamma+\frac{3}{2})k}\|\mF_a\mU_kg\|_{L^2}\|\mF _j\mP_kh\|_{L^2}\|\mF _j\mP_kf\|_{L^2}
    \\&+\sum_{a\leq j+5N_0}C_N 2^{(2s-1)j}2^{-2sa}2^{(\gamma+\frac{1}{2})k}\|\mF_a\mU_kg\|_{L^2}(\|\mF _j\mP_kh\|_{L^2}+2^{-jN}2^{-kN}\|h\|_{H^{-N}_{-N}})\\
    &\times(\|\mF _j\mP_kf\|_{L^2}+2^{-jN}2^{-kN}\|f\|_{H^{-N}_{-N}})+C_N\sum_{b>j+2N_0}2^{-bN-kN}\|\mF_b\mU_{k}g\|_{L^2}\|\mF_b\mP_kh\|_{L^2}\|\mF_j\mP_kf\|_{L^2}\\
    &+C_N\sum_{b<j-2N_0}2^{-jN-kN}\|\mF_j\mU_{k}g\|_{L^2}\|\mF_b\mP_kh\|_{L^2}\|\mF_j\mP_kf\|_{L^2},
    \end{aligned}
\end{equation*}
and
    \begin{equation}\label{bound3ofD}
        \begin{aligned}
            \Delta_{3,1}^1\ls& \min\big\{\Delta_{3,1}^{1,(1)},\Delta_{3,1}^{1,(2)}\big\}+\sum_{k\leq 2N_0}2^{-k}2^{(2s-1)j}\|\mS_j\mU_{N_0}g\|_{L^2}\|\mF_j\mU_{N_0}h\|_{L^2}\|\mF_j\mP_kf\|_{L^2}\\
            &+C_N\sum_{k\leq 2N_0}\sum_{b>j+2N_0}2^{-k}(2^{-(\gamma+\f52)b}+2^{-\zeta b}b^\a 2^{(2s-1)j})\|\mF_b\mU_{N_0}g\|_{L^2}\\&\times(\|\mF_b\mU_{N_0}h\|_{L^2}+2^{-bN-kN}\|h\|_{H^{-N}_{-N}})
            (\|\mF_j\mP_kf\|_{L^2}+2^{-jN-kN}\|f\|_{H^{-N}_{-N}})\\&+C_N\sum_{k\leq 2N_0}\sum_{b<j-2N_0}2^{-k}(2^{-(\gamma+\f52)j}+2^{\f12 b}b^\a 2^{-(\gamma+3)j})\|\mF_j\mU_{N_0}g\|_{L^2}\\
            &\times(\|\mF_b\mU_{N_0}h\|_{L^2}+2^{-bN-kN}\|h\|_{H^{-N}_{-N}})(\|\mF_j\mP_kf\|_{L^2}+2^{-jN-kN}\|f\|_{H^{-N}_{-N}})=:\mC_3(g,h,f),
        \end{aligned}
    \end{equation}
where
\begin{equation*}
    \begin{aligned}
        \Delta_{3,1}^{1,(1)}=&\sum_{m\geq k-2N_0}2^{(\gamma+2s)m-sk}2^{sj}\|\mS_j\mP_{m}g\|_{L^1}\|\mF_j\mU_mh\|_{L^2}\|\mF_j\mP_kf\|_{L^2}+C_N\sum_{m\geq k-2N_0}\sum_{b>j+2N_0}2^{-bN}\\
            &\times 2^{-mN}\|\mF_b\mP_{m}g\|_{L^1}\|\mF_b\mU_mh\|_{L^2}\|\mF_j\mP_kf\|_{L^2}+C_N\sum_{m\geq k-2N_0}\sum_{b<j-2N_0}2^{-jN-mN}\|\mF_j\mP_{m}g\|_{L^1}\\
            &\|\mF_b\mU_mh\|_{L^2}\|\mF_j\mP_kf\|_{L^2},\\
    \end{aligned}
\end{equation*}
\begin{equation*}
    \begin{aligned}
        \Delta_{3,1}^{1,(2)}=&~\sum_{m\geq k-2N_0}\bigg(\sum_{a\leq j+5N_0}2^{(2s-\tau)(j-a)}2^{(\gamma+\frac{3}{2})m+\tau(m-k)}\|\mF_a\mP_mg\|_{L^2}\|\mF _j\mU_mh\|_{L^2}\|\mF _j\mP_kf\|_{L^2}
    \\&+\sum_{a\leq j+5N_0}C_N 2^{-k}2^{(2s-1)j}2^{-2sa}2^{(\gamma+\frac{3}{2})m}\|\mF_a\mP_mg\|_{L^2}(\|\mF _j\mU_mh\|_{L^2}+2^{-jN}2^{-kN}\|h\|_{H^{-N}_{-N}})\\&\times(\|\mF _j\mP_kf\|_{L^2}+2^{-jN}2^{-kN}\|f\|_{H^{-N}_{-N}})\bigg)+C_N\sum_{m\geq k-2N_0}\sum_{b>j+2N_0}2^{-bN} 2^{-mN}\|\mF_b\mP_{m}g\|_{L^2}\\&\times\|\mF_b\mU_mh\|_{L^2}\|\mF_j\mP_kf\|_{L^2}+C_N\sum_{m\geq k-2N_0}\sum_{b<j-2N_0}2^{-jN-mN}\|\mF_j\mP_{m}g\|_{L^2}\\
            &\times\|\mF_b\mU_mh\|_{L^2}\|\mF_j\mP_kf\|_{L^2},\quad\forall \tau\in[0,1].
    \end{aligned}
\end{equation*}
It ends the proof of this lemma.
\end{proof}

Next, we will handle the terms  $\Delta_{i,2}$ and $\Delta_{i,1}^2$ for $i=1,2,3$ and we have
\begin{lemma}\label{Delta2}
We have the following bound
\begin{equation*}
    \sum_{i=1}^3(\Delta_{i,2}+\Delta_{i,1}^2)\leq \sum_{i=1}^4\mC_i(g,h,f),
\end{equation*}
where $\mC_i,i=1,2,3$ are the same as in Lemma \ref{Delta1} and $\mC_4(g,h,f)$ is defined as 
\begin{equation}\label{mC4}
\begin{aligned}
     \mC_4(g,h,f):=&C_N2^{-Nk-Nj}\Big(2^{(\zeta-2s)M}\|\mP_k g\|_{H^{s-\zeta}}\|\mP_kh\|_{H^{s-\zeta}}\\
     &+C_M\|\mP_k g\|_{H^{-\zeta}}\|\mP_k h\|_{H^{-\zeta}}+\|g\|_{L^1}\|h\|_{H^{-N}}\Big)\|f\|_{H^{-N}}.
\end{aligned}
\end{equation}
\end{lemma}
\begin{proof}
We can estimate all these terms by the tools developed earlier, so we only outline the idea and omit the details.

{\bf\noindent $\bullet$ Estimates of $\Delta_{i,2},i=1,2,3$.} For example, recall from \eqref{00Delta1} that
\begin{equation*}
    \begin{aligned}
       \Delta_{1,2}=&\sum_{m\geq N_0-1,|m-k|\leq 2N_0}(\cP_kQ_m(\cU_{m-N_0}g,\tilde{\cP}_mh)-Q_m(\cU_{m-N_0}g,\cP_k\tilde{\cP}_mh),\big[\tilde{\cP}_m,\cF_j^2\big]\cP_kf)_v\\
       =&\sum_{a=-1}^\infty\sum_{m\geq N_0-1,|m-k|\leq 2N_0}(\cP_kQ_m(\cU_{m-N_0}g,\tilde{\cP}_mh)-Q_m(\cU_{m-N_0}g,\cP_k\tilde{\cP}_mh),\F_a\big[\tilde{\cP}_m,\cF_j^2\big]\cP_kf)_v.
    \end{aligned}
\end{equation*}
If $|a-j|\leq N_0$, we can obtain the same estimate by exactly the same method used for \(\Delta_{1,1}\). If $a<j-N_0$ or $a>j+N_0$, from Lemma \ref{pseudo operator} \eqref{7.5}, we know that
\begin{equation*}
  \big[\tilde{\cP}_m,\cF_j^2\big]=2^{-j}2^{-m}\mF_j\mP_m+C_N2^{-Nj-Nm}r_N(v,D_v),
\end{equation*}
with $\<v\>^{N}r_N(v,\xi)\in S^{-N}_{1,0}$, which implies that $\F_a[\tP_m,\F_j^2]=C_N2^{-Nj-Nm}\F_ar_N(v,D_v)$ since $\F_a\mF_j=0$, then we need to handle with
\begin{equation*}
\begin{aligned}
    \big(\sum_{a<j-N_0}+\sum_{a>j+N_0}\big)&\sum_{m\geq N_0-1,|m-k|\leq 2N_0}\big[\big(Q_m(\cU_{m-N_0}g,\tilde{\cP}_mh),C_N2^{-Nj-Nm}\cP_k\F_ar_N(v,D_v)\cP_kf\big)_v\\
    &-\big(Q_m(\cU_{m-N_0}g,\cP_k\tilde{\cP}_mh),C_N2^{-Nj-Nm}\F_ar_N(v,D_v)\cP_kf\big)_v\big].
\end{aligned}
\end{equation*}
For each term, we can apply the decomposition as in \eqref{eq m1234} and use Lemma \ref{M14}, Lemma \ref{M23} (note that $m\geq N_0-1$ and $|m-k|\leq 2N_0$) to obtain the bound
\begin{equation*}
\begin{aligned}
C_N2^{-kN-jN}\|g\|_{L^1}\|h\|_{H^{-N}_{-N}}\|f\|_{H^{-N}_{-N}},
\end{aligned}
\end{equation*}
where we use the fact $\|\F_a r_N(v,D_v)\cP_kf\|_{L^2}\ls C_N \|f\|_{H^{-N}_{-N}}$ for any $N\in\N$.

Similarly, for $\Delta_{2,2}$ and $\Delta_{3,2}$, we only need to deal with the terms 
\begin{equation*}
\begin{aligned}
     &\big(\sum_{a<j-N_0}+\sum_{a>j+N_0}\big)\bigg(\sum_{l\geq m+N_0,|l-k|\leq 2N_0}(\cP_kQ_{m}(\cP_lg,\tilde{\cP}_lh)-Q_{m}(\cP_lg,\cP_k\tilde{\cP}_lh),\F_a\big[\tilde{\cP}_l,\cF^2_j\big]\cP_kf)_v\\&+\sum_{|l-m|\leq N_0,k\leq m+2N_0}(\cP_kQ_{m}(\cP_lg,\cU_{m+N_0}h)-Q_{m}(\cP_lg,\cP_k\cU_{m+N_0}h),\F_a\big[\cU_{m+N_0},\cF_j^2\big]\cP_kf)_v\bigg).
\end{aligned}
\end{equation*}
We can respectively obtain the following bounds:
\begin{equation*}
    \begin{aligned}
    &C_N2^{-Nk-Nj}\Big(\sum_{p=-1}^\infty 2^{-\zeta p}\|\mF_p\mP_kg\|_{L^2}\|\mF_p\mP_k h\|_{L^2}\|f\|_{H^{-N}_{-N}}+\|g\|_{L^1}\|h\|_{H^{-N}_{-N}}\|f\|_{H^{-N}_{-N}}\Big)~~\mbox{and}\\
    &C_N2^{-Nk-Nj}\Big(\sum_{m\geq0}2^{(\gamma+2s)m}\|\mP_mg\|_{L^1}\|\mU_{m}h\|_{H^{-N}}\|\mU_{m}f\|_{H^{-N}}+\sum_{p=-1}^\infty 2^{-\zeta p}\|\mF_p\mU_{N_0}g\|_{L^2}\|\mF_p\mU_{N_0}h\|_{L^2}\\
    &\times\|f\|_{H^{-N}_{-N}}
+\|g\|_{L^1}\|h\|_{H^{-N}_{-N}}\|f\|_{H^{-N}_{-N}}\Big),~~\forall N\in\N.
    \end{aligned}
\end{equation*}
Observing that the sum over \(p\) can be bounded by
\begin{equation*}
    \begin{aligned}
 &2^{(\zeta-2s)M}\sum_{p>M}2^{(s-\zeta)p}\|\mF_p\mP_k g\|_{L^2}2^{(s-\zeta)p}\|\mF_p\mP_k h\|_{L^2}+C_M\sum_{p\leq M}2^{-\zeta p}\|\mF_p\mP_k g\|_{L^2}2^{-\zeta p}\|\mF_p\mP_k h\|_{L^2}\\
 &\ls ~2^{(\zeta-2s)M}\|\mP_k g\|_{H^{s-\zeta}}\|\mP_kh\|_{H^{s-\zeta}}+C_M\|\mP_k g\|_{H^{-\zeta}}\|\mP_k h\|_{H^{-\zeta}}.
    \end{aligned}
\end{equation*}
Therefore, if we denote $\mC_4(g,h,f)$ by 
\begin{equation*}
\begin{aligned}
    C_N2^{-Nk-Nj}\Big(2^{(\zeta-2s)M}\|\mP_k g\|_{H^{s-\zeta}}\|\mP_kh\|_{H^{s-\zeta}}+C_M\|\mP_k g\|_{H^{-\zeta}}\|\mP_k h\|_{H^{-\zeta}}+\|g\|_{L^1}\|h\|_{H^{-N}}\Big)\|f\|_{H^{-N}},
\end{aligned}
\end{equation*}
we then obtain that
\begin{equation*}
 \sum_{i=1}^3\Delta_{i,2}\leq \sum_{i=1}^4\mC_i(g,h,f).
\end{equation*}

{\bf \noindent $\bullet$ Estimates of $\Delta_{i,1}^2,i=1,2,3$.} For example, for $i=1$, from  \eqref{Deltacomm}, the first term of \eqref{Delta1,12} is 
\begin{equation*}
    \sum_{m\geq N_0-1,|m-k|\leq 2N_0}(Q_m(\cU_{m-N_0}g,\tilde{\cP}_mh),[\cP_k,\cF_j^2]\tilde{\cP}_m\cP_kf)_v.
\end{equation*}
We can use frequency decomposition \eqref{eq m1234} and apply Lemma \ref{M14} and Lemma \ref{M23}. Noticing that there is a commutator $[\cP_k,\F_j^2]$ and $m\geq 0, |m-k|\leq 2N_0$, so we can deduce that the upper bound loses one order in weight and one order in derivative. For example, for the localized term corresponding to $\M^2$,
\begin{equation*}
    \sum_{a\leq b+3N_0,|b-j|\leq N_0}\sum_{m\geq N_0-1,|m-k|\leq 2N_0}(Q_m(\F_a\cU_{m-N_0}g,\tF_b\tilde{\cP}_mh),\tF_b[\cP_k,\cF_j^2]\tilde{\cP}_m\cP_kf)_v,
\end{equation*}
the coefficient contains a factor of \(2^{(\gamma+2s-1)k+(2s-1)j}\) for $L^1$ estimate which is smaller than $2^{(\gamma+s)k+sj}$ since $s<1$ and $2^{-2s a+(2s-1)j}2^{(\gamma+\frac{3}{2}-1)k}$ for $L^2$ estimate which is smaller than $2^{(2s-1)(j-a)}2^{(\gamma+\f32)k}$. Thus it can be bounded by $\mC_1(g,h,f)$. Other terms of $\Delta_{1,1}^2$ can be handled even easier since both phase space and frequency space have been localized.

The similar argument also applies to the terms $\Delta_{2,1}^2$ and $\Delta_{3,1}^2$, they can be bounded by $\mC_2(g,h,f)$ and $\mC_3(g,h,f)$ respectively.

It ends the proof of this lemma.
\end{proof}

Finally, combining Lemma \ref{Delta1} and Lemma \ref{Delta2}, we conclude that
\begin{proposition}[Commutator of $Q$ and $\cP_k$]\label{QPk}
Under Assumption \ref{assumption}, for any smooth functions $g,h$ and $f$, it holds that
\begin{equation}\label{PkQghf}
    \begin{aligned}
        |(\cP_kQ(g,h)-Q(g,\cP_kh),\F^2_j\cP_kf)_v|\leq \sum_{i=1}^4\mC_i(g,h,f),
    \end{aligned}
\end{equation}
where $\mC_i(g,h,f),i=1,2,3$ are given in \eqref{bound1ofD}, \eqref{bound2ofD}, \eqref{bound3ofD} and \eqref{mC4}. In particular, if we denote $\mC_i(g,f,f)$ by $\mC_i(g,f),i=1,2,3$, then
\begin{equation}\label{PkQf}
    \begin{aligned}
        |(\cP_kQ(g,f)-Q(g,\cP_kf),\F^2_j\cP_kf)_v|\leq \sum_{i=1}^4\mC_i(g,f).
    \end{aligned}
\end{equation}
\end{proposition}

\section{Proof of local existence, regularization and uniqueness}\label{proofofmaintheorem}
In this section, we begin to prove our main results, Theorems \ref{maintheorem1} and \ref{maintheorem2}.

\subsection{Proof of the local existence}
We first prove existence. Since local existence can be obtained by the standard iteration scheme, we only present the a priori energy estimates.
\begin{proof}[Proof of the existence part of Theorems \ref{maintheorem1} and \ref{maintheorem2}]
Suppose that the initial data  $f_0\in L^1_{2}(\R^3)\ccap  L^1_{l+\gamma+s}(\R^3)\ccap  L\log L\ccap  H^{-\zeta,\a}_l$ with $\a>0$ and $l\geq-\gamma$, then by Proposition \ref{propagationofL1moment}, we have that
\begin{equation}\label{L1l}
    \|f(t)\|_{L^1_2\ccap  L^1_{l+\gamma+s}}\leq C_l,\quad \forall t\leq 1.
\end{equation}

Multiplying both sides of equation \eqref{Boltzmann} by \(\cP_k\F_j^2\cP_k f\) and integrating with respect to \(v\in\mathbb{R}^3\) gives
\begin{equation}\label{energy1}
\begin{aligned}
&\frac{1}{2}\frac{\d}{\d t}\|\cF_j\cP_k f\|_{L^2}^2=(\p_t\cF_j\cP_k f,\cF_j\cP_k f)_v=(Q(f,f),\cP_k\cF_j^2\cP_kf)_v\\
 =\&(Q(f,\cF_j\cP_kf),\cF_j\cP_kf)_v+\big((Q(f,f),\cP_k\cF_j^2\cP_kf)_v-(Q(f,\cF_j\cP_kf),\cF_j\cP_kf)_v\big)\\
 =\&(Q(f,\cF_j\cP_kf),\cF_j\cP_kf)_v+\big((\cF_jQ(f,\cP_kf)-Q(f,\cF_j\cP_kf),\cF_j\cP_kf\big)_v\\&+\big((\cP_kQ(f,f)-Q(f,\cP_k f),\cF_j^2\cP_kf\big)_v
 =:L_1(j,k)+L_2(j,k)+L_3(j,k).
\end{aligned}    
\end{equation}
Furthermore, for the two cases \(\gamma+2s\in(-1,-1/2)\) and \(\gamma+2s\in(-1/2,0)\), we multiply both sides by \(2^{2lk}2^{2n j}j^{2\alpha}\) with \(l\ge -\gamma\), \(\alpha>0\) ,\(n\geq -\zeta\), where $\zeta$ is defined in \eqref{zeta1}; when \(\gamma+2s=-1/2\), we multiply by \(2^{2lk}2^{2nj}j^{2(1+\alpha)}\). Below we only give the computation for the first case. Summing over \(j,k\ge -1\), we obtain
\begin{equation}\label{energyL123}
    \f12\f{\d}{\d t }\sum_{j,k=-1}^\infty 2^{2l k}2^{2n j}j^{2\a}\|\F_j\cP_k f\|^2_{L^2}=\sum_{i=1}^3\sum_{j,k=-1}^\infty 2^{2lk}2^{2n j}j^{2\a}L_i(j,k).
\end{equation}
In what follows, we denote the energy and dissipation norms by
\begin{equation}\label{norms}
 \begin{aligned}
    \mX_{n,\a,l}(f)=\sum_{j,k=-1}^{\infty}2^{2n j}j^{2\a}2^{2lk}\|\F_j\cP_kf\|_{L^2}^2,\quad\mbox{and}\quad
    \mY_{n,\a,l}(f)=\sum_{j,k=-1}^{\infty}2^{2(s+n)j}j^{2\a}2^{2(l+\f\gamma 2)k}\|\F_j\cP_kf\|_{L^2}^2.
\end{aligned}   
\end{equation}
Thanks to Lemma \ref{le1.4}, we know that $\mX_{n,\a,l}(f)\sim \|f\|^2_{H^{n,\a}_{l}}$ and $\mY_{n,\a,l}(f)\sim \|f\|^2_{H^{s+n,\a}_{l+\gamma/2}}$. 

We will proceed in several steps, using the conclusions from the preceding sections to estimate each term separately.

\textbf{Step 1: Estimates of $L_1(j,k)$.} Noticing that
\begin{equation*}
    \begin{aligned}
     (-Q(g,f),\,f)_{v}=&-\f12\int_{\R^6\times\SS^2}B(|v-v_*|,\sigma)g_*(f'^2-f^2)d\sigma \d v _*\d v \\
     &+\f12\int_{\R^6\times\SS^2}B(|v-v_*|,\sigma)g_*(f'-f)^2d\sigma \d v _*\d v =:\mathcal{L}+\mathcal{E}^\gamma_g(f).
    \end{aligned}
\end{equation*}
By cancellation lemma, 
\begin{equation*}
    \begin{aligned}
        |\mathcal{L}|&=|\SS^1|\Big|\int_{\R^6}\int_0^{\f\pi 2}\sin\theta\Big(\f1{\cos^2\f\theta 2}B(\f{|v-v_*|}{\cos\f\theta 2},\cos\theta)-B(|v-v_*|,\cos\theta)\Big)g_*f^2 d\theta \d v _*\d v \Big|\\
        &\ls \int_{\R^6}|v-v_*|^\gamma g_*f^2\d v _*\d v =\int_{\R^6}\Big((1-\psi)|\cdot|^\gamma\Big)(v-v_*)g_*f^2\d v _*\d v  \\
        &+\int_{\R^6}\Phi^{\gamma}_{-1}(|v-v_*|)g_*f^2\d v _*\d v 
        \ls \|g\|_{L^1}\|f\|^2_{L^2}+\mathcal{R}(g,f,f),
    \end{aligned}
\end{equation*}
where $\Phi^{\gamma}_{-1}$ is defined in \eqref{DefPhi} and we denote the last term by $\mathcal{R}(g,f,f)$.

Thus we have 
\begin{equation*}
    \begin{aligned}
        L_1(j,k)\leq -\mathcal{E}^\gamma_f(\F_j\cP_kf)+\mathcal{R}(f,\F_j\cP_kf,\F_j\cP_k f)+\|f\|_{L^1}\|\F_j\cP_kf\|^2_{L^2}.
    \end{aligned}
\end{equation*}
On one hand, from Proposition \ref{prop:coercivity} (choose $\mathbf{A}=0$), there exists constants $\kappa, C>0$ such that
\begin{equation}
\begin{aligned}
     -\mathcal{E}^\gamma_f(\F_j\cP_kf)&\ls-\kappa\|\cF_j\cP_kf\|_{H^s_{\gamma/2}}^2+C\|\cF_j\cP_kf\|_{L^2_{\gamma/2}}^2.
\end{aligned}    
\end{equation}
On the other hand, by Lemma \ref{mathcalR}, the fact that $\zeta\leq \gamma+2s+\f32$ and $l\geq-\gamma$, we have 
\begin{equation*}
    \mathcal{R}(f,\F_j\cP_kf,\F_j\cP_k f)\ls M^{-\a}\|f\|_{H^{-\zeta,\a}_{l}}\|\F_j\cP_k f\|_{H^s_{\gamma/2}}^2+C_{M,\a}\|f\|_{H^{-\zeta,\a}_{l}}\|\F_j\cP_k f\|^2_{L^2_{\gamma/2}},\quad\forall M,\a>0.
\end{equation*}
It leads to
\begin{equation*}
    \begin{aligned}
        L_1(j,k)&\leq -(\kappa-M^{-\a}\|f\|_{H^{-\zeta,\a}_{l}})\|\cF_j\cP_kf\|_{H^s_{\gamma/2}}^2
+C_{M,\a}(\|f\|_{H^{-\zeta,\a}_{l}}+1)\|\F_j \cP_kf\|^2_{L^2_{\gamma/2}},\quad\forall M,\a>0.
    \end{aligned}
\end{equation*}
On one hand, in view of Lemma \ref{le1.4}, it holds that
\begin{equation*}
    \sum_{j,k=-1}^\infty 2^{2l k}2^{2n j}j^{2\a}\|\F_j\cP_k f\|^2_{H^s_{\gamma/2}}\sim \|f\|^2_{H^{s+n,\a}_{l+\gamma/2}}\sim  \mY_{n,\a,l}(f).
\end{equation*}
On the other hand, using interpolation and Young inequalities, we have
\begin{equation*}
    \begin{aligned}
    2^{2nj}\|\F_j \cP_kf\|^2_{L^2_{\gamma/2}}\leq&2^{-2\theta_1\zeta j}2^{2(1-\theta_1)(n+s)j}\|\F_j \cP_kf\|^2_{L^2_{\gamma/2}}
    \leq \varepsilon2^{2(n+s)j}\|\F_j \cP_kf\|^2_{L^2_{\gamma/2}}+C_{\varepsilon}2^{-2\zeta j}\|\F_j \cP_kf\|^2_{L^2_{\gamma/2}},
\end{aligned}
\end{equation*}
where $\theta_1=\frac{s}{n+s+\zeta}$.
Therefore, we can deduce that 
\begin{equation}\label{L1jk}
    \begin{aligned}
        \sum_{j,k=-1}^\infty 2^{2l k}2^{2n j}j^{2\a}L_1(j,k)\leq -\big(\kappa-(M^{-\a}+\varepsilon)\mX_{-\zeta,\a,l}(f)^{\f12}-\varepsilon\big)\mY_{n,\a,l}(f)\\+C_{M,\a,n,\varepsilon}(\mX_{-\zeta,\a,l}(f)+\mX_{-\zeta,\a,l}(f)^{\f32}),\quad\forall M,\a,\varepsilon>0,n\geq-\zeta.
    \end{aligned}
\end{equation}

\textbf{Step 2: Estimates of $L_2(j,k)$.} Recall that
\begin{equation*}
    L_2(j,k)=(\cF_jQ(f,\cP_kf)-Q(f,\cF_j\cP_kf),\cF_j\cP_kf)_v.
\end{equation*}
Thanks to Corollary \ref{cor1}, we have
\begin{equation*}
    \begin{aligned}
        &\sum_{j,k=-1}^\infty 2^{2l k}2^{2n j}j^{2\a}L_2(j,k)
    \ls\sum_{j,k=-1}^\infty2^{(2l+\gamma+2s-1)k}2^{2(s+n-\frac{1}{2})j}j^{2\a}\|f\|_{L_2^1}\|\mF_j\mP_kf\|_{L^2}^2\\
    &+\sum_{j,k=-1}^\infty2^{(2l+\gamma+2s-2N)k}2^{(2s+2n-2N)j}j^{2\a}\|f\|_{L^1_2}\|f\|_{H^{-N}_{-N}}^2
    +\sum_{j,k=-1}^\infty\sum_{i=1}^32^{2l k}2^{2n j}j^{2\a}\mB_i(f),
\end{aligned} 
\end{equation*}
where we denote $\mB_i(f,f)$ by $\mB_i(f)$ for $i=1,2,3$.

For the first term on the right-hand side, by \eqref{L1l}, Lemma \ref{le1.4},  interpolation and Young inequalities, we have
\begin{equation}\label{third}
    \begin{aligned}
         &\sum_{j,k=-1}^{\infty}2^{(2l+\gamma+2s-1)k}2^{2(s+n-\frac{1}{2})j}j^{2\a}\|f\|_{L_2^1}\|\mF_j\mP_kf\|_{L^2}^2
    \\\leq&\sum_{j,k=-1}^{\infty}2^{2\theta_2(l+\frac{\gamma}{2}+(2s-1)(n+s+\zeta)) k}2^{2(1-\theta_2)(l+\f\gamma 2) k}2^{-2\theta_2\zeta j}2^{2(1-\theta_2)(s+n)j}j^{2\a}\|\mF_j\mP_kf\|_{L^2}^2\\
    \leq &\sum_{j,k=-1}^\infty\big( \varepsilon 2^{(2l+\gamma) k}2^{2(s+n)j}j^{2\a}+C_{\varepsilon}2^{2(l+\f\gamma 2+(s-\f12)/\theta_2) k}2^{-2\zeta j}j^{2\a}\big)\|\mF_j\mP_kf\|_{L^2}^2
    \\\leq& \varepsilon\mY_{n,\a,l}(f)+C_{\varepsilon,n}\mX_{-\zeta,\a,l+\f\gamma 2+(2s-1)(n+s+\zeta)}(f),\quad\forall \varepsilon>0,
    \end{aligned}
\end{equation}
where {$\theta_2=\frac{1}{2(n+s+\zeta)}$ and $(s-\f12)/\theta_2=(2s-1)(n+s+\zeta)$.}

For the second term, it is also easy to get that
\begin{equation}\label{forth}
 \begin{aligned}
    \sum_{j,k=-1}^{\infty}2^{(2l+\gamma+2s-2N)k}2^{(2s+2n-2N)j}j^{2\a}\|f\|_{L^1_2}\|f\|_{H^{-N}_{-N}}^2\ls C_N\|f\|_{H^{-N}_{-N}}^2\ls C_N\mX_{-\zeta,\a,l}(f),
\end{aligned}   
\end{equation}
if we choose suitably large $N\in\N$.

It remains to handle the terms containing $\mB_i(f),i=1,2,3$. The mechanism for all of these sums is the same. After redistributing the exponential weights, each summand is written as a product of energy factors ($\mX$-type) and dissipation factors ($\mY$-type) times a summable tail. Splitting the frequency summation at a large threshold $M$, the high-frequency region carries the polynomial smallness $M^{-\a}$ (coming from the weight $j^{\a}$, or exponential smallness from a strict exponential gap) and is absorbed by the dissipation, while the finitely many low frequencies produce the harmless cubic term $C_{M,\a}\mX^{\f32}$. This is precisely the point where the logarithmic weight $j^{\a}$ acts as a substitute for the missing $\varepsilon$ of regularity in the critical space. Specifically,

{\bf\noindent $\bullet$ Estimate of the term containing $\mB_1(f)$.} Recall from Corollary \ref{cor1} and \eqref{bound1} that
\begin{equation*}
\begin{aligned}
& \mB_1(f)= \sum_{p>j+3N_0}\Big( 2^{-(\gamma+4)p}2^{\frac{5}{2}j}\|\mP_k\mF_pf\|_{L^2}\|\mF_p
        \mP_kf\|_{L^2}\|\mF_j\mP_kf\|_{L^2}
        +\sum_{k\leq 2N_0}2^{-(\gamma+4)p}2^{\frac{5}{2}j}\|\mU_{N_0}\mF_pf\|_{L^2}\\
        &\times\|\mF_p\mU_{N_0}f\|_{L^2}\|\mF_j\mU_{N_0}f\|_{L^2}
        +C_N2^{-\zeta p}p^\a 2^{(2s-1)j}2^{-k} (\|\mF_p\mP_k f\|_{L^2}+2^{-pN-kN}\|f\|_{H^{-N}_{-N}})
      (\|\mF_p\mP_kf\|_{L^2}\\
      &+2^{-pN-kN}\|f\|_{H^{-N}_{-N}})(\|\mF_j\mP_k f\|_{L^2}+2^{-jN-kN}\|f\|_{H^{-N}_{-N}})
        + \sum_{k\leq 2N_0}C_N 2^{-\zeta p}p^\a 2^{(2s-1)j}(\|\mF_p \mU_{N_0}f\|_{L^2}\\
        &+2^{-pN}\|f\|_{H^{-N}_{-N}})
         (\|\mF_p\mU_{N_0}f\|_{L^2}+2^{-pN}\|f\|_{H^{-N}_{-N}})(\|\mF_j\mU_{N_0}f\|_{L^2}+2^{-jN}\|f\|_{H^{-N}_{-N}})\Big).
\end{aligned}
\end{equation*}
Here we only estimate the first and third term, the rest terms can be handled similarly. For the first term, since $l\geq-\gamma$ which implies $2l\leq 3l+\gamma$, then 
\begin{equation*}
  \begin{aligned}
    &\sum_{j,k=-1}^{\infty}\sum_{p>j+3N_0}2^{-(\gamma+4)p}2^{2l k}2^{2n j}j^{2\a}2^{\frac{5}{2}j}\|\mP_k\mF_pf\|_{L^2}\|\mF_p
    \mP_kf\|_{L^2}\|\mF_j\mP_kf\|_{L^2}\\
    \leq&\sum_{j,k=-1}^{\infty}\sum_{p>j+3N_0}(2^{(l+\f\gamma 2 )k}2^{(s+n)p}p^{\a}\|\mP_k\mF_pf\|_{L^2})(2^{(l+\f\gamma 2) k}2^{(s+n)p}p^{\a}\|\mF_p
    \mP_kf\|_{L^2})\\
    &\times(2^{l k}2^{-\zeta j}j^{\a}\|\mF_j\mP_kf\|_{L^2})2^{(2n+\frac{5}{2}+\zeta)(j-p)}p^{-2\a}j^{\a}.
\end{aligned}  
\end{equation*}
Noticing that $(-\gamma-4)p+(\frac{5}{2}+2n)j=2(n+s)p-\zeta j+[-(\gamma+2s+\frac{3}{2})p+\zeta j]+(2n+\frac{5}{2})(j-p)\leq2(n+s)p-\zeta j+(2n+\frac{5}{2}+\zeta)(j-p)$ and $p^{-2\a}j^\a\leq j^{-\a}$ since $\zeta\leq \gamma+2s+\f32$ and $p>j$, we split the above sum into three parts: \(j\ge M\); \(j<M\) and \(p\ge 2M\); and \(p<2M\). For the first and second case, we have $j^{-\a}\leq M^{-a}$ and $2^{(\f52+2n+\zeta)(j-p)}\leq 2^{-(\f52-\zeta)M}$ respectively since $n\geq -\zeta$. Then, by Cauchy inequality and Lemma \ref{le1.4}, it has the bound
\begin{equation*}
\begin{aligned}
      &(M^{-\a}+2^{-(\f52-\zeta)M})\big(\sum_{p,k=-1}^\infty2^{2(l+\f\gamma 2 )k}2^{2(s+n)p}p^{2\a}\|\mP_k\mF_pf\|^2_{L^2}\big)\\
      &\times\big(\sum_{j,k=-1}^\infty2^{2lk}2^{-2\zeta j}j^{2\a}\|\mP_k\mF_jf\|^2_{L^2}\big)^{\f12}
    \ls (M^{-\a}+2^{-(\f52-\zeta)M})\mX^{\f12}_{-\zeta,\a,l}(f)\mY_{n,\a,l}(f).
\end{aligned}
\end{equation*}
While for the third case, we have $j<p<2M$ and it has the bound $ C_{M,\a}\mX^{\f32}_{-\zeta,l}(f)$. Therefore, we obtain the bound
\begin{equation*}
    \begin{aligned}
         (M^{-\a}+2^{-(\f52-\zeta)M})\mX^{\f12}_{-\zeta,\a,l}(f)\mY_{n,\a,l}(f)+C_{M,\a}\mX^{\f32}_{-\zeta,\a,l}(f).
    \end{aligned}
\end{equation*}

For the third term, we have 
\begin{equation*}
   \begin{aligned}
    &\sum_{j,k=-1}^{\infty}\sum_{p>j+3N_0}2^{2lk}2^{2n j}j^{2\a}2^{-\zeta p}p^\a 2^{(2s-1)j}2^{-k}\|\mF_p
    \mP_kf\|_{L^2}\|\mF_p
    \mP_kf\|_{L^2}\|\mF_j\mP_kf\|_{L^2}\\
    \leq&\sum_{j,k=-1}^{\infty}\sum_{p>j+3N_0}(2^{(l+\f\gamma2) k}2^{(s+n)p}p^{\a}\|\mF_p\mP_kf\|_{L^2})(2^{(l+\f\gamma2) k}2^{(s+n)p}p^{\a}\|\mF_p
    \mP_kf\|_{L^2})\\
    &\times(2^{l k}2^{-\zeta j}j^{\a}\|\mF_j\mP_kf\|_{L^2})2^{(2n+\zeta+2s)(j-p)}2^{-j}p^{-\a}j^{\a}2^{-k}.
\end{aligned} 
\end{equation*}
Noticing that $\zeta\leq1<2s,2n+\zeta+2s\geq2s-\zeta>0$, then we can use the same method as before,  consider three cases: \(j\ge M\); \(j<M\) and \(p\ge 2M\); and \(p<2M\) to get the bound
\begin{equation*}
    (2^{-M}M^\a+2^{(\zeta-2s)M})\mX_{-\zeta,\a,l}(f)^{\f12}\mY_{n,\a,l}(f)+C_{M,\a}\mX_{-\zeta,\a,l}(f)^{\f32}.
\end{equation*}

The rest terms can be handled by the same manner. Since $2^{-(\f52-\zeta)M}+2^{-M}M^\a+2^{(\zeta-2s)M}\ls M^{-\a}$, so we conclude that
\begin{equation}\label{sumB1}
    \sum_{j,k=-1}^\infty2^{2l k}2^{2n j}j^{2\a}\mB_1(f)\ls M^{-\a}\mX^{\f12}_{-\zeta,\a,l}(f)\mY_{n,\a,l}(f)+C_{M,\a}\mX^{\f32}_{-\zeta,\a,l}(f).
\end{equation}

{\bf\noindent $\bullet$ Estimate of the term containing $\mB_2(f)$.} We can obtain \(\mB_2(f)\) from Corollary \ref{cor1} and \eqref{bound2}, and we only treat the first term since the others are similar. 
We have
\begin{equation*}
   \begin{aligned}
    &\sum_{j,k=-1}^{\infty}\sum_{p<j-3N_0}2^{2lk}2^{2nj}j^{2\a}2^{-(\gamma+\f32)j}\|\mP_k \mF_jf\|_{L^2}\|\mF_p\mP_k f\|_{L^2}\|\mF_j \mP_k f\|_{L^2}\\
    =&\sum_{j,k=-1}^{\infty}\sum_{p<j-3N_0}(2^{(l+\f\gamma2) k}2^{(s+n)j}j^{\a}\|\mP_k\mF_jf\|_{L^2})(2^{l k}2^{-\zeta p}p^{\a}\|\mF_p
    \mP_kf\|_{L^2})\\
    &\times(2^{(l+\f\gamma 2)k}2^{(s+n) j}j^{\a}\|\mF_j\mP_kf\|_{L^2})2^{(\zeta-\gamma-2s-\frac{3}{2})j}2^{\zeta(p-j)}p^{-\a}.
\end{aligned} 
\end{equation*}
Noticing that $0<\zeta\leq \gamma+2s+\f32$, we split the above sum into three parts: \(p\ge M\); \(p<M\) and \(j\ge 2M\); and \(j<2M\). Then we can get the bound
\begin{equation*}
    \begin{aligned}
         (M^{-\a}+2^{-\zeta M})\mX^{\f12}_{-\zeta,\a,l}(f)\mY_{n,\a,l}(f)+C_{M,\a}\mX^{\f32}_{-\zeta,\a,l}(f).
    \end{aligned}
\end{equation*}
 After estimating term by term, we summarize to obtain
\begin{equation}\label{sumB2}
    \sum_{j,k=-1}^\infty2^{2l k}2^{2n j}j^{2\a}\mB_2(f)\ls M^{-\a}\mX^{\f12}_{-\zeta,\a,l}(f)\mY_{n,\a,l}(f)+C_{M,\a}\mX^{\f32}_{-\zeta,\a,l}(f).
\end{equation}

{\bf\noindent $\bullet$ Estimate of the term containing $\mB_3(f)$.} We can obtain \(\mB_3(f)\) from Corollary \ref{cor1} and \eqref{bound3}, and we also treat the first term. We have
\begin{equation*}
    \begin{aligned}
        &\sum_{j,k=-1}^\infty\sum_{q<j+7N_0}2^{2l k}2^{2n j}j^{2\a}2^{-(\zeta-1)q}2^{(2s-1)j}\|\mP_k\mF_q f\|_{L^2}\|\mF_j\mP_kf\|^2_{L^2}\\
        =&\sum_{j,k=-1}^\infty(2^{(2l+\gamma) k}2^{2(s+n)j}j^{2\a}\|\mF_j\mP_kf\|^2_{L^2})\sum_{q<j+7N_0}(2^{l k}2^{-\zeta q}q^\a\|\mP_k\mF_q f\|_{L^2})q^{-\a}2^{q-j}.
    \end{aligned}
\end{equation*}
We can consider three cases as before, that is \(q\ge M\); \(q<M\) and \(j\ge 2M\); and \(j<2M\) to get the following bound
\begin{equation*}
    (M^{-\a}+2^{- M})\mX_{-\zeta,\a,l}(f)^{\f12}\mY_{n,\a,l}(f)+C_{M,\a}\mX_{-\zeta,\a,l}(f)^{\f32}.
\end{equation*}
And we can conclude that
\begin{equation}\label{sumB3}
    \sum_{j,k=-1}^\infty2^{2l k}2^{2n j}j^{2\a}\mB_3(f)\ls M^{-\a}\mX_{-\zeta,\a,l}(f)^{\f12}\mY_{n,\a,l}(f)+C_{M,\a}\mX_{-\zeta,\a,l}(f)^{\f32}.
\end{equation}

Combining \eqref{sumB1}, \eqref{sumB2} and \eqref{sumB3}, we obtain
\begin{equation}\label{mBif}
    \sum_{j,k=-1}^\infty \sum_{i=1}^32^{2l k}2^{2n j}j^{2\a}\mB_i(f)\ls M^{-\a}\mX_{-\zeta,\a,l}(f)^{\f12}\mY_{n,\a,l}(f)+C_{M,\a}\mX_{-\zeta,\a,l}(f)^{\f32}.
\end{equation}

Summarizing \eqref{third}, \eqref{forth} and \eqref{mBif}, we have 
\begin{equation}\label{L2jk}
    \begin{aligned}
        \sum_{j,k=-1}^\infty 2^{2l k}2^{2n j}j^{2\a}L_2(j,k)\leq \big(M^{-\a}\mX_{-\zeta,\a,l}(f)^{\f12}+\varepsilon\big)\mY_{n,\a,l}(f)+C_{M,\a,\varepsilon,n}(\mX_{-\zeta,\a,l}(f)+\mX_{-\zeta,\a,l}(f)^{\f32})\\+C_{\varepsilon,n}\mX_{-\zeta,\a,l+\frac{\gamma}{2}+(2s-1)(n+s+\zeta)}(f),\quad \forall \varepsilon,M,\a>0,n\geq-\zeta.
    \end{aligned}
\end{equation}

\textbf{Step 3: Estimates of $L_3(j,k)$.} Recall that
\begin{equation*}
    L_3(j,k)=(\cP_kQ(f,f)-Q(f,\cP_k f),\cF_j^2\cP_kf)_v.
\end{equation*}
Thanks to \eqref{PkQf}, we have
\begin{equation*}
    \begin{aligned}
        \sum_{j,k=-1}^\infty 2^{2l k}2^{2n j}j^{2\a}L_3(j,k)
    \leq\sum_{j,k=-1}^\infty\sum_{i=1}^42^{2l k}2^{2n j}j^{2\a}\mC_i(f),
\end{aligned} 
\end{equation*}
where we denote $\mC_i(f,f)$ by $\mC_i(f), i=1,2,3,4.$ The estimation method is essentially the same as $\mB_i$.

\medskip

{\bf \noindent $\bullet$ Estimate of the term containing $\mC_1(f)$.} Recall from \eqref{bound1ofD}(choose $\Delta_{1,1}^{1,(1)}$) that
 \begin{equation*}
        \begin{aligned}
&\mC_1(f)\leq2^{(\gamma+s)k}2^{sj}\|\mS_j\mU_{k}f\|_{L^1}\|\mF_j\mP_kf\|_{L^2}\|\mF_j\mP_kf\|_{L^2}+C_N\sum_{b>j+2N_0}2^{-bN-kN}\|\mF_b\mU_{k}f\|_{L^1}\\
&\times\|\mF_b\mP_kf\|_{L^2}\|\mF_j\mP_kf\|_{L^2}
            +C_N\sum_{b<j-2N_0}2^{-jN-kN}\|\mF_j\mU_{k}f\|_{L^1}\|\mF_b\mP_kh\|_{L^2}\|\mF_j\mP_kf\|_{L^2}.
        \end{aligned}
    \end{equation*}
We take the first term as a typical one and we have
\begin{equation*}
    \begin{aligned}
       & \sum_{j,k=-1}2^{2lk}2^{2n j}j^{2\a}2^{(\gamma+s)k}2^{sj}\|\mS_j\mU_{k}f\|_{L^1}\|\mF_j\mP_kf\|_{L^2}\|\mF_j\mP_kf\|_{L^2}\\
       \ls &\sum_{j,k=-1}\|\mS_j\mU_{k}f\|_{L^1}(2^{2\theta_3(l+\frac{\gamma}{2}+n+s+\zeta)k}2^{-2\theta_3\zeta j}j^\a\|\mF_j\mP_kf\|_{L^2})(2^{2(1-\theta_3)(l+\f\gamma2) k}2^{2(1-\theta_3)(s+n)j}j^{\a}\|\mF_j\mP_kf\|_{L^2})\\
       \ls &~\|f\|_{L^1}\mX_{-\zeta,\a,l+\frac{\gamma}{2}+n+s+\zeta}(f)^{\theta_3}\mY_{n,\a,l}(f)^{1-\theta_3}\leq\varepsilon \mY_{n,\a,l}(f)+C_{\varepsilon,n} \mX_{-\zeta,\a,l+\frac{\gamma}{2}+n+s+\zeta}(f),\quad\forall \varepsilon>0,
    \end{aligned}
\end{equation*}
where $\theta_3=\frac{s}{2(n+s+\zeta)}$ and we use the fact  $\|\mS_j\mU_k f\|_{L^1}\ls \|f\|_{L^1}$ for any $k,j\geq -1$. The rest terms is obvious the lower order terms and we can derive that
\begin{equation}\label{00mC1}
   \sum_{j,k=-1}^\infty2^{2l k}2^{2n j}j^{2\a}\mC_1(f)\leq \varepsilon \mY_{n,\a,l}(f)+C_{\varepsilon,n} \mX_{-\zeta,\a,l+\frac{\gamma}{2}+n+s+\zeta}(f),\quad\forall \varepsilon>0,n\geq -\zeta.
\end{equation}

{\bf\noindent $\bullet$ Estimate of the term containing $\mC_2(f)$.} We can obtain $\mC_2(f)$  from \eqref{bound2ofD}(choose $\Delta_{2,1}^{1,(1)}$) that
  \begin{equation*}
        \begin{aligned}
            &\mC_2(f)\leq 2^{-sk}2^{sj}\|\mS_j\mP_{k}f\|_{L^1}\|\mF_j\mP_kf\|_{L^2}\|\mF_j\mP_kf\|_{L^2}+C_N\sum_{b>j+2N_0}2^{-bN}\|\mF_b\mP_{k}f\|_{L^1}\|\mF_b\mP_kf\|_{L^2}\\
            &\times\|\mF_j\mP_kf\|_{L^2}+C_N\sum_{b<j-2N_0}2^{-jN}\|\mF_j\mP_{k}f\|_{L^1}\|\mF_b\mP_kf\|_{L^2}\|\mF_j\mP_kf\|_{L^2}+2^{-k}2^{(2s-1)j}\\
            &\times \|\mS_j\mP_{k}f\|_{L^2}\|\mF_j\mP_kf\|_{L^2}\|\mF_j\mP_kf\|_{L^2}+C_N\sum_{b>j+2N_0}2^{-k}(2^{-(\gamma+\f52)b}+2^{-\zeta b}b^\a 2^{(2s-1)j})\|\mF_b\mP_{k}f\|_{L^2}\\
            &\times(\|\mF_b\mP_kf\|_{L^2}+2^{-bN-kN}\|f\|_{H^{-N}_{-N}})(\|\mF_j\mP_kf\|_{L^2}+2^{-jN-kN}\|f\|_{H^{-N}_{-N}})+C_N\sum_{b<j-2N_0}2^{-k}(2^{-(\gamma+\f52)j}\\
            &+2^{\f12 b}b^\a 2^{-(\gamma+3)j})\|\mF_j\mP_{k}f\|_{L^2}(\|\mF_b\mP_kf\|_{L^2}+2^{-bN-kN}\|f\|_{H^{-N}_{-N}})(\|\mF_j\mP_kf\|_{L^2}+2^{-jN-kN}\|f\|_{H^{-N}_{-N}}).
        \end{aligned}
    \end{equation*}
    We take one of the fifth term as a typical one. By the same method before, we can derive that
    \begin{equation*}
        \begin{aligned}
            &\sum_{j,k=-1}^\infty\sum_{b>j+2N_0}2^{2l k}2^{2n j}j^{2\a}2^{-k}2^{-\zeta b}b^\a 2^{(2s-1)j}\|\mF_b\mP_kf\|_{L^2}^2\|\mF_j\mP_kf\|_{L^2}\\
            = &\sum_{j,k=-1}^\infty\sum_{b>j+2N_0}(2^{2(l+\f\gamma2) k}2^{2(s+n)b}b^{2\a}\|\mF_b\mP_kf\|^2_{L^2})(2^{l k}2^{-\zeta j}j^\a\|\mF_j\mP_k f\|_{L^2})2^{(2n+\zeta+2s)(j-b)}2^{-j}j^\a b^{-\a}\\
            \ls& ~M^{-\a}\mX_{-\zeta,\a,l}(f)^{\f12}\mY_{n,\a,l}(f)+C_{M,\a}\mX_{-\zeta,\a,l}(f)^{\f32},
        \end{aligned}
    \end{equation*}
    where we also use the fact $\zeta<2s$.

    Other terms can be handled similarly and we can get 
\begin{equation}\label{00mC2}
   \sum_{j,k=-1}^\infty2^{2l k}2^{2n j}j^{2\a}\mC_2(f)\leq M^{-\a}\mX_{-\zeta,\a,l}(f)^{\f12}\mY_{n,\a,l}(f)+C_{M,\a}\mX_{-\zeta,\a,l}(f)^{\f32}\quad\forall M>0.
\end{equation}

{\bf\noindent $\bullet$ Estimate of the term containing $\mC_3(f)$.}  Recall \eqref{bound3ofD}(choose $\Delta_{3,1}^{1,(1)}$) that
\begin{equation*}
        \begin{aligned}
            &\mC_3(f)\leq \sum_{m\geq k-2N_0}2^{(\gamma+2s)m-sk}2^{sj}\|\mS_j\mP_{m}f\|_{L^1}\|\mF_j\mU_mf\|_{L^2}\|\mF_j\mP_kf\|_{L^2}+C_N\sum_{m\geq k-2N_0}\sum_{b>j+2N_0}2^{-bN}\\
            &\times 2^{-mN}\|\mF_b\mP_{m}f\|_{L^1}\|\mF_b\mU_mf\|_{L^2}\|\mF_j\mP_kf\|_{L^2}+C_N\sum_{m\geq k-2N_0}\sum_{b<j-2N_0}2^{-jN-mN}\|\mF_j\mP_{m}f\|_{L^1}\|\mF_b\mU_mf\|_{L^2}\\
            &\times\|\mF_j\mP_kf\|_{L^2}+\sum_{k\leq 2N_0}2^{-k}2^{(2s-1)j}\|\mS_j\mU_{N_0}f\|_{L^2}\|\mF_j\mU_{N_0}f\|_{L^2}\|\mF_j\mP_kf\|_{L^2}\\
            &+C_N\sum_{k\leq 2N_0}\sum_{b>j+2N_0}2^{-k}(2^{-(\gamma+\f52)b}+2^{-\zeta b}b^\a 2^{(2s-1)j})\|\mF_b\mU_{N_0}f\|_{L^2}(\|\mF_b\mU_{N_0}f\|_{L^2}+2^{-bN-kN}\|f\|_{H^{-N}_{-N}})\\
            &\times(\|\mF_j\mP_kf\|_{L^2}+2^{-jN-kN}\|f\|_{H^{-N}_{-N}})+C_N\sum_{k\leq 2N_0}\sum_{b<j-2N_0}2^{-k}(2^{-(\gamma+\f52)j}+2^{\f12b}b^\a 2^{-(\gamma+3)j})\|\mF_j\mU_{N_0}f\|_{L^2}\\
            &\times(\|\mF_b\mU_{N_0}f\|_{L^2}+2^{-bN-kN}\|f\|_{H^{-N}_{-N}})(\|\mF_j\mP_kf\|_{L^2}+2^{-jN-kN}\|f\|_{H^{-N}_{-N}}).
        \end{aligned}
    \end{equation*}
We take the first term as an example, which is different from the previous estimate. We have
\begin{equation*}
    \begin{aligned}
        &\sum_{j,k=-1}^\infty\sum_{m\geq k-2N_0}2^{2l k}2^{2n j}j^{2\a}2^{(\gamma+2s)m-sk}2^{sj}\|\mS_j\mP_{m}f\|_{L^1}\|\mF_j\mU_mf\|_{L^2}\|\mF_j\mP_kf\|_{L^2}\\
        \leq &\sum_{j,k=-1}^\infty\sum_{m\geq k-2N_0}(2^{-\frac{\gamma}{2}(1-\theta_4)m}2^{(l-s) k}2^{(\gamma+2s)m}\|\mS_j\mP_{m}f\|_{L^1})(2^{(s+n) j}j^\a\|\mF_j\mU_mf\|_{L^2})\\\times&(2^{\theta_4lk}2^{-\theta_4\zeta j}2^{(1-\theta_4)(l+\gamma/2)k}2^{(1-\theta_4)(n+s) j}j^\a\|\mF_j\mP_kf\|_{L^2})\\
        \ls&~\|f\|_{L^1_{l+\gamma+s-\frac{\gamma}{2}(1-\theta_4)}}\mX_{-\zeta,\a,l}(f)^{\frac{\theta_4}{2}}\mY_{n,\a,l}(f)^{\frac{1+1-\theta_4}{2}}\\\leq& \varepsilon \|f\|_{L^1_{l+\gamma+s-\frac{\gamma}{2}(1-\theta_4)}}\mY_{n,\a,l}(f)+C_{\varepsilon,n}~\|f\|_{L^1_{l+\gamma+s-\frac{\gamma}{2}(1-\theta_4)}} \mX_{-\zeta,\a,l}(f),
    \end{aligned}
\end{equation*}
where $\theta_4=\frac{s}{n+\zeta+s}$ and we use the fact $-\f\gamma 2(1-\theta_4)m>-\f\gamma 2(1-\theta_4) k$, $2^{(s+n)j}j^\a\|\mF_j\mU_mf\|_{L^2}\ls \mY_{n,\a,l}(f)^{\f12}$  and $\sum_{m\geq k-2N_0}2^{(l+\gamma+s-\frac{\gamma}{2}(1-\theta_4))m}\|\mS_j\mP_mf\|_{L^1}\ls \|f\|_{L^1_{l+\gamma+s-\frac{\gamma}{2}(1-\theta_4)}},\forall j,k\geq -1$.

Other terms can be handled similarly and we can get that
\begin{equation}\label{00mC3}
   \begin{aligned}
       \sum_{j,k=-1}^\infty2^{2l k}2^{2n j}j^{2\a}\mC_3(f)\leq\big(M^{-\a}\mX_{-\zeta,\a,l}(f)^{\f12}+\varepsilon\|f\|_{L^1_{l+\gamma+s-\frac{\gamma}{2}(1-\theta_4)}}\big)\mY_{n,\a,l}(f)\\+C_{M,\a,\varepsilon,n}(\mX_{-\zeta,\a,l}(f)+\mX_{-\zeta,\a,l}(f)^{\f32})+C_{\varepsilon,n}~\|f\|_{L^1_{l+\gamma+s-\frac{\gamma}{2}(1-\theta_4)}} \mX_{-\zeta,\a,l}(f).
   \end{aligned}
\end{equation}

{\bf\noindent $\bullet$ Estimate of the term containing $\mC_4(f)$.} From \eqref{mC4}, since $\zeta<1<2s$ and $n\geq -\zeta$, it is easy to get that
\begin{equation}\label{00mC4}
   \sum_{j,k=-1}^\infty2^{2lk}2^{2n j}j^{2\a}\mC_4(f)\leq C_N M^{-\a}\mX_{-\zeta,\a,l}(f)^{\f12}\mY_{n,\a,l}(f)+C_{M,\a,\varepsilon,l}(\mX_{-\zeta,\a,l}(f)+\mX_{-\zeta,\a,l}(f)^{\f32}).
\end{equation}

Combining \eqref{00mC1}, \eqref{00mC2}, \eqref{00mC3} and \eqref{00mC4}, we can conclude 
\begin{equation}\label{L3jk}
    \begin{aligned}
        &\sum_{j,k=-1}^\infty 2^{2l k}2^{2n j}j^{2\a}L_3(j,k)\leq C_N\big(M^{-\a}\mX_{-\zeta,\a,l}(f)^{\f12}+\varepsilon+\varepsilon\|f\|_{L^1_{l+\gamma+s-\frac{\gamma}{2}(1-\theta_4)}}\big)\mY_{n,\a,l}(f)\\
        &+C_{M,\a,\varepsilon,n}(\mX_{-\zeta,\a,l}(f)+\mX_{-\zeta,\a,l}(f)^{\f32})+C_{\varepsilon,n}(\|f\|_{L^1_{l+\gamma+s-\frac{\gamma}{2}(1-\theta_4)}} \mX_{-\zeta,\a,l}(f)+ \mX_{-\zeta,\a,l+\frac{\gamma}{2}+n+s+\zeta}(f)),
    \end{aligned}
\end{equation}
for any  $M,\a,\varepsilon>0,n\geq-\zeta$ and $\theta_4=\f s{n+\zeta+s}$.

\medskip

Patching together the above estimates \eqref{L1jk}, \eqref{L2jk} and \eqref{L3jk},  we can get from \eqref{energyL123} that (note that $l+\frac{\gamma}{2}+(2s-1)(n+s+\zeta)<l+\frac{\gamma}{2}+n+s+\zeta$)
\begin{equation}\label{energy}
    \begin{aligned}
        &\f12\f{\d}{\d t }\mX_{n,\a,l}(f)(t)+\big(\kappa-\varepsilon-M^{-\a}\mX_{-\zeta,\a,l}(f)(t)^{\f12}-\varepsilon\|f\|_{L^1_{l+\gamma+s-\frac{\gamma}{2}(1-\theta_4)}}\big)\mY_{n,\a,l}(f)(t)\\&\leq C_{M,\a,\varepsilon,n}(1+\mX_{-\zeta,\a,l}(f)(t)^{\f32})+C_{\varepsilon,n} \mX_{-\zeta,\a,l+\frac{\gamma}{2}+n+s+\zeta}(f)(t)+C_{\varepsilon,n}~\|f\|_{L^1_{l+\gamma+s-\frac{\gamma}{2}(1-\theta_4)}} \mX_{-\zeta,\a,l}(f)(t).
    \end{aligned}
\end{equation}
Let $n=-\zeta$, then $\theta_4=1$, using $\frac{\gamma}{2}+s<0$ and $\sup_{t}\|f(t)\|_{L^1_{l+\gamma+s}}<C$, we have
\begin{equation*}
    \f12\f{\d}{\d t }\mX_{-\zeta,\a,l}(f)(t)+\big(\kappa-\varepsilon-\varepsilon C-M^{-\a}\mX_{-\zeta,\a,l}(f)(t)^{\f12}\big)\mY_{-\zeta,\a,l}(f)(t)\leq C_{M,\a,\varepsilon,l}(1+\mX_{-\zeta,\a,l}(f)(t)^{\f32}).
\end{equation*}
Thus for any $f_0$ verifying $0<\mX_{-\zeta,\a,l}(f_0)<C^2_0$, we can choose $\varepsilon<\f\kappa {8(1+C)}$ and $M>(4C_0/\kappa)^{1/\a}$ such that for some $T>0$ depending on $C_0,\kappa,l,\a$ such that
\begin{equation}\label{XYbound}
  \begin{aligned}
    \underset{t\in [0,T]}{\sup}\mX_{-\zeta,\a,l}(f)(t)+\f\kappa 2\int_0^T\mY_{-\zeta,\a,l}(f)(t)\d t< 2C_0^2.
\end{aligned}  
\end{equation}
That is, there exists a solution $f$ of \eqref{Boltzmann} such that $f\in L^{\infty}\big([0,T],H^{-\zeta,\a}_{l}\big)\ccap  L^2\big([0,T],H^{s-\zeta,\a}_{l+\frac{\gamma}{2}}\big)$.

We complete the proof of the existence part of this theorem.
\end{proof}

Using the estimates established above, we now prove the regularization statements in Theorems \ref{maintheorem1} and \ref{maintheorem2}.
\begin{proof}[Proof of the regularization estimates in Theorems \ref{maintheorem1} and \ref{maintheorem2}]
For $n\geq0$, set
\[
\t l:=l+\frac{\gamma(n+\zeta)}{2s}
\quad\mbox{and recall that}\quad 
\theta_4=\frac{s}{n+s+\zeta},
\]
so that the assumption
\[
l\geq-\gamma-\frac{\gamma(n+\zeta)}{2s}
\]
implies $\t l\geq-\gamma$. Following the estimates \eqref{energy}, we have
\begin{equation}
\begin{aligned}
&\frac12\frac{\d}{\d t}\mX_{n,\a,\t l}(f)
+
\Big(
\kappa-\eps
-M^{-\a}\mX_{-\zeta,\a,\t l}(f)^{1/2}
-\eps\|f\|_{L^1_{\t l+\gamma+s-\frac{\gamma}{2}(1-\theta_4)}}
\Big)
\mY_{n,\a,\t l}(f)
\\
\leq&~
C_{M,\a,\eps,n}
\big(1+\mX_{-\zeta,\a,\t l}(f)^{3/2}\big)
+
C_{\eps,n}
\mX_{-\zeta,\a,\t l+\frac{\gamma}{2}+n+s+\zeta}(f)
+
C_{\eps,n}
\|f\|_{L^1_{\t l+\gamma+s-\frac{\gamma}{2}(1-\theta_4)}}
\mX_{-\zeta,\a,\t l}(f).
\label{eq:main2-energy}
\end{aligned}    
\end{equation}
Since
\begin{equation}
    \begin{aligned}
     \t l+\frac{\gamma}{2}+n+s+\zeta-l
=
\frac{\gamma+2s}{2s}(n+s+\zeta)<0,\quad \mbox{and}\quad    \t l+\gamma+s-\frac{\gamma}{2}(1-\theta_4)
<
l+\gamma+s,
    \end{aligned}
\end{equation}
we have
\begin{equation}
    \begin{aligned}
\mX_{-\zeta,\a,\t l+\frac{\gamma}{2}+n+s+\zeta}(f)
\lesssim
\mX_{-\zeta,\a,l}(f),\quad\mbox{and}\quad \|f\|_{L^1_{\t l+\gamma+s-\frac{\gamma}{2}(1-\theta_4)}}
\lesssim
\|f\|_{L^1_{l+\gamma+s}}.
    \end{aligned}
\end{equation}
Also, since $\t l\leq l$, it holds that
\[
\mX_{-\zeta,\a,\t l}(f)
\lesssim
\mX_{-\zeta,\a,l}(f).
\]

Let $T>0$ be defined as before and set
\[
K_T
:=
1+
\sup_{t\in[0,T]}
\mX_{-\zeta,\a,l}(f)(t)
+
\sup_{t\in[0,T]}
\|f(t)\|_{L^1_{l+\gamma+s}}.
\]
By \eqref{L1l} and \eqref{XYbound},  we know that $K_T<\infty$. Choosing $M$ sufficiently large and  $\eps>0$ sufficiently small depending on $K_T$, we obtain from
\eqref{eq:main2-energy} that
\[
\frac{\d}{\d t}\mX_{n,\a,\t l}(f)
+
\f\kappa 2\mY_{n,\a,\t l}(f)
\leq
C_{K_T,n,\a}.
\]
Noticing that
\[
n
=
-\theta_4\zeta+(1-\theta_4)(n+s),\quad \mbox{and}\quad\t l
=
\theta_4 l
+
(1-\theta_4)\left(\t l+\frac{\gamma}{2}\right).
\]
Hence interpolation yields
\[
\|f\|_{H^{n,\a}_{\t l}}
\lesssim
\|f\|_{H^{-\zeta,\a}_l}^{\theta_4}
\|f\|_{H^{n+s,\a}_{\t l+\frac{\gamma}{2}}}^{1-\theta_4}.
\]
In terms of $\mX$ and $\mY$,
\[
\mX_{n,\a,\t l}(f)^{\frac1{1-\theta_4}}
\lesssim
\mX_{-\zeta,\a,l}(f)^{\frac{\theta_4}{1-\theta_4}}
\mY_{n,\a,\t l}(f),
\]
so that
\[
\mY_{n,\a,\t l}(f)
\gtrsim
K_T^{-\frac{s}{n+\zeta}}
\mX_{n,\a,\t l}(f)^{1+\frac{s}{n+\zeta}}.
\]
Therefore
\[
\frac{\d}{\d t}\mX_{n,\a,\t l}(f)
+
C_{\kappa,K_T}
\mX_{n,\a,\t l}(f)^{1+\frac{s}{n+\zeta}}
\leq
C_{T,n,\a},\quad t\in(0,T].
\]
The standard comparison argument gives
\[
\mX_{n,\a,\t l}(f)(t)
\lesssim
t^{-\frac{n+\zeta}{s}},\quad
t\in(0,T]\quad
\mbox{and hence}
\quad
\|f(t)\|_{H^{n,\a}_{\,l+\frac{\gamma(n+\zeta)}{2s}}}
\lesssim
t^{-\frac{n+\zeta}{2s}},
\quad
t\in(0,T].
\]
Since $\zeta=1$ in the moderately soft case,
\[
\|f(t)\|_{H^{n,\a}_{\,l+\frac{\gamma(n+1)}{2s}}}
\lesssim
t^{-\frac{n+1}{2s}},\qquad
t\in(0,T].
\]
This proves the regularization estimates in Theorems \ref{maintheorem1} and \ref{maintheorem2}.
\end{proof}

\subsection{Proof of uniqueness}\label{proveunique}
Having established existence, we now turn to the proof of the uniqueness part of Theorems \ref{maintheorem1} and \ref{maintheorem2}.

\textbf{The case of $\gamma+2s\in (-1,-\f12]$}. We first use the energy method to prove the uniqueness for the case $\gamma+2s\in (-1,-\f12]$. Suppose $f$ and $g$ are two solutions to the equation \eqref{Boltzmann} with the initial data $f_0$ and $g_0$. We recall that the assumptions are $f_0,g_0\in L^1_{2}(\R^3)\ccap  L^1_{l+\gamma+s}(\R^3)\ccap  L\log L\ccap  H^{-\zeta,\a}_l$ with $\a>\f12$ and $l\geq-\gamma+2s$. It leads to 
\begin{equation}\label{integrable}
 f,g\in L^\infty\big([0,T],L^1_{l+\gamma +s}\big)\ccap  L^{\infty}\big([0,T],H^{-\zeta,\a}_{l}\big)\ccap  L^2\big([0,T],H^{s-\zeta,\a}_{l+\frac{\gamma}{2}}\big).
\end{equation}
Indeed, by \eqref{XYbound}, its bound depends on the bounds of $f_0$ and $g_0$.

Let $h=f-g$, then $h(0)=h_0=f_0-g_0$ and it satisfies
\begin{equation*}
\p_t h=Q(f,h)+Q(h,g), 
\end{equation*}
Similar to \eqref{energy1}, multiplying both sides by $\cP_k\F_j^2\cP_k h$ and integrating with respect to $v\in\R^3$ gives
\begin{equation*}
  \begin{aligned}
\frac{1}{2}\frac{\d}{\d t}\|\cF_j\cP_k h\|_{L^2}^2=\&(\p_t\cF_j\cP_k h,\cF_j\cP_k h)_v=(Q(f,h),\cP_k\cF_j^2\cP_kh)_v+(Q(h,g),\cP_k\cF_j^2\cP_kh)_v\\
 =\&(Q(f,\cF_j\cP_kh),\cF_j\cP_kh)_v+\big((Q(f,h),\cP_k\cF_j^2\cP_kh)_v-(Q(f,\cF_j\cP_kh),\cF_j\cP_kh)_v\big)\\
\&+(Q(h,\cP_k^{\f12}g),\cP_k^{\f12}\cF^2_j\cP_kh)_v+\big((Q(h,g),\cP_k\cF_j^2\cP_kh)_v-(Q(h,\cP_k^{\f12}g),\cP_k^{\f12}\cF_j^2\cP_kh)_v\big)\\
=:\& \cD_1(j,k)+\cD_2(j,k)+\cD_3(j,k)+\cD_4(j,k).
\end{aligned}  
\end{equation*}
Here and below, $\cP_k^{\f12}$ denotes the multiplication operator by $\vphi^{\f12}(2^{-k}\cdot)$ for $k\geq0$ (by $\psi^{\f12}$ for $k=-1$), so that $\cP_k=(\cP_k^{\f12})^2$.
Multiply $2^{-2\gamma k}2^{-2\zeta j}j^{2\a}$ with $\a>\f12$ on both sides and sum over $j,k\geq -1$, we obtain 
\begin{equation}\label{6.19}
    \f12\f{\d}{\d t }\sum_{j,k=-1}^\infty 2^{-2\gamma k}2^{-2\zeta j}j^{2\a}\|\F_j\cP_k h\|^2_{L^2}=\sum_{i=1}^4\sum_{j,k=-1}^\infty 2^{-2\gamma k}2^{-2\zeta j}j^{2\a}\cD_i(j,k).
\end{equation}
In other words, the difference $h$ is measured in the low-regularity norm $H^{-\zeta,\a}_{-\gamma}$.

\textbf{Step 1: Estimates of $\cD_1(j,k)$ and $\cD_2(j,k)$.} For $\cD_1(j,k)$, similar to $L_1(j,k)$(see \eqref{L1jk}), we have
\begin{equation}\label{cD1}
    \begin{aligned}
        \sum_{j,k=-1}^\infty 2^{-2\gamma k}2^{-2\zeta j}j^{2\a}\cD_1(j,k)\leq -\big(\kappa-\varepsilon-M^{-\a}\mX_{-\zeta,\a,-\gamma}(f)^{\f12}\big)\mY_{-\zeta,\a,-\gamma}(h)\\+C_{M,\a,\varepsilon}(\mX_{-\zeta,\a,-\gamma}(h)+\mX_{-\zeta,\a,-\gamma}(h)^{\f32}),\quad \forall \varepsilon, M>0.
    \end{aligned}
\end{equation}
For $\cD_2(j,k)$, noticing that
\begin{equation*}
\begin{aligned}
     \cD_2(j,k)\leq |(\cF_jQ(f,\cP_kh)-Q(f,\cF_j\cP_kh),\cF_j\cP_kh)_v|+ |(\cP_kQ(f,h)-Q(f,\cP_kh),\F^2_j\cP_kf)_v|.
\end{aligned}
\end{equation*}
Then we can use Corollary \ref{cor1} and \eqref{PkQf} and the same method of estimates of $L_2(j,k)$ and $L_3(j,k)$ to get that
\begin{equation}\label{cD2}
    \begin{aligned}
        \sum_{j,k=-1}^\infty 2^{-2\gamma k}2^{-2\zeta j}j^{2\a}\cD_2(j,k)\leq (M^{-\a}\mX_{-\zeta,\a,-\gamma}(f)^{\f12}+\varepsilon)\mY_{-\zeta,\a,-\gamma}(h)\\+C_{M,\a,\varepsilon}(1+\mY_{-\zeta,\a,-\gamma}(f))(\mX_{-\zeta,\a,-\gamma}(h)+\mX_{-\zeta,\a,-\gamma}(h)^{\f32}),\quad \forall \varepsilon,M>0.
    \end{aligned}
\end{equation}

\textbf{Step 2: Estimates of $\cD_3(j,k)$.} For $\cD_3(j,k)$, first by \eqref{ubdecom}, we have
\begin{equation*}
\begin{aligned}
     \cD_3(j,k)=&\sum_{\substack{m\ge N_0-1\\|m-k|\leq 2N_0}}( Q_m(\mathcal{U}_{m-N_0} h, \tilde{\mathcal{P}}_m\cP_k^{\f12}g), \tilde{\mathcal{P}}_m\cP_k^{\f12}\F_j^2\cP_kh )_v +
\sum_{\substack{l\ge m+N_0\\|l-k|\leq 2N_0}}( Q_m(\mathcal{P}_{l} h, \tilde{\mathcal{P}}_l\cP_k^{\f12}g), \tilde{\mathcal{P}}_l \cP_k^{\f12}\F_j^2\cP_kh )_v\\
&+\sum_{\substack{|l-m|\le N_0\\k\leq m+2N_0}}( Q_m( \mathcal{P}_{l} h, \mathcal{U}_{m+N_0}\cP_k^{\f12}g), \mathcal{U}_{m+N_0}\cP_k^{\f12}\F_j^2\cP_kh  )_v=:\cD^1_3(j,k)+\cD^2_3(j,k)+\cD^3_3(j,k).
\end{aligned}  
\end{equation*}
Furthermore, we can split each term into $\M^1$ to $\M^4$. For instance, we have 
\begin{equation*}
    \begin{aligned}
        &\cD^1_3(j,k)=\sum_{\substack{m\ge N_0-1\\|m-k|\leq 2N_0}}\Big( \sum_{l\leq p-N_0}\M_{m,p,l}^1(\mathcal{U}_{m-N_0} h, \tilde{\mathcal{P}}_m\cP_k^{\f12}g, \tilde{\mathcal{P}}_m\cP_k^{\f12}\F_j^2\cP_kh )\\
        &+\sum_{l\geq p+N_0}\M_{m,p,l}^2(\mathcal{U}_{m-N_0} h, \tilde{\mathcal{P}}_m\cP_k^{\f12}g, \tilde{\mathcal{P}}_m\cP_k^{\f12}\F_j^2\cP_kh )+\sum_{p\geq -1}\M_{m,p}^3(\mathcal{U}_{m-N_0} h, \tilde{\mathcal{P}}_m\cP_k^{\f12}g, \tilde{\mathcal{P}}_m\cP_k^{\f12}\F_j^2\cP_kh )\\
        &+\sum_{a< p-N_0}\M_{m,p,a}^4(\mathcal{U}_{m-N_0} h, \tilde{\mathcal{P}}_m\cP_k^{\f12}g, \tilde{\mathcal{P}}_m\cP_k^{\f12}\F_j^2\cP_kh )\Big).
    \end{aligned}
\end{equation*}
Then we can use Lemma \ref{M14} to get that 
\begin{equation*}
    \begin{aligned}
        &\sum_{\substack{m\ge N_0-1\\|m-k|\leq 2N_0}}\sum_{l\leq p-N_0}\M_{m,p,l}^1(\mathcal{U}_{m-N_0} h, \tilde{\mathcal{P}}_m\cP_k^{\f12}g, \tilde{\mathcal{P}}_m\cP_k^{\f12}\F_j^2\cP_kh )\\
        \ls&\sum_{\substack{m\ge N_0-1\\|m-k|\leq 2N_0}} \sum_{l\leq p-N_0}C_N2^{-Nm}2^{-Np}\|\F_p \mathcal{U}_{m-N_0} h\|_{L^2}\|\F_l \tilde{\mathcal{P}}_m\cP_k^{\f12}g\|_{L^2}\|\t\F_p\tilde{\mathcal{P}}_m\cP_k\F_j^2\cP_k^{\f12}h\|_{L^2};\\
        &\sum_{\substack{m\ge N_0-1\\|m-k|\leq 2N_0}}\sum_{a\leq p-N_0}\M_{m,p,a}^4(\mathcal{U}_{m-N_0} h, \tilde{\mathcal{P}}_m\cP_k^{\f12}g, \tilde{\mathcal{P}}_m\cP_k^{\f12}\F_j^2\cP_kh )\\
        \ls&\sum_{\substack{m\ge N_0-1\\|m-k|\leq 2N_0}} \sum_{a\leq p-N_0}C_N2^{-Nm}2^{-Np}\|\F_p \mathcal{U}_{m-N_0} h\|_{L^2}\|\F_p \tilde{\mathcal{P}}_m\cP_k^{\f12}g\|_{L^2}\|\t\F_a\tilde{\mathcal{P}}_m\cP_k^{\f12}\F_j^2\cP_kh\|_{L^2}.
    \end{aligned}
\end{equation*}
And using Lemma \ref{M23} to get that
\begin{equation*}
    \begin{aligned}
        &\sum_{\substack{m\ge N_0-1\\|m-k|\leq 2N_0}}\sum_{p\geq -1}\M_{m,p}^3(\mathcal{U}_{m-N_0} h, \tilde{\mathcal{P}}_m\cP_k^{\f12}g, \tilde{\mathcal{P}}_m\cP_k^{\f12}\F_j^2\cP_kh )\\
        \ls&\sum_{\substack{m\ge N_0-1\\|m-k|\leq 2N_0}} \sum_{p\geq-1}2^{(\gamma+\f32)m}\|\F_p \mathcal{U}_{m-N_0} h\|_{L^2}\|\tF_p \tilde{\mathcal{P}}_m\cP_k^{\f12}g\|_{L^2}\|\t\F_p\tilde{\mathcal{P}}_m\cP_k^{\f12}\F_j^2\cP_kh\|_{L^2},\quad m\geq0;\\
        &\sum_{\substack{m\ge N_0-1\\|m-k|\leq 2N_0}}\sum_{p\leq l-N_0}\M_{m,p,l}^2(\mathcal{U}_{m-N_0} h, \tilde{\mathcal{P}}_m\cP_k^{\f12}g, \tilde{\mathcal{P}}_m\cP_k^{\f12}\F_j^2\cP_kh )\\
        \ls&\sum_{\substack{m\ge N_0-1\\|m-k|\leq 2N_0}} \sum_{p\leq l-N_0}2^{(\gamma+\f32)m}2^{2s(l-p)}\|\F_p \mathcal{U}_{m-N_0} h\|_{L^2}\|\tF_l \tilde{\mathcal{P}}_m\cP_k^{\f12}g\|_{L^2}\|\t\F_l\tilde{\mathcal{P}}_m\cP_k^{\f12}\F_j^2\cP_kh\|_{L^2},\quad m\geq0.
    \end{aligned}
\end{equation*}

 To estimate \(\sum_{j,k=-1}^\infty 2^{-2\gamma k}2^{-2\zeta j}j^{2\a}\cD^1_3(j,k)\), we take \(\M^2_{m,p,l}\) as an example. Noticing that
\begin{equation*}
    \begin{aligned}
      &\sum_{j,k=-1}^\infty 2^{-2\gamma k}2^{-2\zeta j}j^{2\a}\sum_{\substack{m\ge N_0-1\\|m-k|\leq 2N_0}} \sum_{p\leq l-N_0}2^{(\gamma+\f32)m}2^{2s(l-p)}\|\F_p \mathcal{U}_{m-N_0} h\|_{L^2}\|\tF_l \tilde{\mathcal{P}}_m\cP_k^{\f12}g\|_{L^2}\|\t\F_l\tilde{\mathcal{P}}_m\cP_k^{\f12}\F_j^2\cP_kh\|_{L^2}\\
      &\ls \sum_{j,k=-1}^\infty\sum_{\substack{ p\leq l-N_0}}2^{-2\gamma k}2^{-2\zeta j}j^{2\a}2^{(\gamma+\f32)k}2^{2s(l-p)}\|\F_p \mU_{k} h\|_{L^2}\|\tF_l \mP_kg\|_{L^2}\|\t\F_l\mP_k\F_j^2\cP_kh\|_{L^2},
    \end{aligned}
\end{equation*}
where we use the fact that $|m-k|\leq 2N_0$.  We now consider two cases: $|l-j|\leq N_0$ and $|l-j|>N_0$. For the latter case, applying Lemma \ref{PFcommutator}, we can bounded it by 
\begin{equation*}
\begin{aligned}
       &C_N\sum_{j,k=-1}^\infty\sum_{p\leq l-N_0}2^{-jN-kN-lN}\|\F_p \mU_{k} h\|_{L^2}\|\tF_l \mP_kg\|_{L^2}\|\t\F_l\mP_k\F_j^2\cP_kh\|_{L^2}\\
       &\leq C_N \|g\|_{H^{-N}_{-N}}\|h\|^2_{H^{-N}_{-N}}\leq C_N \mY_{-\zeta,\a,-\gamma+2s}(g)^{\f12}\mX_{-\zeta,\a,-\gamma}(h),~\mbox{for example, choose}~N=10.
\end{aligned}
\end{equation*}
For the former case, since $\gamma+2s<-\f12$ implies $\zeta<\f32<2s$, we have the bound for any $\eps>0$
\begin{equation*}
\begin{aligned}
    &\sum_{j,k=-1}^\infty\sum_{p\leq l-N_0}(2^{-\zeta p}\|\F_p \mU_{k} h\|_{L^2})(2^{(-\f\gamma 2+2s)k}2^{(s-\zeta)j}j^\a\|\mF_j \mP_kg\|_{L^2})(2^{-\f\gamma 2 k}2^{(s-\zeta)j}j^\a\|\mF_j\mP_k\F_j^2\cP_kh\|_{L^2})2^{(\zeta-2s)p}\\
    &\ls \mX_{-\zeta,\a,-\gamma}(h)^{\f12}\mY_{-\zeta,\a,-\gamma+2s}(g)^{\f12}\mY_{-\zeta,\a,-\gamma}(h)^{\f12}\leq \varepsilon \mY_{-\zeta,\a,-\gamma}(h)+C_\varepsilon \mY_{-\zeta,\a,-\gamma+2s}(g)\mX_{-\zeta,\a,-\gamma}(h). 
\end{aligned}
\end{equation*}
Other terms can be handled by the same manner and we can obtain that
\begin{equation}\label{D31jk}
    \begin{aligned}
        \sum_{j,k=-1}^\infty 2^{-2\gamma k}2^{-2\zeta j}j^{2\a}\cD^1_3(j,k)
        \ls \varepsilon \mY_{-\zeta,\a,-\gamma}(h)+C_{\varepsilon} (1+\mY_{-\zeta,\a,-\gamma+2s}(g))\mX_{-\zeta,\a,-\gamma}(h),\quad\forall \varepsilon>0. 
    \end{aligned}
\end{equation}

Similarly, we can handle the term $\cD^2_3(j,k)$ and $\cD^3_3(j,k)$. Indeed, we have the split 
\begin{equation*}
    \begin{aligned}
        \cD^2_3(j,k)=&\sum_{l\ge m+N_0,|l-k|\leq 2N_0}\Big( \sum_{a\leq p-N_0}\M_{m,p,a}^1(\cP_{l} h, \tilde{\mathcal{P}}_l\cP_k^{\f12}g, \tilde{\mathcal{P}}_l\cP_k^{\f12}\F_j^2\cP_kh )\\
        &+\sum_{a\geq p+N_0}\M_{m,p,a}^2(\cP_{l} h, \tilde{\mathcal{P}}_l\cP_k^{\f12}g, \tilde{\mathcal{P}}_l\cP_k^{\f12}\F_j^2\cP_kh )+\sum_{p\geq -1}\M_{m,p}^3(\cP_{l} h, \tilde{\mathcal{P}}_l\cP_k^{\f12}g, \tilde{\mathcal{P}}_l\cP_k^{\f12}\F_j^2\cP_kh )\\
        &+\sum_{a< p-N_0}\M_{m,p,a}^4(\cP_{l} h, \tilde{\mathcal{P}}_l\cP_k^{\f12}g, \tilde{\mathcal{P}}_l\cP_k^{\f12}\F_j^2\cP_kh )\Big),
    \end{aligned}
\end{equation*}
and
\begin{equation*}
    \begin{aligned}
        \cD^3_3(j,k)=&\sum_{|l-m|\le N_0,k\leq m+2N_0}\Big( \sum_{a\leq p-N_0}\M_{m,p,a}^1(\cP_{l} h, \cU_{m+N_0}\cP_k^{\f12}g, \cU_{m+N_0}\cP_k^{\f12}\F_j^2\cP_kh )\\
        &+\sum_{a\geq p+N_0}\M_{m,p,a}^2(\cP_{l} h, \cU_{m+N_0}\cP_k^{\f12}g, \cU_{m+N_0}\cP_k^{\f12}\F_j^2\cP_kh )\\&+\sum_{p\geq -1}\M_{m,p}^3(\cP_{l} h, \cU_{m+N_0}\cP_k^{\f12}g, \cU_{m+N_0}\cP_k^{\f12}\F_j^2\cP_kh )\\
        &+\sum_{a< p-N_0}\M_{m,p,a}^4(\cP_{l} h, \cU_{m+N_0}\cP_k^{\f12}g, \cU_{m+N_0}\cP_k^{\f12}\F_j^2\cP_kh )\Big).
    \end{aligned}
\end{equation*}

By the same method as above, we can get the same estimate for $m\geq 0$. For $m=-1$, we use Lemmas \ref{M14} and \ref{M23} and follow the same steps in the proof of existence to obtain that $\forall \varepsilon>0$,
\begin{equation}\label{D323jk}
    \begin{aligned}
        \sum_{j,k=-1}^\infty 2^{-2\gamma k}2^{-2\zeta j}j^{2\a}(\cD^2_3(j,k)+\cD^3_3(j,k))
        \ls\varepsilon \mY_{-\zeta,\a,-\gamma}(h)+C_{\varepsilon} \big(1+\mY_{-\zeta,\a,-\gamma+2s}(g)\big)\big(1+\mX_{-\zeta,\a,-\gamma}(h)^{\f32}\big). 
    \end{aligned}
\end{equation}
\textbf{Step 3: Estimates of $\cD_4(j,k)$.} Here we use Proposition \ref{QPk}, then we have
\begin{equation*}
    \begin{aligned}
        \cD_4(j,k)=&(\cP_k^{\f12}Q(h,g)-Q(h,\cP_k^{\f12}g),\cP_k^{\f12}\cF_j^2\cP_kh)_v\ls\sum_{i=1}^4\mC_i(h,g,h).
    \end{aligned}
\end{equation*}
Note that the term $\cP_k^{\f12}\cF_j^2\cP_kh$ on the third position does not influence the estimate. Indeed, we do the phase decomposition first in \eqref{commutator2} and we do not need $\Delta_{i,2},i=1,2,3$ here. 

We take $\mC_3(h,g,h)$ as a typical one and recall from \eqref{bound3ofD}(choose $\Delta_{3,1}^{1,(2)}$). For the first term, i.e., 
\beno
\sum_{m\geq k-2N_0}\sum_{a\leq j+5N_0}2^{(\gamma+\f32)m}2^{\tau(m-k)}2^{(2s-\tau)(j-a)}\|\mF_a\mP_mh\|_{L^2}\|\mF _j\mU_mg\|_{L^2}\|\mF _j\mP_kh\|_{L^2},
\eeno
with $\tau\in[0,1]$. Since $\gamma<-\f12-2s<-\f32$, we have
\begin{equation}\label{100a}
    \begin{aligned}
        &\sum_{m\geq k-2N_0}\sum_{a\leq j+5N_0}2^{-2\gamma k}2^{-2\zeta j}j^{2\a}2^{(\gamma+\f32)m}2^{\tau(m-k)}2^{(2s-\tau)(j-a)}\|\mF_a\mP_mh\|_{L^2}\|\mF _j\mU_mg\|_{L^2}\|\mF _j\mP_kh\|_{L^2}\\
        \ls&\sum_{m\geq k-2N_0}\sum_{a\leq j+5N_0}2^{(\gamma+\tau)(m-k)}2^{(\gamma+\f32)m}2^{(\zeta-2s+\tau)a}a^{-\a}(2^{-\gamma m}2^{-\zeta a}a^{\a}\|\mF_a\mP_mh\|_{L^2})\\
        &\times(2^{-\gamma k}2^{-\zeta j}j^\a\|\mF _j\mP_kh\|_{L^2})(2^{(2s-\tau-\zeta)j}j^\a\|\mF _j\mU_kg\|_{L^2}),
    \end{aligned}
\end{equation}
where we take $\tau=\min\{1,2s-\zeta\}>0$. Since $\gamma+\tau,\gamma+\f32<0$, $\zeta-2s+\tau\leq0,\a>\f12$ and it is easy to verify that $-\gamma-\frac{\gamma(n+\zeta)}{2s}\leq -\gamma-\frac{\gamma(2s-\tau)}{2s}<-\gamma+2s\leq l$ when $n=2s-\tau-\zeta$, which means that we can use Theorem \ref{maintheorem1} to obtain
\begin{equation}\label{D_4a}
    \begin{aligned}
\sum_{j,k=-1}^\infty\eqref{100a}
        \ls~\mX_{-\zeta,\a,-\gamma}(h)\mY_{(2s-\tau-\zeta),\a,-\gamma}(g)^{\f12}
        \ls t^{-\frac{2s-\tau}{2s}}\mX_{-\zeta,\a,-\gamma}(h).
    \end{aligned}
\end{equation}
The second term is similar to the first one and we can get the same bound. As for the rest terms $\mC_1(h,g,h)$, $\mC_2(h,g,h)$ and $\mC_4(h,g,h)$, we can use the same method in the estimate of $\cD_3(j,k)$ so that we can get
\begin{equation}\label{D4jk}
    \begin{aligned}
        &\sum_{j,k=-1}^\infty 2^{-2\gamma k}2^{-2\zeta j}j^{2\a}\cD_4(j,k)
        \ls\varepsilon \mY_{-\zeta,\a,-\gamma}(h)\\
        &+C_{\varepsilon} (1+\mY_{-\zeta,\a,-\gamma+2s}(g))(1+\mX_{-\zeta,\a,-\gamma}(h)^{\f32})+t^{-\frac{2s-\tau}{2s}}\mX_{-\zeta,\a,-\gamma}(h),\quad\forall \varepsilon>0. 
    \end{aligned}
\end{equation}
Plug \eqref{cD1}, \eqref{cD2}, \eqref{D31jk}, \eqref{D323jk} and \eqref{D4jk} into \eqref{6.19}, we obtain 
\begin{equation*}
    \begin{aligned}
        &\f{\d}{\d t }\mX_{-\zeta,\a,-\gamma}(h)(t)+(\kappa-\varepsilon-M^{-\a}\mX_{-\zeta,\a,-\gamma}(f)^{\f12})\mY_{-\zeta,\a,-\gamma}(h)(t)\\&\leq C_{M,\a,\varepsilon}(1+\mY_{-\zeta,\a,-\gamma}(f)+\mY_{-\zeta,\a,-\gamma+2s}(g)+t^{-\frac{2s-\tau}{2s}})(\mX_{-\zeta,\a,-\gamma}(h)+\mX_{-\zeta,\a,-\gamma}(h)^{\f32}).
    \end{aligned}
\end{equation*}
Since $\mX_{-\zeta,\a,-\gamma}(h)\in L^\infty[0,T]$ and $\mY_{-\zeta,\a,-\gamma}(f),\mY_{-\zeta,\a,-\gamma+2s}(g)\in L^1[0,T]$(see \eqref{integrable}), choose sufficiently small $\varepsilon$  and large $M>0$, we obtain the stability result:
\begin{equation}\label{energyuniqueness}
    \begin{aligned}
        \mX_{-\zeta,\a,-\gamma}(h)(t)&\ls C_{\a,f_0,g_0}e^{\int_0^tC(\omega) d\omega} \mX_{-\zeta,\a,-\gamma}(h_0),
    \end{aligned}
\end{equation}
where $C(\omega)=1+\mY_{-\zeta,\a,-\gamma}(f)(\omega)+\mY_{-\zeta,\a,-\gamma+2s}(g)(\omega)+\omega^{-\frac{2s-\tau}{2s}}\in L^1[0,T]$.

The uniqueness for the case $\gamma+2s\in (-1,-\f12]$ is an immediate corollary.

\bigskip

\textbf{The case of $\gamma+2s\in (-\f12,0)$}.  Noticing that $\gamma+2s>-\f12$ implies 
$\f 3{3+\gamma}<\f 6{5-4s}$. We claim that 
\begin{equation}\label{claim}
   f\in L^1_{\rm{loc}}([0,T],L^p(\R^3)),\quad \mbox{for}\quad p\in  \big(\f 3{3+\gamma},\f 6{5-4s}\big),
\end{equation}
if $f_0\in H^{-1}_{l}$ with $l\geq -2\gamma$. Thanks to Theorem \ref{conditionaluniqueness}, the solution is unique.

Now we only need to prove the claim, which can be derive from the integrability of $\|f(t)\|_{H_{-\gamma}^{\frac{3}{2}-\frac{3}{p}}}$ because of the Sobolev embedding. Here we use Theorem \ref{maintheorem2}, taking $n=\f32-\f3p$. Note that when $l\geq-2\gamma$, we have $-\gamma-\frac{\gamma(n+1)}{2s}<-2\gamma=l$, which satisfy the conditions of Theorem \ref{maintheorem2}, so that
\begin{equation*}
    \begin{aligned}
    \|f(t)\|_{H_{-\gamma}^{\frac{3}{2}-\frac{3}{p}}}\ls t^{-\frac{\frac{5}{2}-\frac{3}{p}}{2s}}+1,\quad \forall t\in(0,T].
\end{aligned}
\end{equation*}
Note that $-\frac{\frac{5}{2}-\frac{3}{p}}{2s}=\frac{3}{2sp}-\frac{5}{4s}>-1$, thus $\|f(t)\|_{L^p}\ls \|f(t)\|_{H^{\f32-\f3 p}_{-\gamma}}$ is locally integrable with respect to $t\in[0,T]$. So we complete the proof of uniqueness for the case $\gamma+2s\in (-\f12,0)$.

\section{Proof of global existence}\label{sec:global}

We will prove the global statements in Theorems~\ref{maintheorem1} and~\ref{maintheorem2}. The argument controls the critical negative Sobolev norm by a polynomial moment and the Fisher information $I(f):=\int_{\R^3}\f{|\nabla f|^2}{f}\,\d v$.

Let $\zeta\in(\f12,1]$ and $\beta>0$. In dimension three we have $L^{p_*}_l(\R^3)\hookrightarrow H^{-\zeta}_l(\R^3)$ with $\f1{p_*}=\f12+\f\zeta3$, i.e.\ $p_*=\f{6}{3+2\zeta}$, and for $p_\delta>p_*$ sufficiently close to $p_*$ the logarithmic weight is absorbed into the small gain of regularity, so that $L^{p_\delta}_l(\R^3)\hookrightarrow H^{-\zeta,\beta}_l(\R^3)$. Writing $\f1{p_\delta}=\theta_\delta+\f{1-\theta_\delta}{3}$, i.e.\ $\theta_\delta=\f12\big(\f3{p_\delta}-1\big)$, and $q_\delta=\f{l}{\theta_\delta}$, the weighted H\"older interpolation between $L^1_{q_\delta}$ and $L^3$ gives
\begin{equation}\label{eq:global-interpolation}
\|f\|_{H^{-\zeta,\beta}_l}\ls\|f\|_{L^{p_\delta}_l}\ls\|f\|_{L^1_{q_\delta}}^{\theta_\delta}\|f\|_{L^3}^{1-\theta_\delta}.
\end{equation}
At $\delta=0$($p_0=p_*$) one has $\theta_0=\f{1+2\zeta}{4}$ and $q_0=\f{4l}{1+2\zeta}$, so choosing $\delta>0$ small enough yields $q_\delta\leq q_0+1$. On the other hand, Sobolev's inequality applied to $\sqrt f$ yields $\|f\|_{L^3}=\|\sqrt f\|_{L^6}^2\ls\|\nabla\sqrt f\|_{L^2}^2=\f14 I(f)$. Consequently,
\begin{equation}\label{eq:global-fisher}
\|f(t)\|_{H^{-\zeta,\beta}_l}\ls\|f(t)\|_{L^1_{q_0+1}}^{\theta_\delta}\,I(f(t))^{1-\theta_\delta}.
\end{equation}

\medskip\noindent\textbf{The very soft case.} In Theorem~\ref{maintheorem1}, we have $\zeta=\f32+\gamma+2s=2+\f\gamma2$, hence $1+2\zeta=5+\gamma$, $q_0=\f{4l}{5+\gamma}$, and \eqref{eq:global-fisher} reads
\begin{equation}\label{eq:global-very-soft}
\|f(t)\|_{H^{-\zeta,\a}_l}\ls\|f(t)\|_{L^1_{\f{4l}{5+\gamma}+1}}^{\theta_\delta}\,I(f(t))^{1-\theta_\delta}.
\end{equation}
Let $l\geq-\f{9\gamma}{1-\gamma}$. For any fixed $\tau>0$, the regularization estimate \eqref{regularisation-estimate} with $n=2$ gives $\|f(\tau)\|_{H^2_{\,l+\f{\gamma(8+\gamma)}{1-\gamma}}}<\infty$, and
\[
l+\f{\gamma(8+\gamma)}{1-\gamma}\geq-\f{9\gamma}{1-\gamma}+\f{\gamma(8+\gamma)}{1-\gamma}=-\gamma\geq2=\f32+\f12,
\]
since $\gamma\in(-3,-2]$. Hence $f(\tau)\in H^2_{\f32+\eps}$ with $\eps=\f12$;
since moreover $f(\tau)\geq0$, this yields $I(f(\tau))<\infty$. By the monotonicity of the Fisher information \cite{imbert2026monotonicity}, $I(f(t))\leq I(f(\tau))$ for $t\geq\tau$. Together with the propagation of the moment $L^1_{4l/(5+\gamma)+1}$, the bound \eqref{eq:global-very-soft} gives $\sup_{\tau\leq t\leq T}\|f(t)\|_{H^{-\zeta,\a}_l}<\infty$ for every $0<\tau<T<\infty$, while the local estimate controls the interval $[0,\tau]$. Moreover $\f{4l}{5+\gamma}+1\geq l+\f{3\gamma+1}{4}=l+\gamma+s$, so the moment required in \eqref{L1} is also propagated. Hence the hypotheses of the local theory remain bounded on every finite time interval, and the standard continuation argument yields $T_{\max}=\infty$. Finally, $l\geq-\f{9\gamma}{1-\gamma}$ implies $l\geq-\f32\gamma+\f12=-\gamma+2s$ and $l\geq-\gamma$, so for $\a>\f12$ the uniqueness part of Theorem~\ref{maintheorem1} applies. This proves the global well-posedness in Theorem~\ref{maintheorem1}.

\medskip\noindent\textbf{The moderately soft case.} In Theorem~\ref{maintheorem2} we have $\zeta=1$, so $p_*=\f65$, $\theta_0=\f34$, $q_0=\f{4l}3$, and \eqref{eq:global-fisher} becomes
\begin{equation}\label{eq:global-moderate}
\|f(t)\|_{H^{-1,\a}_l}\ls\|f(t)\|_{L^1_{\f{4l}3+1}}^{\theta_\delta}\,I(f(t))^{1-\theta_\delta}.
\end{equation}
Assume $l\geq-\gamma-\f{3\gamma}{2s}$ and $l+\f{3\gamma}{2s}>\f32$; since $-\gamma<2$ here, both conditions hold as soon as $l\geq 2-\f{3\gamma}{2s}=2-\f{6\gamma}{1-\gamma}$. Then, for every $\tau>0$, the regularization estimate \eqref{regularisation-estimate} with $n=2$ gives $\|f(\tau)\|_{H^2_{\,l+3\gamma/(2s)}}<\infty$, hence $f(\tau)\in H^2_{\f32+\eps}$ for some $\eps>0$; since $f(\tau)\geq0$, it follows that $I(f(\tau))<\infty$, and the monotonicity of the Fisher information \cite{imbert2026monotonicity} gives $I(f(t))\leq I(f(\tau))$ for $t\geq\tau$. Using \eqref{eq:global-moderate} and the propagation of $L^1_{4l/3+1}$, we obtain $\sup_{\tau\leq t\leq T}\|f(t)\|_{H^{-1,\a}_l}<\infty$ for every $0<\tau<T<\infty$, while the local estimate controls the interval $[0,\tau]$. Since $\f{4l}3+1\geq l+\gamma+s$, the moment required in \eqref{L1moderate} is also controlled, so the local solution can be continued beyond every finite time. If in addition $\a>\f12$ and $l\geq-2\gamma$, the uniqueness part of Theorem~\ref{maintheorem2} applies on every finite interval. This proves the global well-posedness in the moderately soft potential case.

\section{Appendix}\label{appendix}
In the Appendix, we collect some useful lemmas on pseudo-differential operators and basic commutators.
\begin{lemma}
    Let $s,r\in \R$, and $a(v),b(v)\in C^{\infty}$ satisfy for any $\alpha\in\Z_{+}^{3}$,
    \begin{equation}
        |D_v^{\alpha}a(v)|\leq C_{1,\alpha}\<v\>^{r-|\alpha|},|D_{\xi}^{\alpha}b(\xi)|\leq C_{2,\alpha}\<\xi\>^{s-|\alpha|}
    \end{equation}
    for constant $C_{1,\alpha},C_{2,\alpha}$. Then there exists a constant $C$ depending only on $s,r$ and finite numbers of $C_{1,\alpha},C_{2,\alpha}$ such that for any $f\in \mathscr{S}(\R^3)$,
    \begin{equation}
        \|a(v)b(D)f\|_{L^2}\leq C\|\<D\>^s\<v\>^rf\|_{L^2},\|b(D)a(v)f\|_{L^2}\leq C\|\<v\>^r\<D\>^sf\|_{L^2}.
    \end{equation}
    As a direct consequence, we get that $\|\<D\>^m\<v\>^lf\|_{L^2}\sim\|\<v\>^l\<D\>^mf\|_{L^2}\sim\|f\|_{H^m_l}$.
\end{lemma}
\begin{definition}\label{symbol}
    A smooth function $a(v,\xi)$ is said to be a symbol of type $S_{1,0}^m$ if $a(v,\xi)$ verifies that for any multi-indices $\alpha$ and $\beta$,
    \begin{equation*}
        |(\partial_{\xi}^{\alpha}\partial_{v}^{\beta}a)(v,\xi)|\leq C_{\alpha,\beta}\<\xi\>^{m-|\alpha|},
    \end{equation*}
    where $C_{\alpha,\beta}$ is a constant depending only on $\alpha$ and $\beta$.
\end{definition}

\begin{lemma}\label{pseudo operator}
    Let $l,s,r\in \R,M(\xi)\in S^{r}_{1,0}$ and $\Phi(v)\in S^{l}_{1,0}$. Then there exists a constant $C$ such that $\|[M(D_v)\Phi]f\|_{H^s}\leq C\|f\|_{H^{r+s-1}_{l-1}}$. Moreover, for any $N\in\N$,
    \begin{equation}\label{7.3}
        M(D_v)\Phi=\Phi M(D_v)+\sum_{1\leq|\alpha|\leq N}\frac{1}{\alpha!}\Phi_{\alpha}M^{\alpha}(D_v)+r_N(v,D_v),
    \end{equation}
    where $\Phi_{\alpha}(v)=\partial_v^{\alpha}\Phi,M^{\alpha}(\xi)=\partial^{\alpha}_{\xi}M(\xi)$ and $\<v\>^{N-l}r_N(v,\xi)\in S^{r-N}_{1,0}$. Moreover, for any $\beta,\beta'\in\Z_+^3$, we have
    \begin{equation}
        |\partial^{\beta}_{v}\partial^{\beta'}_{\xi}r_N(v,\xi)|\leq C_{\beta,\beta'}\<\xi\>^{r-N-|\beta|}\<v\>^{l-N-|\beta'|},\|r_{2N+1}(v,D_v)\<D\>^N\<v\>^Nf\|_{L^2}\leq C\|f\|_{L^2}
    \end{equation}
    Furthermore, using \eqref{7.3} repeatedly, we can obtain that
    \begin{equation}\label{7.5}
        M(D_v)\Phi=\Phi M(D_v)+\sum_{1\leq |\alpha|\leq N}C_{\alpha}M^{\alpha}(D_v)\Phi_{\alpha}+C_Nr_N(v,D_v).
    \end{equation}
\end{lemma}
\begin{lemma}{(Bernstein inequality)}\label{Bern}
    There exists a constant $C$ independent of $j$ and function $f$ such that
        
         $\bullet$ For any $s\in \R$ and $j\geq -1$,
        \begin{equation}
            C^{-1}2^{js}\|\cF_j f\|_{L^2(\R^3)}\leq \|\cF_j f\|_{H^s(\R^3)}\leq C 2^{js}\|\cF_j f\|_{L^2(\R^3)}.
        \end{equation}
        
         $\bullet$ For integers $k\geq 0,j\geq-1$ and $p,q\in[1,\infty],q\geq p$,
        \begin{equation}
        \begin{aligned}
             &\underset{|\alpha|=k}{\sup}\|\partial^{\alpha}\cF_jf\|_{L^q(\R^3)}\ls 2^{jk}2^{3j(\frac{1}{p}-\frac{1}{q})}\|\cF_jf\|_{L^p(\R^3)},\\
             &~~2^{jk}\|\cF_jf\|_{L^p(\R^3)}\ls\underset{|\alpha|=k}{\sup}\|\partial^{\alpha}\cF_jf\|_{L^p(\R^3)}\ls 2^{jk}\|\cF_jf\|_{L^p(\R^3)}.
        \end{aligned}
        \end{equation}
\end{lemma}

\begin{lemma}\label{mathcalR}
Let \(\gamma\in(-3,0),s\in(0,1)\), \(\Phi^{\gamma}_{-1}\) be defined by \eqref{DefPhi}. For any $|\alpha|\leq 2$, we have 
\begin{equation}\label{parPhi-1}
    |\partial^\alpha_\xi \widehat{\Phi^\gamma_{-1}}(\xi)|\ls \<\xi\>^{-(\gamma+3+|\alpha|)}.
\end{equation}
Let \(j\ge -1\) and for smooth functions \(g,h\) and \(f\), define
\[
\mathcal{R}(g,h,f):=\int_{\mathbb{R}^6}\Phi_{-1}^\gamma(|v-v_*|)\,g(v_*)\,h(v)\,f(v)\,\d v _*\,\d v .
\]
Then it holds that
    \begin{equation}\label{Rghf2}
        \begin{aligned}
            |\mathcal{R}(g,\F_jh,\F_jf)|
        \ls~& (M^{-\a}+2^{-2Ms})\|g\|_{H^{-(\gamma+2s+\f32),\a}_{-\gamma}}\|\F_jh\|_{H^s_{\gamma/2}}\|\F_jf\|_{H^s_{\gamma/2}}\\
            &+C_{M,\a}\|g\|_{H^{-(\gamma+2s+\f32),\a}_{-\gamma}}\|\F_jh\|_{L^2_{\gamma/2}}\|\F_jf\|_{L^2_{\gamma/2}}\quad \forall M,\a>0.
        \end{aligned}
    \end{equation}
\end{lemma}
\begin{proof}
The proof of \eqref{parPhi-1} can be found in \cite{CHJ} Lemma 4.16.

We now give the proof of \eqref{Rghf2}. By Fourier transform, we have 
    \begin{equation*}
        \begin{aligned}
            \mathcal{R}(g,\F_jh,\F_jf)=\int_{\R^6}\widehat{\Phi^\gamma_{-1}}(\eta)\hat{g}(\eta)\widehat{\F_jh}(\xi-\eta)\widehat{\F_jf}(\xi)d\eta d\xi.
        \end{aligned}
    \end{equation*}
Since $|\xi|\sim |\xi-\eta|\sim 2^{j}$, then $|\eta|\ls 2^j$, i.e., 
    \begin{equation*}
        \begin{aligned}
            \mathcal{R}(g,\F_jh,\F_jf)=\sum_{l\leq j}\mathcal{R}(\mF_lg,\F_j h,\F_j f).
        \end{aligned}
    \end{equation*}
For fixed $l$, note that $|\widehat{\Phi^\gamma_{-1}}(\eta)|\ls \<\eta\>^{-(\gamma+3)}$, by Cauchy-Schwarz inequality, we have
   \begin{equation*}
        \begin{aligned}
            &\sum_{l\leq j}\mathcal{R}(\mF_lg,\F_j h,\F_j f)\ls \sum_{l\leq j}2^{-(\gamma+3)l}\Big(\int_{\R^6}|\widehat{\mF_l g}(\eta)|^2|\widehat{\F_j h}(\xi-\eta)|^2d\eta d\xi\Big)^{\f12}\Big(\int_{|\eta|\sim 2^l}|\widehat{\F_j f}(\xi)|^2d\eta d\xi\Big)^{\f12}\\
            &\ls\sum_{l\leq j}2^{-(\gamma+\f32)l}\|\mF_l g\|_{L^2}\|\F_jh\|_{L^2}\|\F_jf\|_{L^2}\ls \sum_{l\leq j}l^{-\a}(2^{-(\gamma+2s+\f32)l}l^\a\|\mF_l g\|_{L^2})2^{2s(l-j)} 2^{sj}\|\F_jh\|_{L^2}2^{sj}\|\F_jf\|_{L^2}\\
            &\ls (M^{-\a}+2^{-2Ms})\|g\|_{H^{-(\gamma+2s+\f32),\a}}\|\F_jh\|_{H^s}\|\F_jf\|_{H^s}+C_{M,\a}\|g\|_{H^{-(\gamma+2s+\f32),\a}}\|\F_jh\|_{L^2}\|\F_jf\|_{L^2},
        \end{aligned}
    \end{equation*}
where we consider three cases: \(l\ge M\); \(l<M\) and \(j>2M\); and \(j\le 2M\).

Finally, noting that \(|v-v_*|\lesssim 1\) implies \(\langle v\rangle\sim\langle v_*\rangle\), the weight indices
\(\omega_1,\omega_2,\omega_3\) on \(g,h,f\) need only satisfy \(\omega_1+\omega_2+\omega_3=0\). Here we take \(\omega_1=-\gamma,\omega_2=\omega_3=\gamma/2\). We complete the proof of this lemma.
 
\end{proof}

\bibliographystyle{alpha}
\bibliography{ref}

\newcommand{\etalchar}[1]{$^{#1}$}
\begin{thebibliography}{AMU{\etalchar{+}}12}

\bibitem[AMU{\etalchar{+}}12]{AMU1}
Radjesvarane Alexandre, Yoshinori Morimoto, Seiji Ukai, Chao-Jiang Xu, and Tong Yang.
\newblock The {B}oltzmann equation without angular cutoff in the whole space: {I}, global existence for soft potential.
\newblock {\em Journal of Functional Analysis}, 262(3):915--1010, 2012.

\bibitem[CCL09]{Lu2009}
Eric~A. Carlen, Maria~C. Carvalho, and Xuguang Lu.
\newblock On strong convergence to equilibrium for the {B}oltzmann equation with soft potentials.
\newblock {\em Journal of Statistical Physics}, 135(4):681--736, 2009.

\bibitem[CH11]{CH1}
Yemin Chen and Ling-Bing He.
\newblock Smoothing estimates for {B}oltzmann equation with full-range interactions: spatially homogeneous case.
\newblock {\em Archive for Rational Mechanics and Analysis}, 201(2):501--548, 2011.

\bibitem[CH12]{CH2}
Yemin Chen and Ling-Bing He.
\newblock Smoothing estimates for {B}oltzmann equation with full-range interactions: spatially inhomogeneous case.
\newblock {\em Archive for Rational Mechanics and Analysis}, 203(2):343--377, 2012.

\bibitem[CHJ24]{CHJ}
Chuqi Cao, Ling-Bing He, and Jie Ji.
\newblock Propagation of moments and sharp convergence rate for inhomogeneous noncutoff {B}oltzmann equation with soft potentials.
\newblock {\em SIAM Journal on Mathematical Analysis}, 56(1):1321--1426, 2024.

\bibitem[CS23]{jamil2023entropy}
Jamil Chaker and Luis Silvestre.
\newblock Entropy dissipation estimates for the {B}oltzmann equation without cut-off.
\newblock {\em Kinetic and Related Models}, 16(5):748--763, 2023.

\bibitem[DM05]{DesvillettesMouhot}
Laurent Desvillettes and Cl{\'e}ment Mouhot.
\newblock About {$L^p$} estimates for the spatially homogeneous {B}oltzmann equation.
\newblock {\em Annales de l'Institut Henri Poincar{\'e}, Analyse Non Lin{\'e}aire}, 22(2):127--142, 2005.

\bibitem[DM09]{desvillettes2009stability}
Laurent Desvillettes and Cl{\'e}ment Mouhot.
\newblock Stability and uniqueness for the spatially homogeneous {B}oltzmann equation with long-range interactions.
\newblock {\em Archive for Rational Mechanics and Analysis}, 193(2):227--253, 2009.

\bibitem[DW04]{DW}
Laurent Desvillettes and Bernt Wennberg.
\newblock Smoothness of the solution of the spatially homogeneous {B}oltzmann equation without cutoff.
\newblock {\em Communications in Partial Differential Equations}, 29(1-2):133--155, 2004.

\bibitem[FG08]{fournier2008uniqueness}
Nicolas Fournier and H{\'e}l{\`e}ne Gu{\'e}rin.
\newblock On the uniqueness for the spatially homogeneous {B}oltzmann equation with a strong angular singularity.
\newblock {\em Journal of Statistical Physics}, 131(4):749--781, 2008.

\bibitem[Fou10]{FournierLandau}
Nicolas Fournier.
\newblock Uniqueness of bounded solutions for the homogeneous {L}andau equation with a {C}oulomb potential.
\newblock {\em Communications in Mathematical Physics}, 299(3):765--782, 2010.

\bibitem[GGL25]{golding2025global}
William Golding, Maria Gualdani, and Am{\'e}lie Loher.
\newblock Global smooth solutions to the {L}andau--{C}oulomb equation in {$L^{3/2}$}.
\newblock {\em Archive for Rational Mechanics and Analysis}, 249(3):34, 2025.

\bibitem[GL24]{golding2024local}
William Golding and Am{\'e}lie Loher.
\newblock Local-in-time strong solutions of the homogeneous {L}andau--{C}oulomb equation with {$L^p$} initial datum.
\newblock {\em La Matematica}, 3(1):337--369, 2024.

\bibitem[GS25a]{GSUN}
Maria Gualdani and Weiran Sun.
\newblock Uniqueness for the homogeneous {L}andau--{C}oulomb equation in {$L^{3/2}$}.
\newblock {\em arXiv preprint arXiv:2512.20899}, 2025.

\bibitem[GS25b]{guillen2025landau}
Nestor Guillen and Luis Silvestre.
\newblock The {L}andau equation does not blow up.
\newblock {\em Acta Mathematica}, 234(2):315--375, 2025.

\bibitem[He18]{he2018sharp}
Ling-Bing He.
\newblock Sharp bounds for {B}oltzmann and {L}andau collision operators.
\newblock {\em Annales Scientifiques de l'{\'E}cole Normale Sup{\'e}rieure (4)}, 51(5):1253--1341, 2018.

\bibitem[HJL24]{HeJiLuo2024}
Ling-Bing He, Jie Ji, and Yue Luo.
\newblock Existence, uniqueness and smoothing estimates for spatially homogeneous {L}andau--{C}oulomb equation in {$H^{-1/2}$} space with polynomial tail.
\newblock {\em arXiv preprint arXiv:2412.07287}, 2024.

\bibitem[HJZ20]{HJZ}
Ling-Bing He, Jin-Cheng Jiang, and Yu-Long Zhou.
\newblock On the cutoff approximation for the {B}oltzmann equation with long-range interaction.
\newblock {\em Journal of Statistical Physics}, 181(5):1817--1905, 2020.

\bibitem[HST20]{HST2}
Christopher Henderson, Stanley Snelson, and Andrei Tarfulea.
\newblock Local well-posedness of the {B}oltzmann equation with polynomially decaying initial data.
\newblock {\em Kinetic and Related Models}, 13(4):837--867, 2020.

\bibitem[HST25]{henderson2025classical}
Christopher Henderson, Stanley Snelson, and Andrei Tarfulea.
\newblock Classical solutions of the {B}oltzmann equation with irregular initial data.
\newblock {\em Annales Scientifiques de l'{\'E}cole Normale Sup{\'e}rieure}, 2025.

\bibitem[HW22]{HW}
Christopher Henderson and Weinan Wang.
\newblock Local well-posedness for the {B}oltzmann equation with very soft potential and polynomially decaying initial data.
\newblock {\em SIAM Journal on Mathematical Analysis}, 54(3):2845--2875, 2022.

\bibitem[IS20a]{IS4}
Cyril Imbert and Luis Silvestre.
\newblock Regularity for the {B}oltzmann equation conditional to macroscopic bounds.
\newblock {\em EMS Surveys in Mathematical Sciences}, 7(1):117--172, 2020.

\bibitem[IS20b]{IS2}
Cyril Imbert and Luis Silvestre.
\newblock The weak {H}arnack inequality for the {B}oltzmann equation without cut-off.
\newblock {\em Journal of the European Mathematical Society}, 22(2):507--592, 2020.

\bibitem[IS22]{IS1}
Cyril Imbert and Luis Silvestre.
\newblock Global regularity estimates for the {B}oltzmann equation without cut-off.
\newblock {\em Journal of the American Mathematical Society}, 35(3):625--703, 2022.

\bibitem[ISV26]{imbert2026monotonicity}
Cyril Imbert, Luis Silvestre, and C{\'e}dric Villani.
\newblock On the monotonicity of the {F}isher information for the {B}oltzmann equation.
\newblock {\em Inventiones Mathematicae}, 243(1):127--179, 2026.

\bibitem[Vil98]{villani1998new}
C{\'e}dric Villani.
\newblock On a new class of weak solutions to the spatially homogeneous {B}oltzmann and {L}andau equations.
\newblock {\em Archive for Rational Mechanics and Analysis}, 143(3):273--307, 1998.

\end{thebibliography}

\end{document}